\documentclass[11pt,a4paper]{article}

\usepackage[a4paper,top=1in,bottom=1in,left=1in,right=1in]{geometry} % For sufficient margins on A4
\usepackage{setspace}
\RequirePackage{amsthm,amsmath,amsfonts,amssymb}
\allowdisplaybreaks[4]
\RequirePackage[numbers,sort&compress]{natbib}
\RequirePackage[colorlinks,citecolor=blue,urlcolor=blue]{hyperref}
\RequirePackage{graphicx}
\usepackage{booktabs}
\usepackage{float}
\usepackage[caption=false]{subfig} % 避免与 caption 包/类的 caption 设置打架
\usepackage{authblk}

\theoremstyle{plain}

\newtheorem{theorem}{Theorem}[section]
\newtheorem{lemma}[theorem]{Lemma}

\newtheorem{proposition}[theorem]{Proposition}
\theoremstyle{remark}
\newtheorem{definition}[theorem]{Definition}
\newtheorem{assumption}[theorem]{Assumption}
\newtheorem{remark}[theorem]{Remark}

\newcommand{\bbA}{{\bf A}}
\newcommand{\bSigma}{\boldsymbol{\Sigma}}
\newcommand{\bLambda}{\boldsymbol{\Lambda}}
\newcommand{\bGamma}{\boldsymbol{\Gamma}}
\newcommand{\bbnu}{\boldsymbol{\nu}}

\newcommand{\bbB}{{\bf B}}
\newcommand{\bbC}{{\bf C}}

\newcommand{\bbD}{{\bf D}}

\newcommand{\bbe}{{\bf e}}
\newcommand{\bbE}{{\bf
		E}}

\newcommand{\bbF}{{\bf F}}
\newcommand{\bbP}{{\bf P}}
\newcommand{\bbg}{{\bf g}}

\newcommand{\bbK}{{\bf K}}
\newcommand{\bbH}{{\bf H}}

\newcommand{\bbw}{{\bf w}}
\newcommand{\bbI}{{\bf I}}

\newcommand{\bbM}{{\bf M}}

\newcommand{\bbQ}{{Q}}

\newcommand{\bbR}{{\bf R}}

\newcommand{\bbS}{{\bf S}}

\newcommand{\bbT}{{\bf T}}

\newcommand{\bbU}{{\bf U}}
\newcommand{\bbu}{{\boldsymbol u}}
\newcommand{\bbV}{{\bf V}}

\newcommand{\bbW}{{\bf W}}

\newcommand{\bbX}{{\bf X}}

\newcommand{\bbx}{{\bf x}}

\newcommand{\bbY}{{\bf Y}}
\newcommand{\bby}{{\bf y}}

\newcommand{\bbL}{{\bf L}}
\newcommand{\bxi} {\boldsymbol \xi}

	\newcommand{\tr}{{\rm tr}}

	\newcommand{\bqa}{\begin{eqnarray}}
		\newcommand{\eqa}{\end{eqnarray}}
	\newcommand{\bqn}{\begin{eqnarray*}}
		\newcommand{\eqn}{\end{eqnarray*}}
	\hypersetup{
		colorlinks=true,
		linkcolor=blue,
		anchorcolor=blue,
		citecolor=blue,
		urlcolor=blue,
		CJKbookmarks=true
	}
	
\begin{document}

		\title{Asymptotic Theory for Sample Canonical Direction Alignment in High-Dimensional Spiked CCA}

		\author[1]{Zhangni Pu\thanks{\texttt{puzn98@nwu.edu.cn}}}
		\author[2]{Xiaozhuo Zhang\thanks{\texttt{zhangxz722@nenu.edu.cn}}}
		\author[2]{Jiang Hu\thanks{\texttt{huj156@nenu.edu.cn}}}
		
		\affil[1]{School of Mathematics, Northwest University, Xi'an, China}
		\affil[2]{KLASMOE, Key Laboratory of Big Data Analysis of Jilin Province and School of Mathematics \& Statistics, Northeast Normal University, No. 5268 People's Street, Changchun 130024, China}
		\date{}
		%\footnotetext[1]{*Corresponding authors.} 
		\maketitle
			
			\begin{abstract}
				This paper studies the asymptotic behavior of sample canonical directions in a finite-rank spiked high-dimensional canonical correlation analysis model under a Gaussian population assumption. Under the asymptotic regime in which the dimensions of the two data blocks grow proportionally with the sample size, sample canonical directions are generally not consistent estimators of their population counterparts, even when the corresponding sample canonical correlations separate from the bulk spectrum. To quantify directional recovery, we investigate the squared alignment between a sample canonical direction and its associated population direction. For a supercritical simple population spike, we establish an explicit nonzero deterministic limit for the squared alignment. In contrast, for population non-spikes, we show that the alignment with the corresponding population signal directions vanishes asymptotically. Thus, the same phase transition that governs eigenvalue separation also marks a transition in directional recovery. For a supercritical simple spike, we further prove a central limit theorem for its fluctuations around the deterministic limit, with an explicit asymptotic variance expressed through deterministic limits of resolvent trace functionals. To make the theoretical quantities computable from data, we further construct plug-in estimators for both the deterministic alignment limit and the asymptotic variance by inverting the deterministic outlier eigenvalue map, and prove their consistency. Numerical simulations and a real-data illustration support the theoretical results and demonstrate how the proposed estimators assess the recovery quality of sample canonical directions.
			\end{abstract}
		
		\section{Introduction}
		
		Canonical correlation analysis (CCA), introduced by Hotelling \cite{hotelling1935most, harold1936relations}, is a classical method for studying the linear association between two random vectors. Given two centered random vectors $\bbx$ in $\mathbb{R}^p$ and $\bby$ in $\mathbb{R}^q$, CCA seeks pairs of linear combinations whose correlations are successively maximized. In the population formulation, the squared canonical correlations and the associated canonical directions are determined by the eigenstructure of the population canonical correlation matrix
		$$\bSigma_{xx}^{-1}\bSigma_{xy}\bSigma_{yy}^{-1}\bSigma_{yx}.$$ In the sample formulation, they are estimated by the eigenvalues and eigenvectors of the sample canonical correlation matrix $$\bbS_{xx}^{-1}\bbS_{xy}\bbS_{yy}^{-1}\bbS_{yx}.$$ This eigenvalue-eigenvector representation makes CCA particularly suitable for analysis through random matrix theory.
		
		Classical multivariate analysis studies CCA in the regime where $p$ and $q$ are fixed while $n\to\infty$; see, for example, \cite{muirhead1982aspects,anderson2003introduction}. In many modern applications, however, the dimensions of the two data blocks are comparable to the sample size. In such high-dimensional regimes, sample canonical correlations and their associated directions exhibit behavior that is fundamentally different from classical low-dimensional asymptotics. In particular, sample canonical directions need not consistently estimate their population counterparts, even when the associated signal eigenvalues are separated from the noise bulk.
		
		Let the columns of $(\bbX^\top,\bbY^\top)^\top$ be $n$ independent observations of the centered random vector $(\bbx^\top,\bby^\top)^\top$. Denote the population covariance matrix by
		\begin{align*}
			\bSigma=\begin{pmatrix} \bSigma_{xx} & \bSigma_{xy}\\
				\bSigma_{yx} &\bSigma_{yy}
			\end{pmatrix}.
		\end{align*}
		The population squared canonical correlations $r_i$ and the corresponding population canonical directions $\bbu_i$ are determined by $$\bSigma_{xx}^{-1}\bSigma_{xy}\bSigma_{yy}^{-1}\bSigma_{yx}\bbu_i = r_i \bbu_i.$$ The sample analogues are the eigenvalues $l_i$ and eigenvectors $\bbnu_i$ of $$\bbS_{xx}^{-1}\bbS_{xy}\bbS_{yy}^{-1}\bbS_{yx}.$$
		The main objective of this paper is to quantify, in a finite-rank spiked high-dimensional CCA model, how well a sample canonical direction $\bbnu_i$ aligns with its population counterpart $\bbu_i$.
		
		Throughout the theoretical analysis, we work in the Gaussian spiked CCA model
		\[
		\bbX
		=
		\bLambda\bbY+\bGamma\bbW,
		\qquad
		\bGamma
		=
		\left(
		\bbI_p-\bLambda\bLambda^\top
		\right)^{1/2},
		\]
		following the spiked CCA framework in \cite{bao2019canonical}. The entries of $\bbY$ and $\bbW$ are mutually independent and identically distributed as $\mathcal N(0,n^{-1})$. Consequently,
		\[
		\bLambda\bLambda^\top+\bGamma\bGamma^\top=\bbI_p.
		\]
		The nonzero eigenvalues of $\bLambda\bLambda^\top$ are the nonzero population squared canonical correlations $r_1,\ldots,r_k$. Since canonical directions are defined only up to sign, we measure directional recovery through the squared alignment $$\langle \bbu_i, \bbnu_i \rangle^2.$$
		This quantity is the squared cosine of the angle between the population and sample canonical directions and is invariant under sign changes of the eigenvectors.
		
		We work under the following dimensional assumption.
		
		\begin{assumption}[On the dimensions]\label{assumption1}
			We assume that $p:=p(n)$, $q:=q(n)$, and, as $n\to \infty$,
			$$p/n\to c_1\in(0,1),~~~~q/n\to c_2\in(0,1),~~~~\mbox{ s.t.

			} c_1+c_2\in(0,1).$$
			Without loss of generality, we always work with the additional assumption $p>q$, thus $c_1>c_2$.
			We further assume
			\[
			\frac pn-c_1=O(n^{-1}),\qquad
			\frac qn-c_2=O(n^{-1}).
			\]
			This condition is satisfied by $p=\lfloor c_1n\rfloor$ and
			$q=\lfloor c_2n\rfloor$.
		\end{assumption}
		
		\begin{assumption}[On the population spiked model]\label{assumption2}
			We assume that $\text{rank}(\bSigma_{xy})= k$ for some fixed positive integer $k$. Let $r_i$ denote the $i$-th largest non-zero eigenvalue of the population matrix $\bSigma_{xx}^{-1}\bSigma_{xy}\bSigma_{yy}^{-1}\bSigma_{yx}$.
			We assume that, for some fixed integer $k_0$ satisfying
			$1\leq k_0\leq k$, the first $k_0$ eigenvalues are ``spiked'':
			$$1>r_1\geq \dots \geq r_i\geq \dots \geq r_{k_0}>r_c:=\sqrt{\frac{c_1c_2}{(1-c_1)(1-c_2)}}.$$
			If $k_0<k$, the remaining nonzero eigenvalues are non-spiked and satisfy
			\[
			r_c>r_{k_0+1}\geq\cdots\geq r_k>0.
			\]
			Throughout the paper, we exclude the critical case $r_i=r_c$.
		\end{assumption}
		\begin{remark}
			The threshold
			\[
			r_c
			=
			\sqrt{
				\frac{c_1c_2}
				{(1-c_1)(1-c_2)}
			}
			\]
			separates the supercritical and subcritical regimes for the sample canonical correlation matrix \cite{bao2019canonical,ma2023sample}. For $1\leq i\leq k_0$, the population spikes satisfy $r_i>r_c$ and generate sample eigenvalues that separate from the upper edge $d_+$ of the limiting bulk spectrum. When such a population spike is simple, its associated sample eigenvector can therefore be matched with the corresponding population direction. In contrast, population eigenvalues below $r_c$ do not generate sample outliers, and the corresponding sample eigenvalues stick to the upper bulk edge. Theorem \ref{th1} shows that this spectral phase transition is accompanied by a corresponding phase transition in eigenvector alignment.
		\end{remark}
	
		\subsection{Previous related theoretical results}\label{eigenvalue}
		We first recall the spectral theory for sample canonical correlation matrices in the high-dimensional null and spiked settings. The empirical spectral distribution (ESD)
		$$F_n(x):=\frac{1}{q} \sum_{i=1}^q I(l_i \le x)$$
		converges weakly to a deterministic probability distribution $F(x)$ \cite{wachter1980limiting}, whose density is
		\begin{align}
			f(x)=\frac{1}{2\pi c_2}\frac{\sqrt{(d_+-x)(x-d_-)}}{x(1-x)}I(d_-\leq x\leq d_+),
		\end{align}
		where $d_\pm=(\sqrt{c_1(1-c_2)}\pm\sqrt{c_2(1-c_1)})^2$. Central limit theorems for linear spectral statistics and related global spectral quantities have been further developed in \cite{yang2012convergence,yang2015independence,zhang2023limiting}.
		
		In the finite-rank setting, the sample canonical correlation matrix $\bbS_{xx}^{-1}\bbS_{xy}\bbS_{yy}^{-1}\bbS_{yx}$ can be viewed as a finite-rank perturbation of the null model. The perturbation does not change the limiting spectral distribution of the bulk, but sufficiently strong population canonical correlations generate sample outliers.
		Bao et al. \cite{bao2019canonical} analyzed the limiting behavior and fluctuations of its largest eigenvalues in this spiked setting. Let
		\begin{align} \label{eq:3}
			\gamma_i:=r_i(1-c_1+c_1r_i^{-1})(1-c_2+c_2r_i^{-1}).
		\end{align}
		\begin{theorem}[Theorem 1.7 of \cite{bao2019canonical}]\label{th:1.4}
			Under Assumptions \ref{assumption1} and \ref{assumption2}, the sample squared canonical correlations exhibit the following convergence:
			\begin{enumerate}
				\item (Outliers) For $1\leq i\leq k_0$, as $n\to\infty$, we have $l_i-\gamma_i\xrightarrow{a.s.}0$.
				\item (Sticking eigenvalues) For each fixed $i\geq k_0+1$, as $n\to\infty$, we have $l_i-d_+\xrightarrow{a.s.}0$.
			\end{enumerate}
		\end{theorem}
		
		Extensions of the spiked CCA eigenvalue theory to non-Gaussian populations and related Fisher-type models have been studied in \cite{ma2023sample,yang2022limiting,bai2022limiting,hou2023spiked}. These eigenvalue results provide the spectral foundation for signal detection in high-dimensional CCA. However, eigenvalue separation alone does not quantify directional recovery. A separated sample canonical correlation indicates that a signal is detectable from the spectrum, but it does not describe how accurately the associated sample canonical direction estimates the corresponding population direction. This distinction motivates the eigenvector analysis developed in this paper.
		
		The study of eigenvectors in random matrix theory (RMT) is generally more challenging than that of eigenvalues. Global-level results have investigated eigenvector empirical spectral distributions (VESD) and related quantities \cite{bai2007asymptotics, xia2013convergence, xi2020convergence, li2023eigenvector}. Local-level studies, which are close to the goal of this paper, characterize the "distance" or "alignment" between an individual sample eigenvector and a fixed population direction. Such phenomena have been studied for spiked sample covariance matrices \cite{2007paul,bao2022statistical,pu2026asymptotic}, deformed Wigner matrices \cite{fan2022asymptotic,lei2016goodness,bao2026eigenvector}, and information-plus-noise models \cite{2020baosingular,liu2023asymptotic}. These works show that the squared inner product between a sample eigenvector and its population counterpart is a natural measure of directional recovery. 
		
		For high-dimensional CCA, the eigenvector problem remains less developed. Bykhovskaya and Gorin \cite{bykhovskaya2023high} study the high-dimensional behavior of canonical variables, or canonical variates, and obtain first-order limits for angles between sample and population canonical variables such as $\bbX^\top\bbnu_i$, $\bbX^\top\bbu_i$. These angles are defined in the sample space. They are different from the alignment between the canonical direction vectors themselves, namely the eigenvectors and their population counterparts in the original variable spaces. The latter quantity is sensitive to the covariance structure and coordinate normalization. Consequently, consistency or non-consistency of the canonical variates does not directly determine the recovery accuracy of the canonical direction vectors themselves. In particular, an asymptotic theory for the squared alignment $\langle \bbu_i, \bbnu_i \rangle^2$ in a finite-rank spiked CCA model does not seem to be available in the existing literature.
		
		\subsection{Contributions and Overview of Results}\label{contribution}
		
		This paper develops first- and second-order asymptotic theory for sample canonical directions associated with finite-rank population signals in high-dimensional CCA under Gaussian assumptions. Our main object of interest is the squared alignment
		\[
		\langle\bbu_i,\bbnu_i\rangle^2,
		\]
		which measures the amount of population-level directional information retained by the sample canonical direction.
		
		Our first contribution is a complete first-order characterization of directional recovery across the phase transition. For a simple population spike $r_i>r_c$, we prove that
		\[
		\langle\bbu_i,\bbnu_i\rangle^2
		\xrightarrow{p}
		\frac{1}{1+d(r_i)},
		\]
		where
		\[
		d(r_i)
		=
		\frac{
			c_1(1-r_i)\bigl((1-c_2)r_i+c_2\bigr)
		}{
			(1-c_1)(1-c_2)r_i^2-c_1c_2
		}.
		\]
		The limiting alignment lies strictly between zero and one. Hence the corresponding sample canonical direction retains nontrivial information about the population direction, but is not classically consistent.
		
		In contrast, in the subcritical regime $r_i<r_c$, we prove that the squared alignment between the corresponding population direction and the associated sample canonical direction vanishes asymptotically,
		\[
		\langle \bbu_i,\bbnu_i\rangle^2\xrightarrow{P}0.
		\]
		Thus, the phase transition governing eigenvalue separation is accompanied by a directional-recovery phase transition: above the threshold, the sample direction has a nonvanishing asymptotic overlap with the corresponding population direction, whereas below the threshold its projection onto the population signal subspace vanishes.
		
		Our main second-order contribution is a central limit theorem for the fluctuation of the squared alignment in the simple supercritical regime. Specifically, we show that
		\[
		\sqrt{n}
		\left(
		\langle\bbu_i,\bbnu_i\rangle^2
		-
		\frac{1}{1+d(r_i)}
		\right)
		\]
		converges in distribution to a centered Gaussian random variable with variance
		\[
		\frac{\sigma^2(r_i)}
		{\bigl(1+d(r_i)\bigr)^4}.
		\]
		Here $\sigma^2(r_i)$ is given explicitly in terms of deterministic limits of resolvent trace functionals. To the best of our knowledge, this is the first central limit theorem for the squared Euclidean alignment between a sample canonical direction and its population counterpart in a finite-rank high-dimensional spiked CCA model. This result upgrades the directional-recovery theory from a deterministic approximation to a distributional approximation.
		
		Our third contribution is to make the preceding asymptotic quantities estimable from data. Since both the limiting mean and the asymptotic variance depend on the unknown population spike $r_i$, we estimate $r_i$ by inverting the deterministic outlier map,
		\[
		\hat r_i=\gamma^{-1}(l_i).
		\]
		Substituting $\hat r_i$ into the deterministic expressions yields the plug-in estimators
		\[
		\hat\mu_i
		=
		\frac{1}{1+d(\hat r_i)},
		\qquad
		\hat\tau_i^2
		=
		\frac{\sigma^2(\hat r_i)}
		{\bigl(1+d(\hat r_i)\bigr)^4}.
		\]
		We prove that these estimators consistently estimate the limiting mean and asymptotic variance of the squared alignment, respectively. Consequently, the proposed theory provides a sample-based measure of both the strength and the uncertainty of directional recovery.
		
		On the technical side, our proof proceeds through an eigenvector reduction based on the Schur complement. The original high-dimensional eigenvector equation is reduced to a finite-dimensional random system involving resolvents of the noise block of the CCA matrix. For a simple spike, this reduction isolates the coordinate corresponding to the population spike and yields the key normalization relation
		\[
		\langle\bbu_i,\bbnu_i\rangle^2
		[\bbM_i]_{ii}
		=
		1+O_p(n^{-1}).
		\]
		We then obtain deterministic equivalents and second-order fluctuation estimates for the resulting resolvent trace functionals. In the subcritical regime, the same block representation is combined with the singular behavior of the edge resolvent to show that the squared alignment between the sample canonical direction and the corresponding population canonical direction vanishes asymptotically. The Gaussian assumption allows us to use Gaussian integration by parts and the Gaussian Poincar\'e inequality to control the relevant trace fluctuations and covariance structure.
		
		Finally, we provide numerical illustrations of both the first- and second-order results. The simulation study examines the deterministic alignment limits, the subcritical directional phase transition, and the Gaussian fluctuation in rank-one and multi-spike settings. We also present a real-data illustration using the limestone grassland community data originally reported by Gittins \cite{Gittins1985} and subsequently analyzed in Bao et al.\ \cite{bao2019canonical}. In this example, the leading sample squared canonical correlation lies above the asymptotic bulk edge, and the plug-in estimate $\hat\mu_1$ provides, under the working spiked CCA model, a quantitative assessment of the asymptotic recovery quality of the leading sample canonical direction.
		
		\subsection{Organization and notation} 
		The remainder of the paper is organized as follows. Section \ref{mainresults} presents the first-order directional phase transition and the central limit theorem for the squared alignment. Section \ref{sec:estimation} develops plug-in estimators for the limiting mean and asymptotic variance and presents a real-data illustration. Section \ref{simulation} reports Monte Carlo simulations for both the supercritical regime. Section \ref{zmcl} proves the main theorems, while the subsequent sections establish the technical resolvent and fluctuation estimates required in the proofs.
		
		Throughout the paper, $||\cdot||$ denotes the Euclidean norm for vectors and the operator norm for matrices. For a matrix $\bbA$, $\tr\bbA$ denotes its trace, and $\bbA^\top$ denotes its transpose. We write $\xrightarrow{p}$, $\xrightarrow{a.s.}$, and $\xrightarrow{D}$ for convergence in probability, almost sure convergence, and convergence in distribution, respectively.
		
		\section{Main Results}\label{mainresults}
		
		In this section, we present the first- and second-order asymptotic behavior of the squared alignment between sample and population canonical directions. We first establish the directional phase transition and then derive the central limit theorem in the simple supercritical regime.
		
		 Under the model $\bbX=\bLambda\bbY+\bGamma\bbW$, the finite-rank population correlation structure is encoded in $\bLambda$. Since the population and sample canonical directions are defined only up to sign and are normalized to have unit norm, we study the squared inner product between them. The squared inner product $\langle \bbu_i,\bbnu_i\rangle^2$ is the squared cosine of the angle between the population and sample canonical directions. The first-order approximation of the squared cosine of this angle is given below. 
		 
		 \begin{theorem}[First-order limit]\label{th1} 
			
			Under Assumptions \ref{assumption1}--\ref{assumption2}, the following
			results hold.
			\begin{enumerate}
				\item[(i)] 
				For each fixed $1\leq i\leq k_0$ such that $r_i$ is a simple
				population eigenvalue,
				\begin{align}
					\langle \bbu_i,\bbnu_i\rangle^2
					\xrightarrow{p}
					\frac{1}{1+d(r_i)} .
				\end{align}
				
				\item[(ii)]
				If $k_0<k$, then, for each fixed
				$k_0+1\leq i\leq k$,
				\begin{align}
					\langle \bbu_i,\bbnu_i\rangle^2
					\xrightarrow{p}0 .
				\end{align}
			\end{enumerate}
			
			Here, for $r_i>r_c$,
			\[
			d(r_i)=\frac{c_1(1-r_i)((1-c_2)r_i+c_2)}
			{(1-c_1)(1-c_2)r_i^2-c_1c_2}.\]
			\iffalse
			Here, in part (i),
			\[
			d(r_i)=D_1(r_i)+D_2(r_i),
			\]
			where
			\begin{align*}
				D_1(r_i)
				=&(1-r_i)\left(
				\frac{2(\eta r_i-\omega)(1-r_i)-r_i}{r_i}F_1(r_i)\right.\\
				&\left.
				+\frac{r_i(\eta r_i-\omega)(1-r_i)
					-(\eta r_i-\omega)^2(1-r_i)^2}{r_i^2}F_2(r_i)
				\right),
				\\
				D_2(r_i)
				=&
				\frac{r_i(\eta r_i-\omega)^2(1-r_i)^2}{c_2r_i^2}
				\left(
				\left(
				\frac{c_1r_i+\omega(1-r_i)}
				{r_i-(\eta r_i-\omega)(1-r_i)}
				-\frac{\omega}{\eta r_i-\omega}
				\right)F_1(r_i)\right.\\
				&\left.
				+\left(
				c_1+c_2+\frac{2\omega(1-r_i)}{r_i}
				\right)F_2(r_i)
				\right),
				\\
				F_1(r_i)
				=&
				\frac{-c_1r_i((1-c_2)r_i+c_2)}
				{r_i(\eta r_i-\omega)
					-(\eta r_i-\omega)^2(1-r_i)},
				\\
				F_2(r_i)
				=&
				\frac{
					r_i^2
					\left[
					\omega(1-c_2)r_i\big((1-c_1)r_i+c_1\big)
					+c_1(\eta r_i-\omega)(\eta r_i^2-\omega)
					\right]
				}
				{
					\big((1-c_1)r_i+c_1\big)^2
					\big((1-c_2)r_i+c_2\big)
					(\eta r_i-\omega)^2
					(\eta r_i^2-\omega)
				},
			\end{align*}
			with $
			\eta=(1-c_1)(1-c_2), \omega=c_1c_2.$
			\fi
			\end{theorem} 
			
			\begin{remark}
				For $r>r_c$, the limiting squared alignment
				\[
				\mu(r)
				=
				\frac{1}{1+d(r)}
				\]
				lies strictly between zero and one. Moreover,
				\[
				\lim_{r\downarrow r_c}
				\frac{1}{1+d(r)}
				=
				0,
				\qquad
				\lim_{r\uparrow1}
				\frac{1}{1+d(r)}
				=
				1.
				\]
				Thus, a population spike just above the phase-transition threshold retains asymptotically little directional information, whereas the sample canonical direction approaches perfect recovery as the population squared canonical correlation approaches one.
			\end{remark}
		
			Figure~\ref{fig:directional_phase_transition} illustrates this transition
			for $(c_1,c_2)=(0.5,0.3)$, for which $r_c\approx0.655$.
			
			\begin{figure}[htbp]
				\centering
				\includegraphics[width=0.70\textwidth]
				{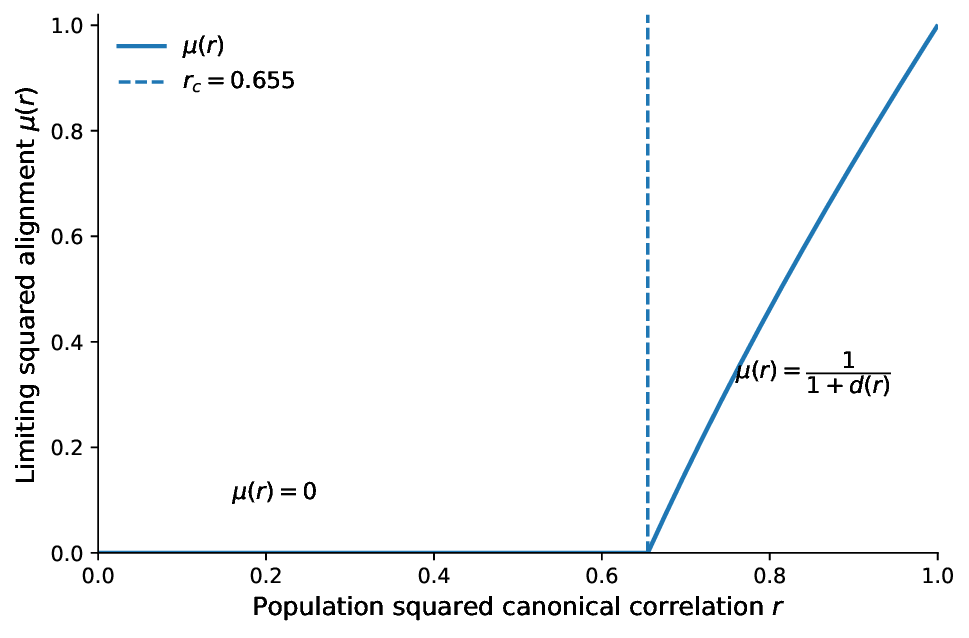}
				\caption{Limiting directional recovery as a function of the population squared canonical correlation $r$ for $(c_1,c_2)=(0.5,0.3)$.
				The vertical line marks the phase-transition threshold $r_c\approx0.655$. In the subcritical regime $r<r_c$, the squared alignment between the corresponding sample and population canonical directions vanishes asymptotically. In the supercritical regime $r>r_c$, the limiting squared alignment is $\mu(r)=1/(1+d(r))$, which increases from zero at $r_c$ to one as $r\to1$.}
				\label{fig:directional_phase_transition}
			\end{figure}
			
			Theorem \ref{th1} reveals two qualitatively different regimes. For a simple
			population spike $r_i>r_c$, the sample canonical direction has a nonvanishing
			but imperfect asymptotic alignment with the corresponding population
			direction, and hence lies asymptotically on a cone around the population
			direction. In contrast, for a non-spiked population canonical correlation
			$r_i<r_c$, the squared alignment between the corresponding sample and
			population canonical directions converges to zero. Thus, the phase
			transition in the sample canonical correlations is accompanied by a
			corresponding phase transition in directional recovery.
			
			We next restrict attention to the simple supercritical regime and study the second-order fluctuations of the squared alignment around its nontrivial first-order limit. 
			
			\begin{theorem}[Central limit theorem]\label{th2}
				Under the same assumptions as in part~(i) of Theorem \ref{th1},
				fix the same simple population spike $r_i>r_c$.
				Then
				\begin{align}
					\sqrt{n}
					\left(
					\langle \bbu_i,\bbnu_i\rangle^2
					-
					\frac{1}{1+d(r_i)}
					\right)
					\xrightarrow{D}
					\mathcal{N}
					\left(
					0,
					\frac{\sigma^2(r_i)}
					{(1+d(r_i))^4}
					\right),
					\label{eq:clt-alignment}
				\end{align}
				where the asymptotic variance $\sigma^2(r_i)$ is given by
				\begin{align}
					\sigma^2(r_i)
					={}&
					A_1(r_i)+A_2(r_i)+A_3(r_i)
					\notag\\
					&-
					\frac{
						4\bigl(c_2(1-r_i)+r_i(1-\gamma_i)\bigr)
					}{
						c_2\gamma_i\Delta_{ii}
					}
					\left(
					F_1(r_i)
					+\gamma_iF_2(r_i)
					-(1-\gamma_i)Q_5(r_i)
					\right)
					\notag\\
					&\qquad\times
					\Bigg(
					2
					\left(
					(1-r_i)
					+
					\frac{r_i(1-\gamma_i)}{c_2}
					\right)^2
					(1-\gamma_i)
					\left(
					\frac{
						\omega(1-c_2)r_i^2
					}{
						(\eta r_i-\omega)^2
						(\eta r_i^2-\omega)
					}
					-
					(1-\gamma_i)Q_5(r_i)
					\right)
					\notag\\
					&\qquad\qquad+
					\frac{2r_i^2}{c_2^2}
					\left(
					1-\frac{1}{c_2}
					\right)
					(1-\gamma_i)^3
					\frac{
						\omega(1-c_2)r_i^2
					}{
						(\eta r_i-\omega)^2
						(\eta r_i^2-\omega)
					}
					\left(
					c_2
					+
					\frac{\omega(1-r_i)}{r_i}
					\right)
					\Bigg)
					\notag\\
					&+
					\left( 
					\frac{
						2\bigl(c_2(1-r_i)+r_i(1-\gamma_i)\bigr)
					}{
						c_2\gamma_i
					}
					\left(
					F_1(r_i)
					+\gamma_iF_2(r_i)
					-(1-\gamma_i)Q_5(r_i)
					\right)
					\right) ^2
					\notag\\
					&\qquad\times
					\frac{
						2(1-r_i)^2
						\bigl(
						2\eta r_i+c_1+c_2-2\omega
						\bigr)
						(\eta r_i^2-\omega)
					}{
						r_i^2
					},
					\label{sigmari2}
				\end{align}
				with $\eta=(1-c_1)(1-c_2), \omega=c_1c_2$.
				Here,
				\begin{align}
					A_1(r_i)
					={}&
					2(1-r_i)^2P_1(r_i)
					-
					4\frac{\eta r_i-\omega}{r_i}
					(1-r_i)^3
					\bigl(
					P_1(r_i)+F_3(r_i)
					\bigr)
					\notag\\
					&+
					\frac{
						2(\eta r_i-\omega)^2
					}{
						r_i^2
					}
					(1-r_i)^4
					\bigl(
					P_1(r_i)
					+2F_3(r_i)
					+Q_3(r_i)
					\bigr),
					\label{eq:A1}
				\end{align}
				\begin{align}
					A_2(r_i)
					={}&
					\frac{
						2r_i^2(1-\gamma_i)^4
					}{
						c_2^2
					}
					\left( 
					Q_3(r_i)
					+
					\left(
					1-\frac{1}{c_2}
					\right)
					\left(
					\frac{
						\omega(1-c_2)r_i^2
					}{
						(\eta r_i-\omega)^2
						(\eta r_i^2-\omega)
					}
					\right)^2
					\right) ,
					\label{eq:A2}
				\end{align}
				and
				\begin{align}
					A_3(r_i)
					={}&
					\frac{
						4r_i(1-r_i)(1-\gamma_i)^2
					}{
						c_2
					}
					\left(
					(1-\gamma_i)Q_3(r_i)
					-
					\gamma_iF_3(r_i)
					\right).
					\label{eq:A3}
				\end{align}
				Here $F_1(r_i)$ and $F_2(r_i)$ are defined in \eqref{QQ1}--\eqref{QQ2},
				$P_1(r_i)$ in \eqref{l31}, and $F_3(r_i)$, $Q_3(r_i)$, and $Q_5(r_i)$ in
				\eqref{EPhi3}, \eqref{Q3}, and \eqref{Q5}, respectively. The quantity $\Delta_{ii}$ is the deterministic diagonal coefficient
				appearing in the outlier expansion
				\eqref{eq:Mlambda-expansion}.
			\end{theorem}
				
				\begin{remark}
					To the best of our knowledge, Theorem \ref{th2} is the first central limit theorem for the squared Euclidean alignment
					\[
					\langle\bbu_i,\bbnu_i\rangle^2
					\]
					between a sample canonical direction and its population counterpart in a finite-rank high-dimensional spiked CCA model. The result is established for a fixed simple spike whose outlier location is separated from those generated by the other population spikes. Repeated spikes, for which individual eigenvectors are not uniquely identifiable, require a subspace-level analysis and are beyond the scope of the present paper. The critical and near-critical regimes, where $r_i-r_c\to0$, are also not covered by the present argument.
				\end{remark}
				
			\section{Estimation of the Asymptotic Mean and Variance of the Squared Alignment}\label{sec:estimation}

			In the previous section, we show that, for each simple spike $r_i$, the squared alignment $$\langle \bbu_i,\bbnu_i\rangle^2$$ does not converge to one. Instead, it converges to a nontrivial limit that is strictly smaller than one. Moreover, the fluctuation around this limit occurs on the $n^{-1/2}$ scale and admits an asymptotically Gaussian description. Thus, under the high-dimensional regime considered in this paper, the sample canonical direction is not a classically consistent estimator of the population canonical direction. Instead, its squared alignment with the population direction has a nontrivial limiting value.
			
			This suggests that, in the present setting, the relevant question is not whether the sample canonical direction fully recovers the population one, but how large the limiting squared alignment is and how it fluctuates around its limit. Theorem \ref{th1} identifies the asymptotic center of this alignment, while Theorem \ref{th2} characterizes its second-order fluctuation. However, both quantities depend on the population spike $r_i$, and are therefore not directly computable from data. The purpose of this section is to construct plug-in estimators for these quantities based on the sample spiked eigenvalue $l_i$, thereby turning the asymptotic results of Theorems \ref{th1} and \ref{th2} into quantities that can be evaluated from the observed sample. In this sense, eigenvalue separation and directional recovery represent two different aspects of the problem. A sample spiked eigenvalue indicates that a signal direction is detectable from the spectrum, whereas the squared alignment quantifies how accurately the associated sample direction represents its population counterpart. The estimators constructed below are designed to assess this directional recovery.
			
			In practical implementation, the asymptotic aspect ratios $c_1$ and $c_2$ are replaced by their observed finite-sample counterparts
			\[
			\frac{p}{n}
			\qquad\text{and}\qquad
			\frac{q}{n},
			\]
			respectively. Thus, all plug-in quantities introduced below are fully determined by the observed dimensions and the sample canonical correlations. For notational simplicity, we continue to use $c_1$ and $c_2$ in the formulas below, with the understanding that they are replaced by $p/n$ and $q/n$ in finite-sample implementation.
			
			\subsection{Estimation of the Asymptotic Mean}
			By Theorem \ref{th:1.4}, for each sample spiked eigenvalue associated with a population spike $r_i$, we have
			$$l_i-\gamma_i\xrightarrow{a.s.}0, ~~\gamma_i=\gamma(r_i),$$ where $\gamma(\cdot)$ is the deterministic mapping from the population spike $r_i$ to its associated outlier location. On the other hand, Theorem \ref{th1} shows that
			$$\langle \bbu_i, \bbnu_i\rangle ^2 \xrightarrow{p} \mu_i, ~~\mu_i:=\frac{1}{1+d(r_i)}.$$
			Accordingly, $\mu_i$ describes the limiting mean of the squared alignment between the sample and population canonical directions. Since $\mu_i$ depends on the unknown spike $r_i$, a natural strategy is to estimate $r_i$ from the observed spiked eigenvalue $l_i$, and then substitute the resulting estimator into the expression for $\mu_i$.
			
			\begin{proposition}\label{pro1}
				Suppose the assumptions of Theorem \ref{th:1.4} and \ref{th1} hold, and fix a simple spike $r_i$. Define $$\hat{r}_i:=\gamma^{-1}(l_i),~~\hat{\mu}_i:=\dfrac{1}{1+d(\hat{r}_i)}.$$
				Then
				$$\hat{r}_i \xrightarrow{p} r_i,~~\hat{\mu}_i \xrightarrow{p} \mu_i.$$
			\end{proposition}
			
			The proof is postponed to Section~\ref{proofpro}.
			
			The construction in Proposition \ref{pro1} can be summarized as the following plug-in procedure.

			\begin{remark}
				The role of Proposition \ref{pro1} is not to estimate the population direction $\bbu_i$ itself, but to estimate the limiting mean $\mu_i$ appearing in Theorem \ref{th1}. Since $\mu_i$ is precisely the limit of the squared alignment $\langle \bbu_i, \bbnu_i\rangle^2$, it may be interpreted as the limiting mean of the squared alignment between the sample and population directions. Proposition \ref{pro1} shows that this quantity can be consistently estimated from the observed sample spiked eigenvalue. Thus, $\hat{\mu}_i$ provides a sample-based measure of recovery quality for a direction that is itself not consistently estimable.
			\end{remark}
			\begin{table}[H]
				\centering
				\caption{Plug-in procedure for estimating the limiting mean $\mu_i$.}
				\label{tab:algorithm_mu}
				\begin{tabular}{p{0.12\textwidth} p{0.78\textwidth}}
					\toprule
					Step & Procedure \\
					\midrule
					1 & 
					Given the centered data matrices $\bbX\in\mathbb{R}^{p\times n}$ and $\bbY\in\mathbb{R}^{q\times n}$, form the sample canonical correlation matrix
					$$
					\bbS_{xx}^{-1}\bbS_{xy}\bbS_{yy}^{-1}\bbS_{yx}.
					$$
					Compute its nonzero eigenvalues $l_1\geq l_2\geq\cdots\geq l_q$. \\[1.2em]
					
					2 &
					Identify the sample spiked eigenvalues. In practice, retain those $l_i$ satisfying
					$$
					l_i>d_+,
					$$
					or, more conservatively,
					\[
					l_i>d_++\varepsilon_n,
					\]
					where $\varepsilon_n\downarrow0$ is chosen sufficiently slowly relative to the bulk-edge fluctuation scale. In finite samples, an edge-based test or a permutation calibration may alternatively be used to distinguish genuine sample outliers from fluctuations near the bulk edge. \\[1.2em]
					
					3 &
					For each retained sample spiked eigenvalue $l_i$, estimate the corresponding population spike by
					$$
					\hat r_i=\gamma^{-1}(l_i)=\dfrac{2c_1c_2-c_1-c_2+l_i+\sqrt{(l_i-d_-)(l_i-d_+)}}{2(c_1c_2-c_1-c_2+1)}.
					$$ \\[1.2em]
					
					4 &
					Compute the plug-in estimator of the limiting mean of the squared alignment:
					$$
					\hat\mu_i=\dfrac{1}{1+d(\hat r_i)}.
					$$
					\\
					\bottomrule
				\end{tabular}
			\end{table}
			\subsection{Estimation of the Asymptotic Variance}
			The limiting mean alone does not fully describe the behavior of the squared alignment. Theorem \ref{th2} further shows that
			$$\sqrt{n}\left( \langle \bbu_i, \bbnu_i\rangle ^2-\frac{1}{1+d(r_i)}\right) \xrightarrow{D}\mathcal{N}(0,\sigma^2(r_i)/(1+d(r_i))^4).$$
			Thus the corresponding asymptotic variance is
			$$\tau_i^2:=\dfrac{\sigma^2(r_i)}{(1+d(r_i))^4}.$$
		    This quantity measures the fluctuation scale of the squared alignment around its limiting mean. As in the previous subsection, we estimate $\tau_i^2$ by substituting $\hat{r}_i$ into the deterministic expression above.
			
			\begin{proposition}\label{pro2}
				Under the assumptions of Proposition \ref{pro1}, further assume that $\sigma^2(\cdot)$ is continuous at $r_i$, and that $0<\sigma^2(r_i)<\infty$.
				Define
				$$\hat{\tau}_i^2:=\dfrac{\sigma^2(\hat{r}_i)}{(1+d(\hat{r}_i))^4},~~\hat{\tau}_i:=\sqrt{\hat{\tau}_i^2},$$ on the event where $\hat{\tau}_i^2\geq 0$.
				Then
				$$\hat{\tau}_i^2\xrightarrow{p}\tau_i^2,~~\hat{\tau}_i\xrightarrow{p}\tau_i.$$
			\end{proposition}
			
			The proof is postponed to Section~\ref{proofpro}.
			
			\begin{remark}
				Proposition \ref{pro2} shows that the asymptotic variance appearing in Theorem \ref{th2} is likewise estimable from the sample. Consequently, one can estimate not only the limiting mean of the squared alignment itself, but also the scale of its random fluctuation. In this sense, Propositions \ref{pro1} and \ref{pro2} provide complementary information: $\hat{\mu}_i$ estimates the limiting mean of the squared alignment, whereas $\hat{\tau}_i$ estimates its fluctuation scale.
			\end{remark}
			
			\subsection{A Real-Data Illustration}
			
			Since the population canonical directions are unobservable in real data, the quantity
			$$\langle \bbu_i,\bbnu_i\rangle^2$$
			cannot be evaluated directly. Accordingly, the purpose of this subsection is not to validate the population-level limit in Theorem \ref{th1}, but rather to illustrate how the plug-in estimator proposed in Proposition \ref{pro1} can be used to assess the quality of the leading sample canonical directions in practice. The real-data illustration focuses on the first-order estimator $\hat{\mu}_i$, since the main purpose is to assess the recovery quality of the detected sample direction.
			
			We consider the limestone grassland community data originally reported in Table A-2 of Gittins \cite{Gittins1985}. The data contain $n=45$ sampling sites from a limestone grassland community in Anglesey, North Wales, and include two blocks of variables: an 8-dimensional block of plant species abundances and a 6-dimensional block of soil characteristics. The same dataset was also used in Bao et al. \cite{bao2019canonical} as a real-data example for high-dimensional canonical correlation analysis. Following the standard CCA preprocessing used there, we first center all variables across samples and then form the sample canonical correlation matrix from the centered data.
			
			Let $\bbX$ and $\bbY$ denote the centered species and soil data matrices, with dimensions $p\times n$ and $q\times n$, respectively. We then form the sample canonical correlation matrix
			$$\bbS_{xx}^{-1}\bbS_{xy}\bbS_{yy}^{-1}\bbS_{yx},$$
			and compute its leading eigenvalues $l_i$. The six sample squared canonical correlations computed from the centered Table A-2 data of Gittins \cite{Gittins1985} are
			$$l_1=0.8293,~~l_2=0.5198,~~l_3=0.3589,~~l_4=0.1074,~~l_5=0.0938,~~l_6=0.0378.$$
			Here $p=8, q=6$, and $n=45$, and the corresponding upper edge is $d_+=0.5236$. Thus only the first sample eigenvalue satisfies $l_1>d_+$. The remaining eigenvalues all lie below the threshold, with the second eigenvalue $l_2=0.5198$ being close to, but still below, $d_+$. Under the asymptotic bulk-edge criterion, only the first sample squared canonical correlation lies above $d_+$. The remaining sample eigenvalues lie within the asymptotic bulk, although the second eigenvalue is close to the upper edge. Thus, the asymptotic edge criterion singles out the first sample canonical component as the only visibly separated component in this dataset.
			
			\begin{figure}[htbp]
				\centering
				\includegraphics[width=0.7\textwidth]{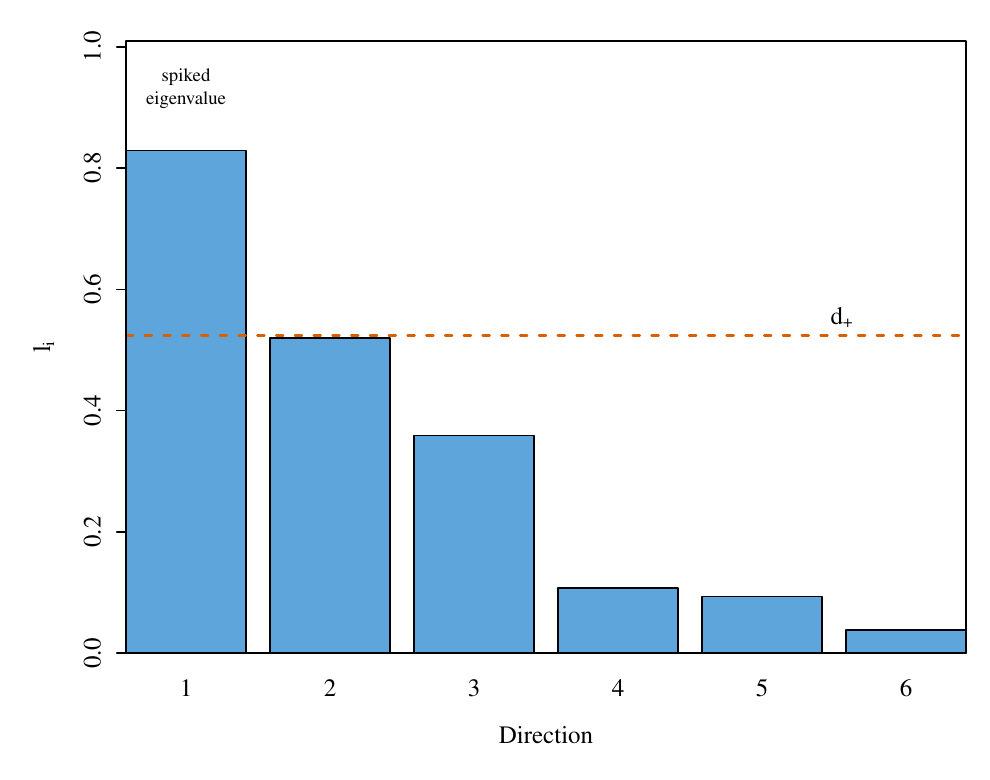}
				\caption{Sample squared canonical correlations for the limestone grassland community data. The dashed horizontal line indicates the asymptotic upper bulk edge $d_+=0.5236$. Only the first sample squared canonical correlation lies above this asymptotic edge.}
				\label{fig:limestone_eigs}
			\end{figure}
			
			We next apply the plug-in estimation procedure stated in Proposition \ref{pro1} to this separated direction. For the sample spiked eigenvalue $l_1=0.8293$, we define
			$$\hat r_1:=\gamma^{-1}(l_1),~~
			\hat\mu_1:=\frac{1}{1+d(\hat r_1)}.$$
			This gives
			$$\hat{r}_1=0.7494,~~\hat{\mu}_1=0.9152.$$
			The value of $l_1$ indicates that the first direction is detectable from the sample spectrum, whereas $\hat{\mu}_1$ gives a different type of information: it quantifies the estimated asymptotic alignment between the first sample canonical direction and its population counterpart. Thus, the first direction in this dataset is not only separated from the bulk, but is also estimated to have relatively high recovery quality.
			
			\begin{table}[htbp]
				\centering
				\caption{The leading sample canonical direction and its estimated recovery quality.}
				\label{tab:realdata_alignment}
				\begin{tabular}{c c}
					\hline
					Quantity & Value \\
					\hline
					\(\bbnu_1\) &
					\((0.7670,\ 0.5141,\ 0.1644,\ 0.1037,\ -0.0952,\ 0.1547,\ 0.1197,\ -0.2495)\) \\
					\(\bxi_1\) &
					\((0.2968,\ 0.4852,\ 0.7982,\ 0.0840,\ -0.1542,\ -0.0924)\) \\
					\(l_1\) & \(0.8293\) \\
					\(\hat r_1\) & \(0.7494\) \\
					\(\hat\mu_1\) & \(0.9152\) \\
					\hline
				\end{tabular}
			\end{table}
			
			The vectors $\bbnu_1$ and $\bxi_1$ in Table \ref{tab:realdata_alignment} are the two sample canonical direction vectors associated with the first sample spiked eigenvalue $l_1$, on the species and soil sides, respectively. They should not be interpreted as consistent estimates of the corresponding population directions. Rather, they identify which variables contribute most to the leading sample direction, while the plug-in estimator $\hat{\mu}_1$ provides the corresponding assessment of recovery quality. On the species side, the first sample canonical direction has its largest coefficients on the first two variables, whose coefficients are 0.7670 and 0.5141. On the soil side, the largest coefficient is attached to the third soil variable, with coefficient 0.7982, while the first two soil variables also have positive contributions. Under the sign convention used in the table, the fifth and sixth soil variables have negative coefficients. Thus, at the sample level, the leading canonical direction is primarily determined by a small number of variables from each block. The value $\hat{\mu}_1=0.9152$ further indicates that this empirically observed direction has relatively high estimated alignment with its population counterpart. Therefore, the loading vectors provide a sample-level interpretation of the leading direction, whereas $\hat{\mu}_1$ quantifies the reliability of this interpretation under the high-dimensional asymptotic framework.
			
			\begin{figure}[htbp]
				\centering
				\includegraphics[width=0.58\textwidth]{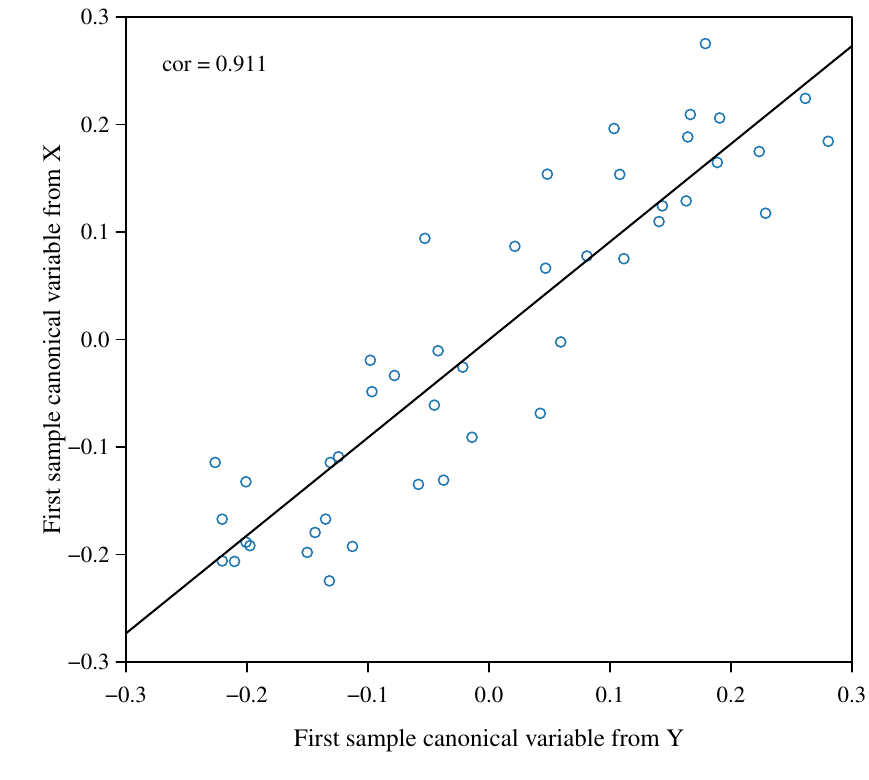}
				\caption{Scatter plot of the first pair of sample canonical variates for the limestone grassland community data.}
				\label{fig:limestone_scatter}
			\end{figure}
			
			Figure \ref{fig:limestone_scatter} plots the first pair of sample variates $\bbnu_1^\top\bbX$ and $\bxi_1^\top\bbY$. The sample correlation between them is approximately $$\sqrt{l_1}=0.9107.$$ The points display a clear positive linear trend, which is consistent with the large first squared canonical correlation. This plot describes the sample-level association between the two canonical variates. The plug-in estimator $\hat{\mu}_1$, on the other hand, provides an estimate of the population-level recovery quality: it measures how reliably the corresponding sample direction represents its population counterpart under the high-dimensional asymptotic framework.
			
			This example illustrates the additional information provided by the proposed plug-in estimator. The inequality $l_1>d_+$ shows that the leading direction is detectable from the sample spectrum, but it does not by itself indicate whether the corresponding sample direction is a reliable representative of the population direction. The estimator $\hat{\mu}_1$ addresses this second question. In the present dataset, $\hat{\mu}_1=0.9152$ suggests that the leading sample canonical direction retains a substantial amount of population-level directional information. Thus, the proposed estimator does not merely restate the presence of a sample spiked eigenvalue; rather, it quantifies the recovery quality of the associated sample direction.
			
			This distinction is important in the high-dimensional regime. Since sample canonical directions are generally not consistent estimators of their population counterparts, the direct recovery of the population directions is no longer the appropriate benchmark. A more relevant question is how much population-level structural information is preserved by the sample directions. Theorem \ref{th1} provides a quantitative answer to this question through the limit of the squared alignment, and Proposition \ref{pro1} makes this quantity estimable from the data. In this sense, $\hat{\mu}_i$ turns the non-consistency of sample canonical directions into a data-based measure of recovery quality, thereby complementing the usual eigenvalue-based detection of spiked directions.
			
			We emphasize that the theoretical model is formulated in isotropic population coordinates. Therefore, the present analysis should be viewed as an illustration of the proposed plug-in methodology under a working high-dimensional spiked CCA approximation, rather than as a direct empirical verification of all population assumptions imposed in the theoretical analysis.
			
			\section{Simulation}\label{simulation}
			We conduct Monte Carlo simulations to examine the finite-sample performance of the first-order limit and the Gaussian fluctuation stated in Theorems \ref{th1} and \ref{th2}. Throughout, we take $p=\lfloor c_1n \rfloor$ and $q=\lfloor c_2n\rfloor$ with $(c_1,c_2)=(0.5,0.3)$. The data matrices are generated under the Gaussian model: the entries $\{y_{ij}\}$ and $\{w_{ij}\}$ are independent and identically distributed (i.i.d.) as $\mathcal N(0,1/n)$. The simulations have two goals: to illustrate the convergence of $\langle \bbu_i,\bbnu_i\rangle^2$ to its deterministic limit, and to assess the normal approximation after the standardization in Theorem \ref{th2}. For each value of $n$, we repeat the experiment independently $5000$ times and compute $\langle \bbu_i,\bbnu_i\rangle^2$ in each repetition. 
			
			We consider two population canonical correlation structures. In the first setting, the population spike structure has rank one, with a single spike $r_1=0.8$. In the second setting, $\boldsymbol\Lambda\bLambda^\top$ has rank three, with spikes $r_1=0.86$, $r_2=0.81$, and $r_3=0.76$. For $(c_1, c_2)=(0.5, 0.3)$, the phase-transition threshold is $$r_c=\sqrt{\dfrac{c_1c_2}{(1-c_1)(1-c_2)}}\approx 0.655.$$ In both settings, the spikes are chosen above the phase-transition threshold, so that the associated sample eigenvalues separate from the bulk.
			
			We first examine the deterministic limit in Theorem \ref{th1}. For the rank-one setting, Figure \ref{Figccalimit1} displays, for
			$n\in\{400, 800, 1200, 1600, 2000, 2400, 2800, 3200, 3600, 4000\}$, the boxplots of $\langle \bbu_1,\bbnu_1\rangle^2$ over $5000$ repetitions, together with the corresponding Monte Carlo mean curve. As $n$ increases, the empirical distribution becomes more concentrated; both the interquartile range and the whiskers shrink. The Monte Carlo mean approaches the deterministic limit $1/(1+d(r_1))$ predicted by Theorem \ref{th1}. A mild finite-sample deviation from the theoretical limit is observed at small $n$ (e.g., $n=400$), but it diminishes rapidly as $n$ grows, indicating a stable convergence to the theoretical limit.
			
			\begin{figure}[htbp]
				\centering
				\includegraphics[scale=0.55]{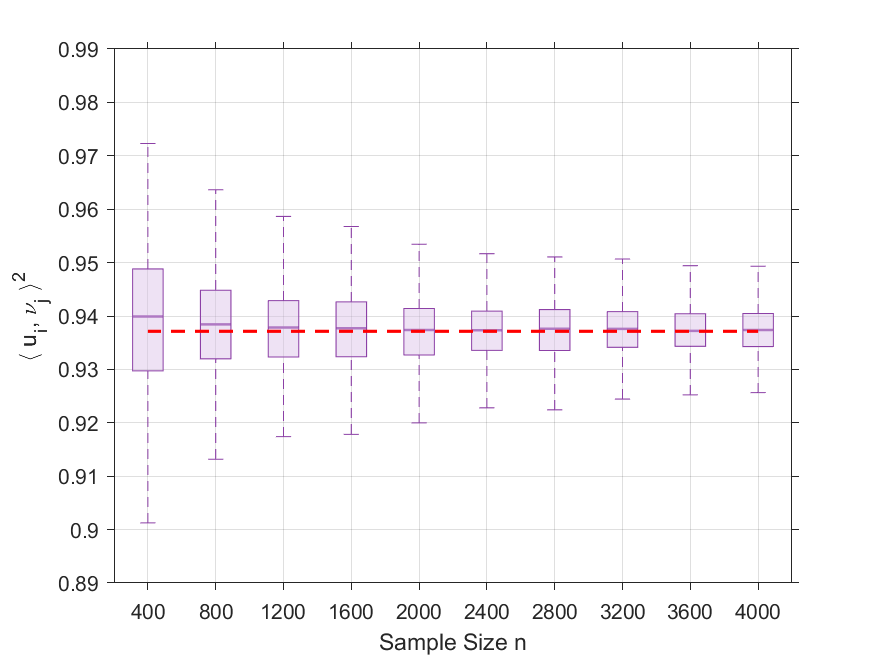}
				\caption{Rank-one setting: boxplots of $\langle \bbu_1, \bbnu_1\rangle^2$ over $5000$ repetitions for $n\in\{400, 800, \ldots,4000\}$, illustrating the first-order limit in Theorem \ref{th1}.}
				\label{Figccalimit1}
			\end{figure}
			
			We next investigate the fluctuation result in Theorem \ref{th2}. Still in the rank-one setting, we fix a large sample size $n=4000$ and standardize $\langle \bbu_1,\bbnu_1\rangle^2$ according to Theorem \ref{th2}: 
			$$
			Z_1
			=
			\frac{\sqrt{n}\Big(\langle \bbu_1,\bbnu_1\rangle^2-1/(1+d(r_1))\Big)}
			{\sigma(r_1)/(1+d(r_1))^2}.
			$$
			Here and below, $\sigma(r_i)=\sqrt{\sigma^2(r_i)}$. Figure \ref{Figccadis1} shows the histogram of $Z_1$ based on $5000$
			repetitions, overlaid with the $\mathcal N(0,1)$ density. The empirical histogram agrees well with the standard normal density, supporting the Gaussian approximation in Theorem \ref{th2} for large but finite $n$.
			
			\begin{figure}[htbp]
				\centering
				\includegraphics[scale=0.5]{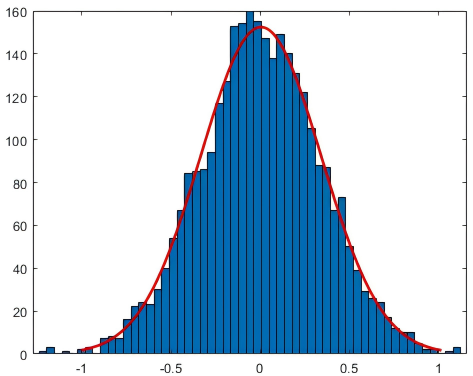}
				\caption{Rank-one setting: histogram of the standardized statistic $Z_1$ at $n=4000$ over $5000$ repetitions, overlaid with the $\mathcal N(0,1)$ density, supporting the Gaussian approximation in Theorem \ref{th2}.}
				\label{Figccadis1}
			\end{figure}
			
			We then turn to the rank-three setting. For each spike $r_i$
			$(i=1,2,3)$, we compute $\langle \bbu_i,\bbnu_i\rangle^2$ over repeated Monte Carlo simulations and plot the running empirical mean against the number of repetitions. Figure \ref{ccalimit2} displays the convergence trajectories together with the corresponding theoretical first-order limits. The three panels show that the running averages stabilize around their predicted limits, providing numerical support for the first-order asymptotic theory. The convergence behavior varies across the three spike components. In all three cases, the running empirical means fluctuate around their corresponding theoretical first-order limits and gradually stabilize as the number of repetitions increases. The components associated with stronger spikes, such as $r_1$ and $r_2$, settle near their limiting values after the initial transient fluctuations. The weakest spike $r_3$ displays a more pronounced initial deviation and a slower stabilization pattern. Overall, these numerical results are broadly consistent with the theoretical
			prediction that stronger spectral separation improves the stability of the associated squared alignment.
			
			\begin{figure}[htbp]
				\centering
				\subfloat[$r_1=0.86$]{\includegraphics[width=.33\linewidth]{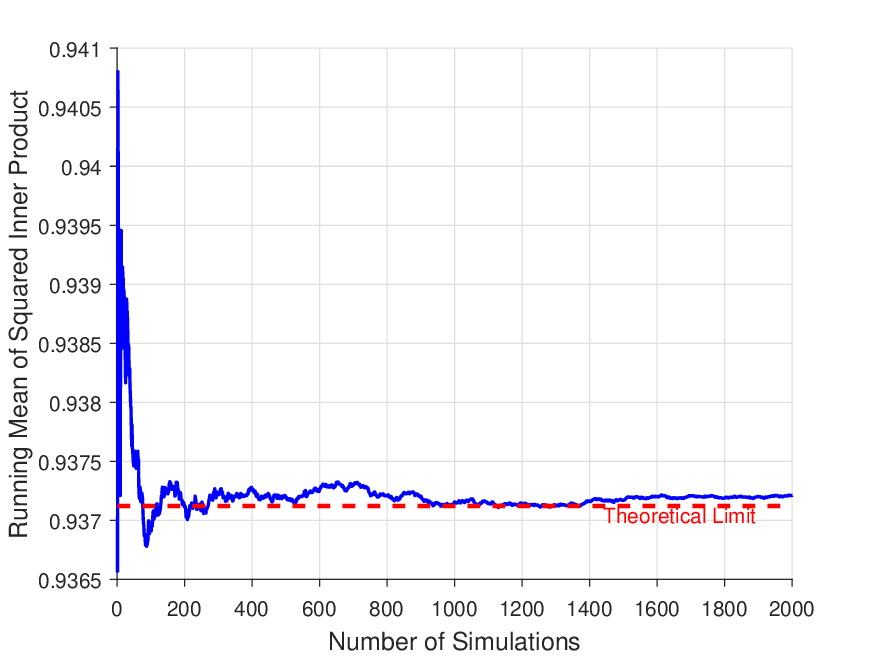}}
				\subfloat[$r_2=0.81$]{\includegraphics[width=.33\linewidth]{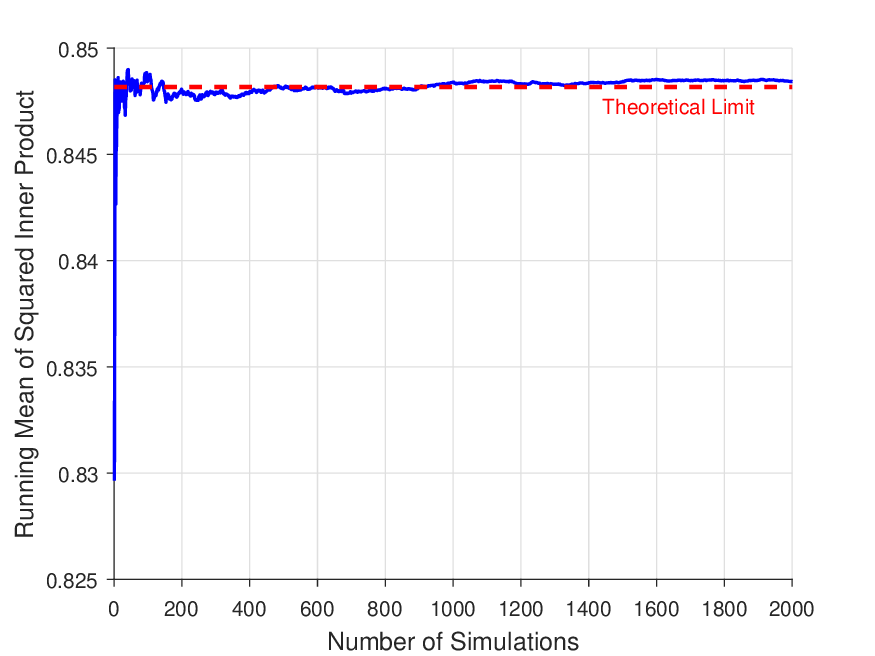}}
				\subfloat[$r_3=0.76$]{\includegraphics[width=.33\linewidth]{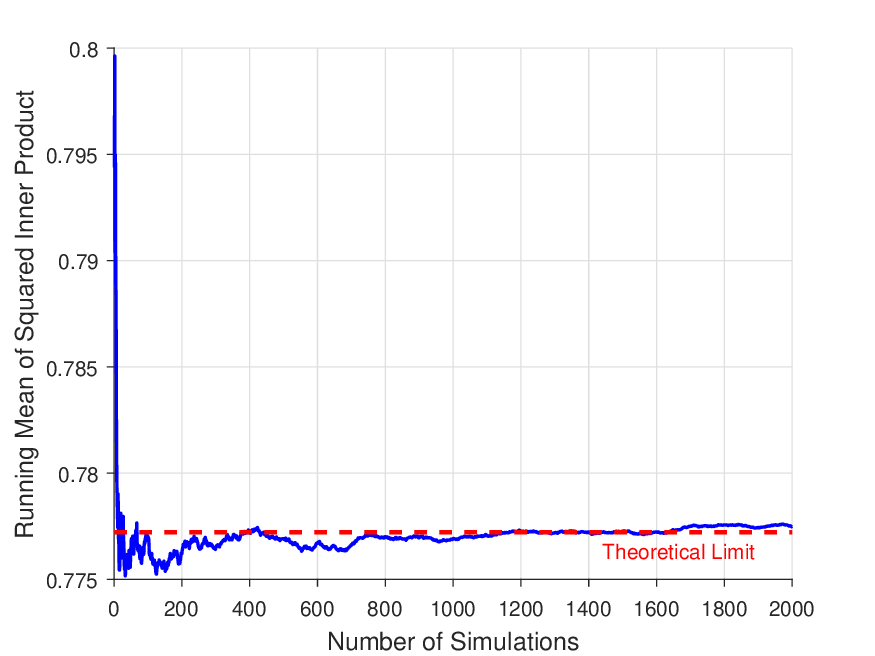}}
				\caption{Rank-three setting: convergence trajectories of the running means of $\langle \bbu_i,\bbnu_i\rangle^2$ $(i=1,2,3)$ over 5000 repetitions, illustrating the first-order limits and the dependence on signal strength.}
				\label{ccalimit2}
			\end{figure}
			
			Finally, we assess the Gaussian fluctuation in the rank-three setting. For $n\in\{2000,4000,6000\}$, we form the standardized statistics
			$$
			Z_i
			=
			\frac{\sqrt{n}\Big(\langle \bbu_i,\bbnu_i\rangle^2-1/(1+d(r_i))\Big)}
			{\sigma(r_i)/(1+d(r_i))^2},
			\qquad i=1,2,3,
			$$
			and repeat the experiment $5000$ times. Figures \ref{ccadis2000}--\ref{ccadis6000} provide three complementary
			diagnostics for the asymptotic normality of the standardized statistics $Z_i$ $(i=1,2,3)$ in the rank-three setting. Figure \ref{ccadis2000} shows the empirical histograms at $n=2000$, overlaid with the standard normal density. The histograms are centered near zero and exhibit an approximately bell-shaped profile, indicating that the normal approximation is already visible at this sample size. Figure \ref{ccadis4000} presents the corresponding Q-Q plots at $n=4000$. The sample quantiles lie close to the reference normal line for all three spiked components, which confirms the normal approximation from a quantile perspective. Mild deviations appear mainly in the tails, as is typical in finite-sample simulations. Figure \ref{ccadis6000} compares the empirical distribution functions at $n=6000$ with the standard normal distribution function. The empirical CDFs almost overlap with the theoretical CDF, showing good agreement at the level
			of the full distribution function. Overall, these graphical diagnostics demonstrate that the finite-sample behavior of the standardized statistics $Z_i$ is consistent with the asymptotic normality result established in Theorem \ref{th2}.
			
			\begin{figure}[htbp]
				\centering
				\subfloat[$r_1=0.86,\, n=2000$]{\includegraphics[width=.33\linewidth]{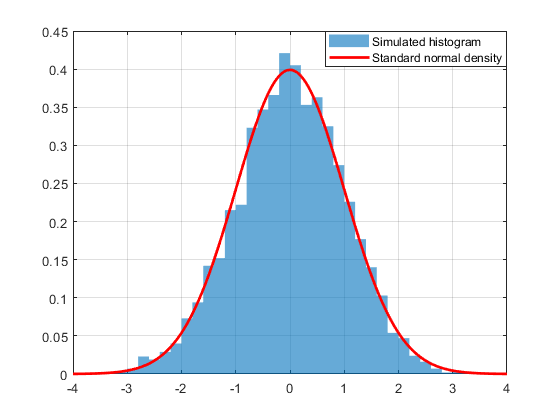}}
				\subfloat[$r_2=0.81,\, n=2000$]{\includegraphics[width=.33\linewidth]{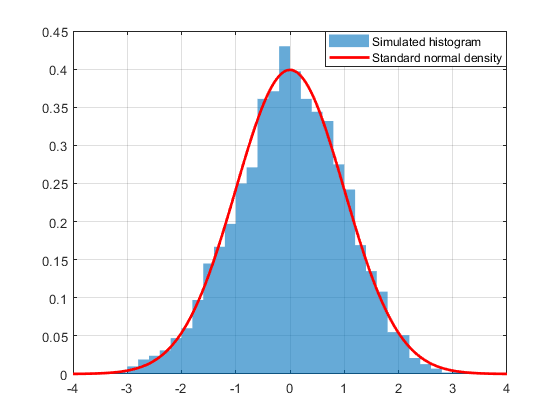}}
				\subfloat[$r_3=0.76,\, n=2000$]{\includegraphics[width=.33\linewidth]{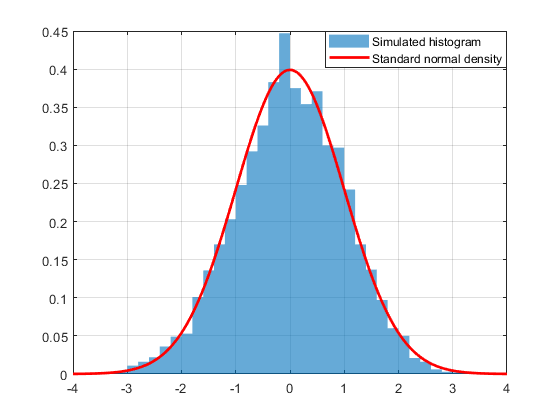}}
				\caption{Rank-three setting: histograms of the standardized statistics $Z_i$ ($i=1,2,3$) at $n=2000$ over $5000$ repetitions, overlaid with the $\mathcal N(0,1)$ density.}
				\label{ccadis2000}
			\end{figure}
			
			\begin{figure}[htbp]
				\centering
				\subfloat[$r_1=0.86,\, n=4000$]{\includegraphics[width=.33\linewidth]{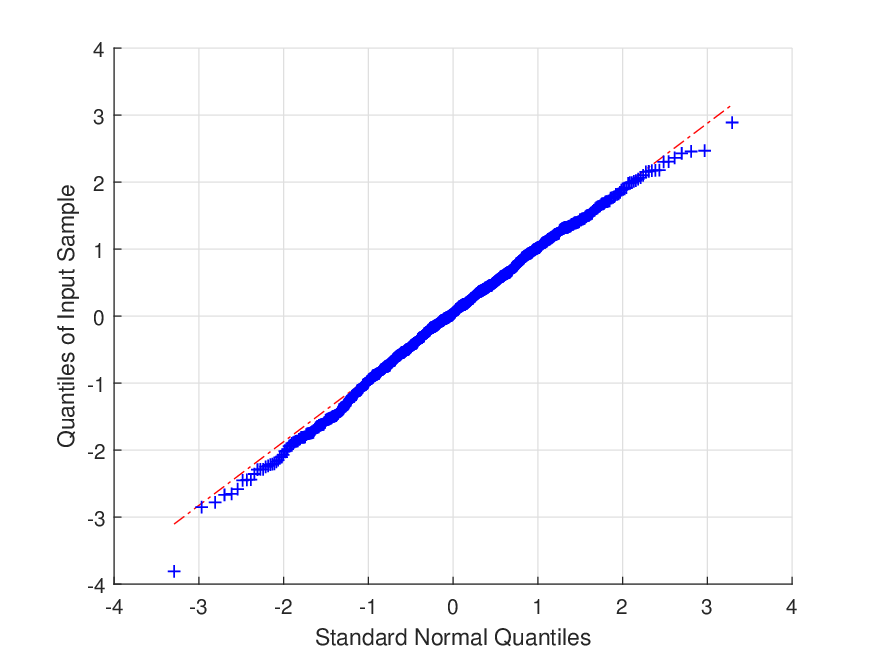}}
				\subfloat[$r_2=0.81,\, n=4000$]{\includegraphics[width=.33\linewidth]{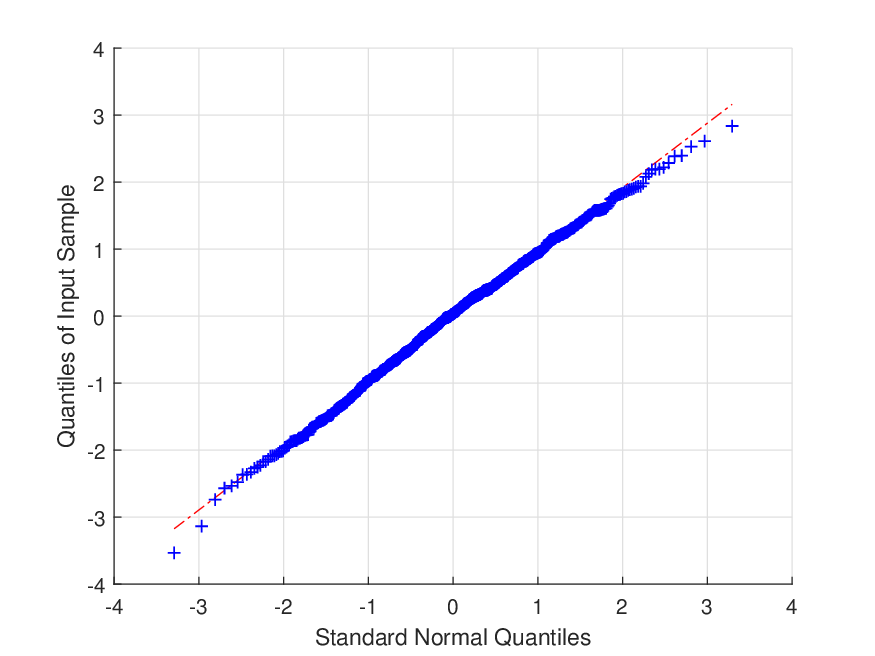}}
				\subfloat[$r_3=0.76,\, n=4000$]{\includegraphics[width=.33\linewidth]{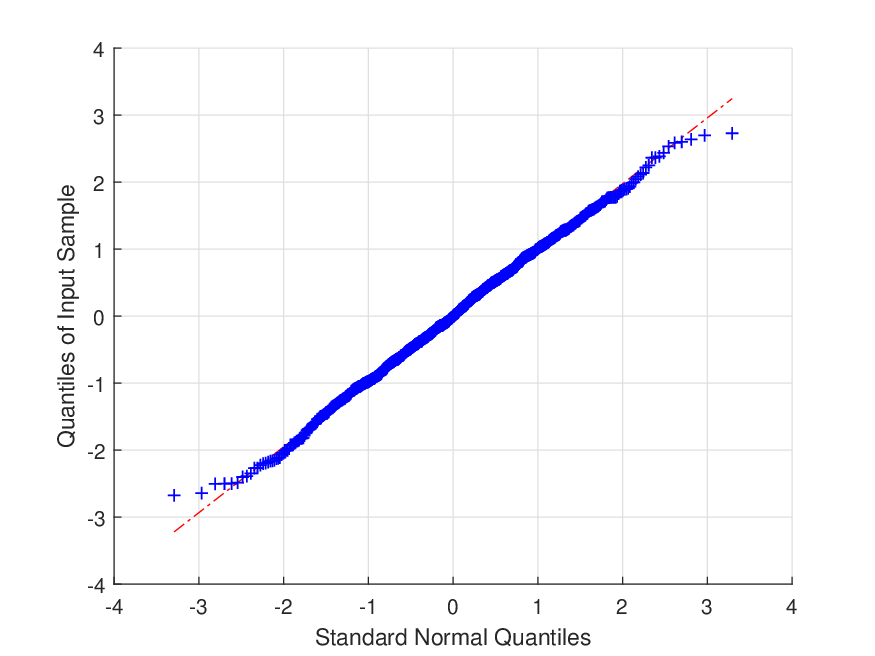}}
				\caption{Rank-three setting: Q-Q plots of the standardized statistics $Z_i$ $(i=1,2,3)$ at $n=4000$ against the standard normal distribution.}
				\label{ccadis4000}
			\end{figure}
			
			\begin{figure}[htbp]
				\centering
				\subfloat[$r_1=0.86,\, n=6000$]{\includegraphics[width=.33\linewidth]{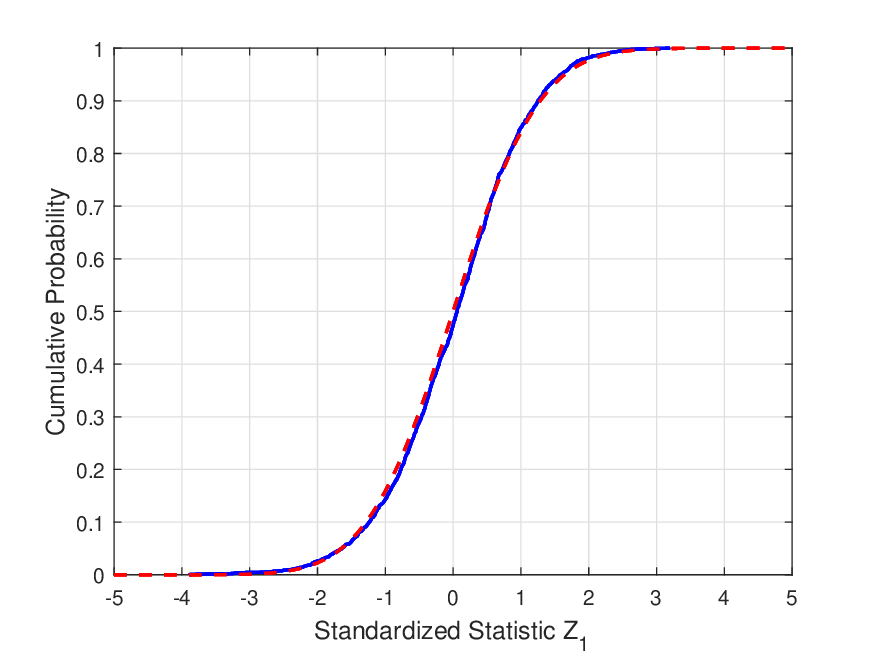}}
				\subfloat[$r_2=0.81,\, n=6000$]{\includegraphics[width=.33\linewidth]{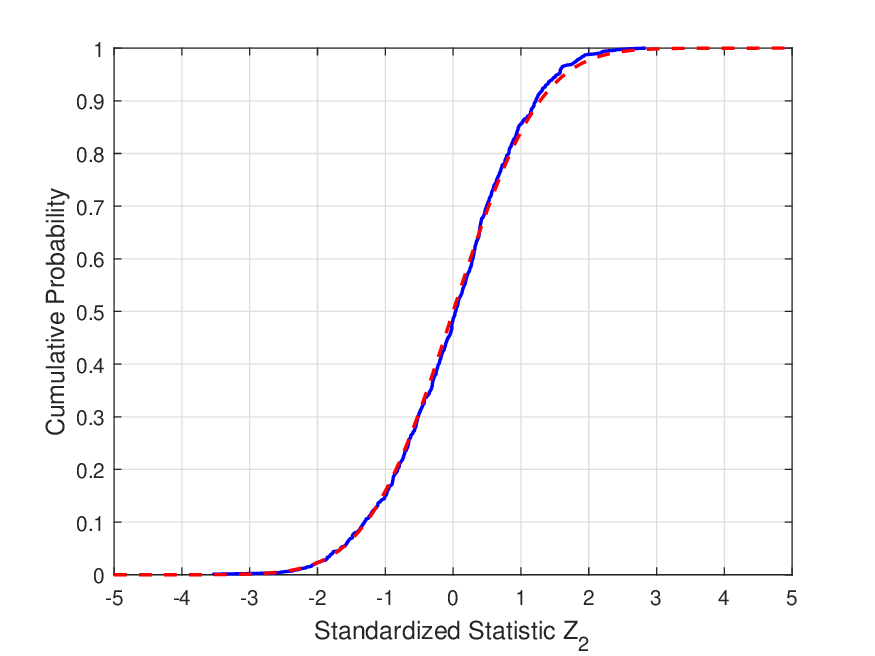}}
				\subfloat[$r_3=0.76,\, n=6000$]{\includegraphics[width=.33\linewidth]{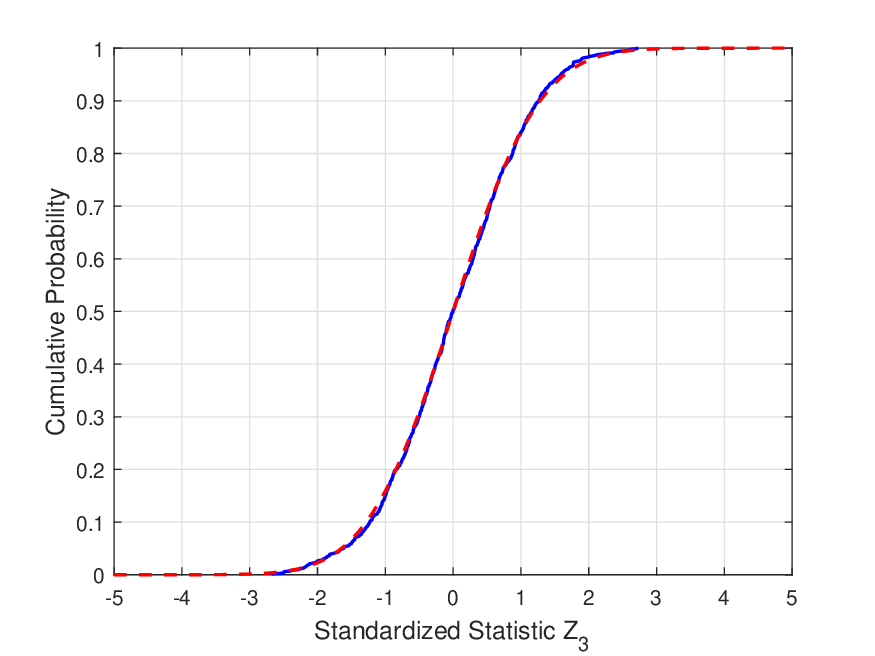}}
				\caption{Rank-three setting: empirical distribution functions of the standardized statistics $Z_i$ $(i=1,2,3)$ at $n=6000$, compared with the standard normal distribution function.}
				\label{ccadis6000}
			\end{figure}
				
				\section{Proof of Main Theorems}\label{zmcl} 
				This section outlines the main proof strategy. We first use the characteristic equation to transform the eigenvector projection problem into a problem concerning quadratic forms of resolvent matrices. We then demonstrate how the proofs of our main results, Theorem \ref{th1} and \ref{th2}, are reduced to the computation of the deterministic limits and fluctuations of these quadratic forms, which are established in the subsequent technical lemmas. 
				
				\subsection{Schur Complement Reduction}\label{ce} 
				Consider the following characteristic equation: 
				\begin{align*} \bbS_{xy}\bbS_{yy}^{-1}\bbS_{yx}\bbnu_i=l_i\bbS_{xx}\bbnu_i, \end{align*} i.e., 
				\begin{align*} \left( \bbS_{xy}\bbS_{yy}^{-1}\bbS_{yx}-l_i\bbS_{xx}\right) \bbnu_i=0, \end{align*} 
				where $l_i$ is the $i$-th eigenvalue of $\bbS_{xx}^{-1}\bbS_{xy}\bbS_{yy}^{-1}\bbS_{yx}$, and $\bbnu_i$ is the eigenvector corresponding to the eigenvalue $l_i$. Using the model $\bbX=\boldsymbol{\Lambda}\bbY+\boldsymbol{\Gamma}\bbW$ and the definitions of the sample covariance matrices, we can write $\bbS_{xx} = \bbX\bbX^\top$ and $\bbS_{xy}\bbS_{yy}^{-1}\bbS_{yx} = \bbX \bbP_y \bbX^\top$, where $\bbP_y = \bbY^\top(\bbY\bbY^\top)^{-1}\bbY$. Note that $\bbP_y$ is an orthogonal projector: $\bbP_y^\top=\bbP_y$ and $\bbP_y^2=\bbP_y$. The equation becomes $\left( \bbX\bbP_y\bbX^\top - l_i \bbX\bbX^\top \right) \bbnu_i = 0$, or 
				\begin{align*} 
					\bbX(\bbP_y-l_i\bbI_n)\bbX^\top \bbnu_i=0. 
				\end{align*} 
			Let $\boldsymbol{\Lambda}$ have the SVD
			\[
			\boldsymbol{\Lambda}
			=
			\bbU
			\begin{pmatrix}
				\boldsymbol{\Lambda}_k & \mathbf 0\\
				\mathbf 0 & \mathbf 0
			\end{pmatrix}
			\bbV^\top,
			\]
			where $\bbU=(\bbU_1,\bbU_2)\in\mathbb R^{p\times p}$ and
			$\bbV=(\bbV_1,\bbV_2)\in\mathbb R^{q\times q}$ are orthogonal,
			and $\boldsymbol{\Lambda}_k\in\mathbb R^{k\times k}$ satisfies
			\[
			\boldsymbol{\Lambda}_k\boldsymbol{\Lambda}_k^\top
			=
			\operatorname{diag}(r_1,\ldots,r_k).
			\]
			Here
			\[
			\bbU_1=(\bbu_1,\ldots,\bbu_k)
			\]
			consists of the population canonical directions associated with the
			$k$ nonzero population squared canonical correlations, while $\bbU_2$
			contains the remaining $p-k$ directions corresponding to zero
			population canonical correlations. We also write $\boldsymbol{\Gamma}=\bbU\mathrm{diag}(\boldsymbol{\Gamma}_k,\bbI_{p-k})\bbU^\top$ with
			$\boldsymbol{\Gamma}_k=\mathrm{diag}(\sqrt{1-r_1},\dots,\sqrt{1-r_k})$.
			 We further analyze the block matrix form of the characteristic equation. Let $\bbU = (\bbU_1, \bbU_2)$ where $\bbU_1$ consists of the first $k$ columns (the population eigenvectors associated with the $k$ nonzero population canonical correlations) and $\bbU_2$ the remaining $p-k$. Substituting the model for $\bbX$, we have that $\bbnu_i$ satisfies 
				\begin{align*} 
					&\left( \bbU_1, \bbU_2 \right) \left(\begin{pmatrix} \boldsymbol{\Lambda}_k &\mathbf 0\\ \mathbf 0 & \mathbf 0{ } \end{pmatrix}(\bbV_1,\bbV_2)^\top \bbY+\begin{pmatrix}\boldsymbol{\Gamma}_k & \mathbf 0 \\ \mathbf 0 & \bbI_{p-k} \end{pmatrix}\bbU^\top \bbW\right)(\bbP_y-l_i\bbI_n)\\ &~~~~~~~~~~~~~~~~~~~~~~~~\times\left(\bbY^\top (\bbV_1,\bbV_2)\begin{pmatrix}\boldsymbol{\Lambda}_k &\mathbf 0 \\ \mathbf 0 & \mathbf 0 \end{pmatrix}+\bbW^\top \bbU\begin{pmatrix}\boldsymbol{\Gamma}_k &\mathbf 0 \\ \mathbf 0 & \bbI_{p-k} \end{pmatrix}\right)\begin{pmatrix} \bbU_1^\top{ } \\ \bbU_2^\top \end{pmatrix}\bbnu_i=0, \end{align*} 
					expanding this yields: 
					\begin{align*} 
						\left( \bbU_1, \bbU_2 \right) \begin{pmatrix} \bbA & \bbB\\ \bbC & \bbD \end{pmatrix} \begin{pmatrix} \boldsymbol{\varpi}_{i1} \\ \boldsymbol{\varpi}_{i2} \end{pmatrix}=0, \end{align*} 
					where $\boldsymbol{\varpi}_{i1} =\bbU_1^\top \bbnu_i$ (a $k \times 1$ vector) and $\boldsymbol{\varpi}_{i2} =\bbU_2^\top \bbnu_i$ (a $(p-k) \times 1$ vector). The matrices are defined as 
					\begin{align*} 
						&\bbA=\bbA(l_i):=(\boldsymbol{\Lambda}_k\bbV_1^\top \bbY\!+\!\boldsymbol{\Gamma}_k \bbU_1^\top \bbW)(\bbP_y\!-\!l_i\bbI_n)(\bbY^\top \bbV_1\boldsymbol{\Lambda}_k\!+\!\bbW^\top \bbU_1\boldsymbol{\Gamma}_k),\\ &\bbB=\bbB(l_i):=(\boldsymbol{\Lambda}_k\bbV_1^\top \bbY\!+\!\boldsymbol{\Gamma}_k \bbU_1^\top \bbW)(\bbP_y\!-\!l_i\bbI_n)\bbW^\top \bbU_2,\\ &\bbC=\bbC(l_i):=\bbU_2^\top \bbW(\bbP_y\!-\!l_i\bbI_n)(\bbY^\top \bbV_1\boldsymbol{\Lambda}_k\!+\!\bbW^\top \bbU_1\boldsymbol{\Gamma}_k),\\ &\bbD=\bbD(l_i):=\bbU_2^\top \bbW(\bbP_y\!-\!l_i\bbI_n)\bbW^\top \bbU_2.
					\end{align*} 
					Since $l_i$ is a sample spike and $l_i\xrightarrow{a.s.}\gamma_i$, on an event whose probability tends to one, the matrix $\bbD(l_i)=\bbU_2^\top\bbW(\bbP_y-l_i\bbI_n)\bbW^\top\bbU_2$ is invertible. On this event, we may apply the Schur complement to the block matrix above. 
					 \begin{align*} 
					 	\left( \bbU_1 ,\bbU_2 \right) \begin{pmatrix}\bbI_k &\bbB\bbD^{-1}\\ 0 &\bbI_{p-k}\end{pmatrix}\begin{pmatrix}\bbA-\bbB\bbD^{-1}\bbC &\mathbf 0\\ \mathbf 0 &\bbD \end{pmatrix}\begin{pmatrix} \boldsymbol{\varpi}_{i1}{ } \\ \bbD^{-1}\bbC\boldsymbol{\varpi}_{i1}+\boldsymbol{\varpi}_{i2}{ } \end{pmatrix}=0, 
					 \end{align*} 
					 This yields the equivalent system of equations: \begin{align*} 
					 	\left\{\begin{matrix} (\bbA-\bbB\bbD^{-1}\bbC)\boldsymbol{\varpi}_{i1}=0,\\ \bbC\boldsymbol{\varpi}_{i1}+\bbD\boldsymbol{\varpi}_{i2}=0, \end{matrix}\right. 
					 \end{align*} which implies \begin{align}\label{twoequation} \left\{\begin{matrix} (\bbA-\bbB\bbD^{-1}\bbC)\boldsymbol{\varpi}_{i1}=0,\\ \boldsymbol{\varpi}_{i2}=-\bbD^{-1}\bbC\boldsymbol{\varpi}_{i1}. \end{matrix}\right. \end{align} The first equation determines the direction of $\boldsymbol{\varpi}_{i1}$, while the second equation links $\boldsymbol{\varpi}_{i2}$ to $\boldsymbol{\varpi}_{i1}$. The vector of interest is $\boldsymbol{\varpi}_{i1}$, which contains the projections $\langle \bbu_j, \bbnu_i \rangle$ for $j=1,\dots,k$. Since $\bbnu_i$ has unit norm, we have $\|\boldsymbol{\varpi}_{i1}\|^2 + \|\boldsymbol{\varpi}_{i2}\|^2 = 1$. Substituting the second equation of \eqref{twoequation} gives the normalization constraint: \begin{align*} \boldsymbol{\varpi}_{i1}^\top \boldsymbol{\varpi}_{i1} + \boldsymbol{\varpi}_{i1}^\top \bbC^\top (\bbD^{-1})^\top \bbD^{-1} \bbC \boldsymbol{\varpi}_{i1} = 1. \end{align*} Since $\bbD$ is symmetric, this simplifies to \begin{align}\label{norm_constraint} \boldsymbol{\varpi}_{i1}^\top \left( \bbI_k+\bbC^\top \bbD^{-2}\bbC\right) \boldsymbol{\varpi}_{i1}=1. \end{align} The $k$-dimensional matrix $\bbA-\bbB\bbD^{-1}\bbC$ above is consistent with the content shown in equation (4.13) of Bao et al. (2019) \cite{bao2019canonical}. For the remainder of the present subsection, we first consider the
					 supercritical case $1\leq i\leq k_0$. The non-spiked case
					 $k_0+1\leq i\leq k$ will be treated separately in the proof of
					 Theorem \ref{th1}. Let 
					 \[
					 M_{\Lambda}(z):=\bbA(z)-\bbB(z)\bbD(z)^{-1}\bbC(z)
					 \]
					 be the (random) $k\times k$ matrix-valued function arising from the Schur complement.
					 For the $i$-th sample outlier $l_i$ with limit $\gamma_i$, we invoke (6.9) of \cite{bao2019canonical}.
					 In particular, for $\alpha,\beta\in I_i$ (the index set associated with the population spike $r_i$), there exists
					 a random matrix $\mathcal{R}$ (defined in Lemma 6.2 of \cite{bao2019canonical}) such that
					 \begin{equation}\label{eq:Mlambda-expansion}
					 	(M_{\Lambda})_{\alpha\beta}(l_i)
					 	= M_{\alpha\beta}(\gamma_i)
					 	-\delta_{\alpha\beta}\Delta_{ii}\,(l_i-\gamma_i)\,(1+o_p(1))
					 	+\frac{1}{\sqrt n}\mathcal{R}_{\alpha\beta}
					 	+o_p(n^{-1/2}),
					 \end{equation}
					 where $M(z)$ is a deterministic diagonal matrix with diagonal entries $m_j(z)$ and $\Delta$ is also diagonal.
					 Moreover, $m_j(\gamma_j)=0$ for each $j$. Let $\mathcal{J}_i$ denote the index set of population eigenvalues equal to $r_i$.
					 For the simple spike $r_i$, we have
					 $$
					 (M_\Lambda)_{ii}(l_i)=O_p(n^{-1/2}).
					 $$
					 Moreover, the off-diagonal blocks satisfy
					 $$
					 (M_\Lambda(l_i))_{-i,i}=O_p(n^{-1/2}),
					 \qquad
					 (M_\Lambda(l_i))_{i,-i}=O_p(n^{-1/2}).
					 $$
					 For $j\neq i$, since $m_j(\gamma_i)\neq0$, the diagonal entries of
					 $M(\gamma_i)$ in the $(-i,-i)$ block are bounded away from zero. Hence
					 $$
					 (M_\Lambda(l_i))_{-i,-i}
					 =
					 (M(\gamma_i))_{-i,-i}+O_p(n^{-1/2})
					 $$
					 is invertible with probability tending to one, and
					 $$
					 \|(M_\Lambda(l_i))_{-i,-i}^{-1}\|=O_p(1).
					 $$
					 
					 Therefore, from the equation $M_{\Lambda}(l_i)\boldsymbol{\varpi}_{i1}=0$ we obtain
					 \[
					 \boldsymbol{\varpi}_{i1,-i}=-(M_{\Lambda}(l_i))_{-i,-i}^{-1}(M_{\Lambda}(l_i))_{-i,i}\,\boldsymbol{\varpi}_{i1,i}
					 =O_p(n^{-1/2})\,\boldsymbol{\varpi}_{i1,i},
					 \]
					 which shows that all components of $\boldsymbol{\varpi}_{i1}$ except the $i$-th one are of order $O_p(n^{-1/2})$. Thus, under the simple spike assumption $\mathcal J_i=\{i\}$,
					$$
					\boldsymbol{\varpi}_{i1,-i}
					=
					O_p(n^{-1/2})\,\boldsymbol{\varpi}_{i1,i}.
					$$
					Since $\|\bbnu_i\|=1$, we have $\boldsymbol{\varpi}_{i1,i}=O_p(1)$, and hence
					$$
					\|\boldsymbol{\varpi}_{i1,-i}\|=O_p(n^{-1/2}).
					$$
					Furthermore, by the \(i\)-th equation of
					\(M_{\Lambda}(l_i)\boldsymbol{\varpi}_{i1}=0\),
					\[
					(M_{\Lambda})_{ii}(l_i)\boldsymbol{\varpi}_{i1,i}
					=
					-
					(M_{\Lambda}(l_i))_{i,-i}
					\boldsymbol{\varpi}_{i1,-i}.
					\]
					Since
					\[
					(M_{\Lambda}(l_i))_{i,-i}=O_p(n^{-1/2}),
					\qquad
					\boldsymbol{\varpi}_{i1,-i}
					=
					O_p(n^{-1/2})\boldsymbol{\varpi}_{i1,i},
					\]
					we obtain
					\[
					(M_{\Lambda})_{ii}(l_i)=O_p(n^{-1}).
					\]
					Therefore, taking \(\alpha=\beta=i\) in
					\eqref{eq:Mlambda-expansion} and using
					\(M_{ii}(\gamma_i)=0\), we obtain
					\[
					\sqrt n(l_i-\gamma_i)
					=
					\frac{\mathcal R_{ii}}{\Delta_{ii}}
					+o_p(1).
					\]

					Let
					\[
					\bbM_i:=\bbI_k+\bbC^\top\bbD^{-2}\bbC.
					\]
					From the Schur complement reduction, we have
					\[
					\boldsymbol{\varpi}_{i1}^\top \bbM_i\boldsymbol{\varpi}_{i1}=1.
					\]
					Moreover, for a simple spike \(r_i\),
					\[
					\boldsymbol{\varpi}_{i1,-i}
					=
					O_p(n^{-1/2})\boldsymbol{\varpi}_{i1,i}.
					\]
					Since \(\|\bbnu_i\|=1\), we have \(\boldsymbol{\varpi}_{i1,i}=O_p(1)\), and hence
					\[
					\|\boldsymbol{\varpi}_{i1,-i}\|=O_p(n^{-1/2}).
					\]
					Define
					\[
					\mathbf r_i
					=
					\boldsymbol{\varpi}_{i1}
					-
					\boldsymbol{\varpi}_{i1,i}\bbe_i.
					\]
					Then the \(i\)-th component of \(\mathbf r_i\) is zero and
					\[
					\boldsymbol{\varpi}_{i1}
					=
					\boldsymbol{\varpi}_{i1,i}\bbe_i+\mathbf r_i,
					\qquad
					\|\mathbf r_i\|=O_p(n^{-1/2}).
					\]
					
					Substituting this decomposition into the normalization constraint gives
					\[
					1
					=
					\boldsymbol{\varpi}_{i1,i}^2[\bbM_i]_{ii}
					+
					2\boldsymbol{\varpi}_{i1,i}\bbe_i^\top\bbM_i\mathbf r_i
					+
					\mathbf r_i^\top\bbM_i\mathbf r_i.
					\]
					By Lemma \ref{ccalemma1}, for \(j\neq i\),
					\[
					[\bbM_i]_{ij}=O_p(n^{-1/2}).
					\]
					Since the \(i\)-th component of \(\mathbf r_i\) is zero and \(k\) is fixed,
					\[
					\bbe_i^\top\bbM_i\mathbf r_i
					=
					\sum_{j\neq i}[\bbM_i]_{ij}(\mathbf r_i)_j
					=
					O_p(n^{-1}).
					\]
					Also, since \(\|\bbM_i\|=O_p(1)\),
					\[
					\mathbf r_i^\top\bbM_i\mathbf r_i
					=
					O_p(n^{-1}).
					\]
					Therefore,
					\[
					1
					=
					\boldsymbol{\varpi}_{i1,i}^2[\bbM_i]_{ii}
					+
					O_p(n^{-1}).
					\]
					Recalling that
					\[
					\boldsymbol{\varpi}_{i1,i}^2
					=
					\langle\bbu_i,\bbnu_i\rangle^2,
					\]
					we obtain the strengthened normalization reduction
					\[
					\langle\bbu_i,\bbnu_i\rangle^2[\bbM_i]_{ii}
					=
					1+O_p(n^{-1}).
					\]
					This strengthened reduction is sufficient for both the first-order limit and the second-order fluctuation. Therefore, up to the order needed for the first- and second-order analysis, the behavior of \(\langle \bbu_i,\bbnu_i\rangle^2\) is governed by the \(i\)-th diagonal entry of \(\bbM_i\). Recalling the definitions of matrices $\bbC$ and $\bbD$, we obtain the decomposition: \begin{align} 
						&\bbM_i\nonumber\\ =&\bbI_k\nonumber\\ &+\boldsymbol{\Gamma}_k \bbU_1^\top \bbW(\bbP_y\!-\!l_i\bbI_n)\bbW^\top \bbU_2\left( \bbU_2^\top \bbW(\bbP_y\!-\!l_i\bbI_n)\bbW^\top \bbU_2\right) ^{-2}\bbU_2^\top \bbW(\bbP_y\!-\!l_i\bbI_n)\bbW^\top \bbU_1\boldsymbol{\Gamma}_k\notag\\ &+\boldsymbol{\Lambda}_k\bbV_1^\top \bbY(\bbP_y\!-\!l_i\bbI_n)\bbW^\top \bbU_2\left( \bbU_2^\top \bbW(\bbP_y\!-\!l_i\bbI_n)\bbW^\top \bbU_2\right) ^{-2}\bbU_2^\top \bbW(\bbP_y\!-\!l_i\bbI_n)\bbY^\top \bbV_1\boldsymbol{\Lambda}_k\notag\\ &+\boldsymbol{\Lambda}_k\bbV_1^\top \bbY(\bbP_y\!-\!l_i\bbI_n)\bbW^\top \bbU_2\left( \bbU_2^\top \bbW(\bbP_y\!-\!l_i\bbI_n)\bbW^\top \bbU_2\right) ^{-2}\bbU_2^\top \bbW(\bbP_y\!-\!l_i\bbI_n)\bbW^\top \bbU_1\boldsymbol{\Gamma}_k\notag\\ &+\boldsymbol{\Gamma}_k \bbU_1^\top \bbW(\bbP_y\!-\!l_i\bbI_n)\bbW^\top \bbU_2\left( \bbU_2^\top \bbW(\bbP_y\!-\!l_i\bbI_n)\bbW^\top \bbU_2\right) ^{-2}\bbU_2^\top \bbW(\bbP_y\!-\!l_i\bbI_n)\bbY^\top \bbV_1\boldsymbol{\Lambda}_k\notag\\ :=&\bbI_k+I_{k1}+I_{k2}+I_{k3}+I_{k4}.\label{Ik1234}
					\end{align} 
					Here $\bbD^{-2}$ means $\bbD^{-1}\bbD^{-1}$.
					Our proofs now hinge on finding the limits and fluctuations of the diagonal elements of these four matrices.

					\subsection{Technical Lemmas}
					We begin with several standard definitions and results that are fundamental to our analysis. 
					\begin{definition}\label{defOp} 
						Let ${\xi_n}$ be a sequence of random variables and $\{C_n\}$ be a sequence of positive constants. We say $\xi_n = O_p(C_n)$ if for all $\varepsilon>0$, there exist constants $N_{\varepsilon}$ and $K_{\varepsilon}$ such that for all $n>N_{\varepsilon}$, $\mathbb{P}(|\xi_n/C_n|\leq K_{\varepsilon})\geq 1-\varepsilon$. \end{definition} 
					Note that if $\sup_n\mathbb{E}|\xi_n|^l<\infty$ for some $l\geq 1$, then $\xi_n=O_p(1)$ by Markov's inequality. The following cumulant expansion formula plays a central role in our computation, with a proof found in \cite{2009LPCentral}. 
					
					\begin{lemma}[Cumulant expansion formula]\label{lemma-cumulant} 
						Let $f:\mathbb{R}\to\mathbb{C}$ be a smooth function, and denote by $f^{(k)}$ its $k$th derivative. Then for every fixed $l\in \mathbb{N}$, we have 
						\begin{align}\label{eqce} 
							\mathbb{E}[\xi f(\xi)]=\sum_{k=0}^l\frac{\kappa_{k+1}(\xi)}{k!}\mathbb{E}[f^{(k)}(\xi)]+\mathcal{R}_{l+1}, 
						\end{align} 
						assuming that all expectations in \eqref{eqce} exist, where $\mathcal{R}_{l+1}$ is the remainder term (depending on $f$ and $\xi$), such that for any $t>0$, 
						\begin{align*}
							\mathcal{R}_{l+1}=O(1)\mathbb{E}[|\xi|^{l+2}I{|\xi|>t}]\sup_{x\in\mathbb{R}}|f^{(l+1)}(x)|+O(1)\mathbb{E}[|\xi|^{l+2}]\sup_{|x|\leq t}|f^{(l+1)}(x)|. 
						\end{align*} 
						In particular, if $\xi$ is from a standard normal distribution, $\kappa_1(\xi)=0, \kappa_2(\xi)=1$, and $\kappa_{j}(\xi)=0$ for $j \ge 3$. We then get the famous Stein's Lemma: \begin{align}\label{sseqce} 
							\mathbb{E}[\xi f(\xi)]=\mathbb{E}[f'(\xi)]. 
						\end{align} 
					\end{lemma} 
					
					\begin{lemma}[Gaussian Poincaré inequality]\label{lemma-pi} 
						Let $\mu$ be the standard Gaussian measure on $M=\mathbb{R}^n$. Assume $f:\mathbb{R}^n\to \mathbb{R}$ is $C^1$. Then, $\operatorname{Var}_{\mu}(f)\leq C\mathbb{E}_{\mu}|\nabla f|^2$ holds with $C=1$. In particular, let $\{\xi_i\}_{i=1}^n$ with standard normal distribution be independent, and assume the function $f:\mathbb{R}^n\to\mathbb{R}$ has first-order continuous partial derivatives. Then we have 
						\begin{align*} \operatorname{Var}[f(\xi_1,\dots,\xi_n)]\leq\mathbb{E}\sum_{i=1}^n\left[\frac{\partial f(\xi_1,\dots,\xi_n)}{\partial \xi_i}\right]^2. \end{align*} 
					\end{lemma}

					The Schur complement reduction above shows that the proofs of Theorems \ref{th1} and \ref{th2} depend on the asymptotic behavior of the entries of
					\[
					\bbM_i:=\bbI_k+\bbC^\top\bbD^{-2}\bbC.
					\]
					Using the decomposition
					\[
					\bbM_i=\bbI_k+I_{k1}+I_{k2}+I_{k3}+I_{k4},
					\]
					we summarize the required deterministic limits, off-diagonal bounds, and fluctuation limits in the following two lemmas. Their proofs are given in Sections \ref{prooflemma} and \ref{ccalemma3proof}.
					
					\begin{lemma}[First-order expansion and off-diagonal control]\label{ccalemma1}
						Under Assumptions \ref{assumption1}--\ref{assumption2},
						fix a simple spiked eigenvalue $r_i$ with $1\leq i\leq k_0$. Then the \(i\)-th diagonal entry of \(\bbM_i\) satisfies
						\[
						[\bbM_i]_{ii}
						=
						1+d(r_i)+O_p(n^{-1/2}),
						\]
						where \(d(r_i)\) is the deterministic function appearing in Theorem \ref{th1}.
						Moreover,
						\[
						[I_{k3}]_{ii}=O_p(n^{-1/2}),
						\qquad
						[I_{k4}]_{ii}=O_p(n^{-1/2}).
						\]
						For any \(a\ne b\), the off-diagonal entries satisfy
						\[
						[I_{kj}]_{ab}=O_p(n^{-1/2}),
						\qquad j=1,2,3,4.
						\]
						Consequently,
						\[
						[\bbM_i]_{ab}=O_p(n^{-1/2}),
						\qquad a\ne b.
						\]
						In particular,
						\[
						\|\bbM_i\|=O_p(1).
						\]
					\end{lemma}
			
			For a deterministic spectral parameter $z$ in a fixed neighborhood $\mathcal N_i$ of $\gamma_i$, define $I_{kj}(z)$, $j=1,\ldots,4$, by replacing $l_i$ with $z$ in the decomposition of the normalization matrix. Write the SVD
			\[
			\bbY=\bbU_y\boldsymbol\Lambda_y\bbV_y^\top
			\]
			and set
			\[
			\mathcal F_i=\sigma(\boldsymbol\Lambda_y,\bbV_y,\bbU_2^\top\bbW).
			\]
			Then $\bbU_1^\top\bbW$ is independent of $\mathcal F_i$, and the left singular-vector matrix $\bbU_y$ is Haar distributed and independent of $\mathcal F_i$. We define
			\[
			[\mathcal V_{k1}(z)]_{ii}
			=[I_{k1}(z)]_{ii}
			-\mathbb E\!\left([I_{k1}(z)]_{ii}\mid\mathcal F_i\right),
			\]
			and
			\[
			[\mathcal V_{k2}(z)]_{ii}
			=[I_{k2}(z)]_{ii}
			-\mathbb E\!\left([I_{k2}(z)]_{ii}\mid\mathcal F_i\right).
			\]
			The first conditional expectation integrates out the Gaussian row $\bbU_1^\top\bbW$, whereas the second integrates out the Haar left singular vectors $\bbU_y$. In particular, the Haar fluctuation is not removed by the conditioning. In the sequel we write
			\[
			\mathcal V_{k1}:=\mathcal V_{k1}(\gamma_i),
			\qquad
			\mathcal V_{k2}:=\mathcal V_{k2}(\gamma_i).
			\]
			The passage from the deterministic point $\gamma_i$ to the random point $l_i$ is performed only afterwards by Taylor expansion.
			The following lemma characterizes the limiting variances and covariances of all fluctuation terms.
			
			\begin{lemma}[Fluctuation of the normalization factor]\label{ccalemma2}
				Under the assumptions of Theorem \ref{th2}, let
				\[
				G_{i,n}
				:=
				\sqrt n\left([\bbM_i]_{ii}-1-d(r_i)\right).
				\]
				Then 
				\[ 
				\begin{aligned} 
					G_{i,n} ={}& X_{1n}+X_{2n}+X_{3n} \\ &- \frac{ 2\left( c_2(1-r_i)+r_i(1-\gamma_i)\right)  }{ c_2\gamma_i } \left(  F_1(r_i) +\gamma_iF_2(r_i) -(1-\gamma_i)Q_5(r_i) \right)  \frac{\mathcal R_{ii}}{\Delta_{ii}} +o_p(1), 
					\end{aligned} 
				\] 
				where $X_{1n}:=\sqrt n[\mathcal V_{k1}]_{ii}$, $X_{2n}:=\sqrt n[\mathcal V_{k2}]_{ii}$, $X_{3n}:=\sqrt n[I_{k3}+I_{k4}]_{ii}$, and \(\mathcal R_{ii}\) and \(\Delta_{ii}\) are those appearing in \eqref{eq:Mlambda-expansion}. Moreover, \[ (X_{1n},X_{2n},X_{3n},\mathcal R_{ii}) \] converges jointly in distribution to a centered Gaussian vector.
				Consequently, \[ G_{i,n} \xrightarrow{D} N\left(0,\sigma^2(r_i)\right). \]
				The asymptotic variance is given by 
				\[ 
				\begin{aligned} 
					\sigma^2(r_i) ={}& A_1(r_i)+A_2(r_i)+A_3(r_i) \\ &- 2 \frac{ 2\left( c_2(1-r_i)+r_i(1-\gamma_i)\right)  }{ c_2\gamma_i\Delta_{ii} } \left(  F_1(r_i) +\gamma_iF_2(r_i) -(1-\gamma_i)Q_5(r_i) \right)  \\ &\qquad\times \lim_{n\to\infty} \operatorname{Cov} \left( X_{1n}+X_{2n}+X_{3n}, \mathcal R_{ii} \right) \\ &+ \left(  \frac{ 2\left( c_2(1-r_i)+r_i(1-\gamma_i)\right)  }{ c_2\gamma_i } \left(  F_1(r_i) +\gamma_iF_2(r_i) -(1-\gamma_i)Q_5(r_i) \right)  \right) ^{2} \\ &\qquad\times \frac{ 2(1-r_i)^2 \big( 2\eta r_i+c_1+c_2-2\omega \big) (\eta r_i^2-\omega) }{ r_i^2 }, 
					\end{aligned} 
				\] 
				where 
				\[ 
				\begin{aligned} 
					A_1(r_i) ={}& 2(1-r_i)^2P_1(r_i) - 4\frac{\eta r_i-\omega}{r_i}(1-r_i)^3 \big(P_1(r_i)+F_3(r_i)\big) \\ &+ \frac{ 2(\eta r_i-\omega)^2 }{ r_i^2 } (1-r_i)^4 \big( P_1(r_i)+2F_3(r_i)+Q_3(r_i) \big), 
				\end{aligned} 
				\] 
				\[ 
				\begin{aligned} 
					A_2(r_i) ={}& \frac{ 2r_i^2(1-\gamma_i)^4 }{ c_2^2 } \left(  Q_3(r_i) + \left(1-\frac{1}{c_2}\right) \left( \frac{ \omega(1-c_2)r_i^2 }{ (\eta r_i-\omega)^2 (\eta r_i^2-\omega) } \right)^2 \right) ,
				\end{aligned} 
				\] 
				and 
				\[ 
				A_3(r_i) = \frac{ 4r_i(1-r_i)(1-\gamma_i)^2 }{ c_2 } \left(  (1-\gamma_i)Q_3(r_i) - \gamma_iF_3(r_i) \right) . \] 
				Furthermore, 
				\[ 
				\begin{aligned} 
					&\lim_{n\to\infty} \operatorname{Cov} \left( X_{1n}+X_{2n}+X_{3n}, \mathcal R_{ii} \right) \\ ={}& 2\left( (1-r_i) + \frac{r_i(1-\gamma_i)}{c_2} \right)^2 (1-\gamma_i) \left(  \frac{ \omega(1-c_2)r_i^2 }{ (\eta r_i-\omega)^2 (\eta r_i^2-\omega) } - (1-\gamma_i)Q_5(r_i) \right)  \\ &+ \frac{2r_i^2}{c_2^2} \left(1-\frac{1}{c_2}\right) (1-\gamma_i)^3 \frac{ \omega(1-c_2)r_i^2 }{ (\eta r_i-\omega)^2 (\eta r_i^2-\omega) } \left( c_2+\frac{\omega(1-r_i)}{r_i} \right). 
				\end{aligned} 
				\]
				The deterministic quantities $P_1(r_i), F_3(r_i), Q_3(r_i), Q_4(r_i), Q_5(r_i), Q_7(r_i)$
				are defined in Section \ref{ccalemma3proof}.
			\end{lemma}
			
			 \subsection{Proofs of Main Results} 
			 Before proving the main results, we summarize the structure of the argument.
			 
			 \medskip
			 \noindent
			 \textbf{Step 1: Eigenvector reduction.}
			 Starting from the generalized eigenvalue equation for the sample canonical correlation matrix, we apply a block decomposition and the Schur complement to obtain a finite-dimensional equation
			 \[
			 \bbM_{\Lambda}(l_i)
			 \boldsymbol{\varpi}_{i1}
			 =
			 0,
			 \]
			 together with
			 \[
			 \boldsymbol{\varpi}_{i2}
			 =
			 -\bbD^{-1}\bbC
			 \boldsymbol{\varpi}_{i1}.
			 \]
			 
			 \medskip
			 \noindent
			 \textbf{Step 2: Isolation of the simple spike and normalization.}
			 For a simple spike, the coordinates of
			 $\boldsymbol{\varpi}_{i1}$ orthogonal to the $i$-th population spike satisfy
			 \[
			 \boldsymbol{\varpi}_{i1,-i}
			 =
			 O_p(n^{-1/2})
			 \varpi_{i1,i}.
			 \]
			 Combining this relation with the normalization of $\bbnu_i$ yields the key reduction
			 \[
			 \langle\bbu_i,\bbnu_i\rangle^2
			 [\bbM_i]_{ii}
			 =
			 1+O_p(n^{-1}).
			 \]
			 
			 \medskip
			 \noindent
			 \textbf{Step 3: Resolvent expansion.}
			 We then establish
			 \[
			 [\bbM_i]_{ii}
			 =
			 1+d(r_i)
			 +
			 n^{-1/2}G_{i,n}
			 +
			 o_p(n^{-1/2}),
			 \]
			 where $G_{i,n}$ has an asymptotically Gaussian distribution. The first-order limit follows directly from the deterministic part of this expansion, while the central limit theorem follows from the fluctuation term together with a first-order Taylor expansion of the reciprocal map.
			 
			 For subcritical regime, the same block representation is instead combined with the divergence of the normalization matrix near the upper bulk edge, which forces
			 \[
			 \|\bbU_1^\top\bbnu_i\|^2
			 \xrightarrow{p}0.
			 \]
			 
			 We now prove the main theorems, assuming Lemmas \ref{ccalemma1} and \ref{ccalemma2} hold. 
			 
			 \begin{proof}[Proof of Theorem \ref{th1}]
			 	We prove the two parts separately.
			 	
			 	\medskip
			 	\noindent
			 	\textbf{Proof of part (i).}
			 	Let $1\leq i\leq k_0$ and suppose that $r_i$ is simple.
			 	By Lemma \ref{ccalemma1},
			 	\[
			 	[\bbM_i]_{ii}
			 	=
			 	1+d(r_i)+O_p(n^{-1/2}).
			 	\]
			 	Together with the strengthened normalization reduction,
			 	\[
			 	\langle\bbu_i,\bbnu_i\rangle^2[\bbM_i]_{ii}
			 	=
			 	1+O_p(n^{-1}),
			 	\]
			 	we obtain
			 	\[
			 	\langle\bbu_i,\bbnu_i\rangle^2
			 	=
			 	\frac{1+O_p(n^{-1})}
			 	{1+d(r_i)+O_p(n^{-1/2})}.
			 	\]
			 	Since $1+d(r_i)>0$, the continuous mapping theorem yields
			 	\[
			 	\langle\bbu_i,\bbnu_i\rangle^2
			 	\xrightarrow{p}
			 	\frac{1}{1+d(r_i)}.
			 	\]
			 	This proves part (i).
			 	
			 	\medskip
			 	\noindent
			 	\textbf{Proof of part (ii).}
			 	Assume that $k_0<k$ and fix
			 	\[
			 	k_0+1\leq i\leq k.
			 	\]
			 	By Theorem \ref{th:1.4},
			 	\begin{equation}\label{eq:nonspike-li-edge}
			 		l_i\xrightarrow{a.s.}d_+.
			 	\end{equation}
			 	
			 	Recall that
			 	\[
			 	\boldsymbol{\varpi}_{i1}
			 	=
			 	\bbU_1^\top\bbnu_i,
			 	\qquad
			 	\boldsymbol{\varpi}_{i2}
			 	=
			 	\bbU_2^\top\bbnu_i.
			 	\]
			 	By the normalization identity \eqref{norm_constraint},
			 	\begin{equation}\label{eq:nonspike-normalization}
			 		1
			 		=
			 		\boldsymbol{\varpi}_{i1}^\top
			 		\left(
			 		\bbI_k+
			 		\bbC^\top
			 		\bbD^{-2}
			 		\bbC
			 		\right)
			 		\boldsymbol{\varpi}_{i1}.
			 	\end{equation}
			 	Hence
			 	\begin{equation}\label{eq:nonspike-varpi-bound}
			 		\|\boldsymbol{\varpi}_{i1}\|^2
			 		\leq
			 		\frac{1}{
			 			1+
			 			\lambda_{\min}
			 			\left(
			 			\bbC^\top
			 			\bbD^{-2}
			 			\bbC
			 			\right)
			 		}.
			 	\end{equation}
			 	We shall show that
			 	\begin{equation}\label{eq:nonspike-min-diverge}
			 		\lambda_{\min}
			 		\left(
			 		\bbC^\top
			 		\bbD^{-2}
			 		\bbC
			 		\right)
			 		\xrightarrow{p}\infty.
			 	\end{equation}
			 	
			 	Set
			 	\[
			 	\bbW_1:=\bbU_1^\top\bbW,
			 	\qquad
			 	\bbW_2:=\bbU_2^\top\bbW.
			 	\]
			 	We use the normalization that the entries of $\bbW$ are
			 	i.i.d.\ $N(0,n^{-1})$.
			 	Since
			 	\[
			 	\bbU_2^\top\bLambda=0
			 	\]
			 	and the restriction of $\bGamma$ to the space spanned by
			 	$\bbU_2$ is the identity, we have
			 	\[
			 	\bbU_2^\top\bbX=\bbW_2.
			 	\]
			 	Define
			 	\[
			 	\bbS_2:=\bbW_2\bbW_2^\top.
			 	\]
			 	Since $(p-k)/n\to c_1<1$ and $k$ is fixed, the extreme
			 	eigenvalues of the Gaussian sample covariance matrix $\bbS_2$
			 	are bounded away from zero and infinity with probability
			 	tending to one. Thus there exist deterministic constants
			 	$0<c<C<\infty$ such that
			 	\begin{equation}\label{eq:S2-bounds-nonspike}
			 		\mathbb P
			 		\left(
			 		c\bbI_{p-k}
			 		\preceq
			 		\bbS_2
			 		\preceq
			 		C\bbI_{p-k}
			 		\right)
			 		\longrightarrow1.
			 	\end{equation}
			 	
			 	Introduce the null sample canonical correlation matrix
			 	\[
			 	\bbT_n
			 	:=
			 	\bbS_2^{-1/2}
			 	\bbW_2\bbP_y\bbW_2^\top
			 	\bbS_2^{-1/2}.
			 	\]
			 	Since $\bbP_y$ is an orthogonal projection,
			 	\[
			 	0\preceq\bbT_n\preceq\bbI_{p-k}.
			 	\]
			 	Moreover,
			 	\begin{equation}\label{eq:D-factorization-nonspike}
			 		\bbD(z)
			 		=
			 		\bbS_2^{1/2}
			 		(\bbT_n-z\bbI_{p-k})
			 		\bbS_2^{1/2}.
			 	\end{equation}
			 	
			 	On the event in \eqref{eq:S2-bounds-nonspike}, we have
			 	$\bbS_2^{-1}\succeq C^{-1}\bbI_{p-k}$. Therefore,
			 	whenever $\bbD(z)$ is invertible,
			 	\begin{align}
			 		\bbD(z)^{-2}
			 		&=
			 		\bbS_2^{-1/2}
			 		(\bbT_n-z\bbI)^{-1}
			 		\bbS_2^{-1}
			 		(\bbT_n-z\bbI)^{-1}
			 		\bbS_2^{-1/2}
			 		\nonumber\\
			 		&\succeq
			 		C^{-1}
			 		\bbS_2^{-1/2}
			 		(\bbT_n-z\bbI)^{-2}
			 		\bbS_2^{-1/2}.
			 		\label{eq:D-inverse-lower-nonspike}
			 	\end{align}
			 	
			 	For any fixed $\varepsilon>0$,
			 	\[
			 	(\bbT_n-z\bbI)^{-2}
			 	\succeq
			 	\left(
			 	(\bbT_n-z\bbI)^2
			 	+\varepsilon^2\bbI
			 	\right)^{-1}.
			 	\]
			 	It follows that
			 	\begin{equation}\label{eq:C-D-lower-nonspike}
			 		\bbC(z)^\top\bbD(z)^{-2}\bbC(z)
			 		\succeq
			 		C^{-1}\bbQ_{n,\varepsilon}(z),
			 	\end{equation}
			 	where
			 	\[
			 	\bbQ_{n,\varepsilon}(z)
			 	:=
			 	\bbK_n(z)^\top
			 	\bbR_{n,\varepsilon}(z)
			 	\bbK_n(z),
			 	\]
			 	\[
			 	\bbK_n(z)
			 	:=
			 	\bbS_2^{-1/2}\bbC(z),
			 	\]
			 	and
			 	\[
			 	\bbR_{n,\varepsilon}(z)
			 	:=
			 	\left(
			 	(\bbT_n-z\bbI)^2
			 	+\varepsilon^2\bbI
			 	\right)^{-1}.
			 	\]
			 	
			 	We next study $\bbQ_{n,\varepsilon}(z)$.
			 	Recall that
			 	\[
			 	\bbU_1^\top\bbX
			 	=
			 	\boldsymbol{\Lambda}_k\bbV_1^\top\bbY
			 	+
			 	\boldsymbol{\Gamma}_k\bbW_1.
			 	\]
			 	Since
			 	\[
			 	(\bbP_y-z\bbI_n)\bbY^\top
			 	=
			 	(1-z)\bbY^\top,
			 	\]
			 	we can write
			 	\begin{equation}\label{eq:K-decomp-nonspike}
			 		\bbK_n(z)
			 		=
			 		\bbF_n(z)
			 		+
			 		\bbL_n(z)\bbW_1^\top
			 		\boldsymbol{\Gamma}_k,
			 	\end{equation}
			 	where
			 	\[
			 	\bbF_n(z)
			 	=
			 	(1-z)
			 	\bbS_2^{-1/2}
			 	\bbW_2\bbY^\top
			 	\bbV_1\boldsymbol{\Lambda}_k
			 	\]
			 	and
			 	\[
			 	\bbL_n(z)
			 	=
			 	\bbS_2^{-1/2}
			 	\bbW_2(\bbP_y-z\bbI_n).
			 	\]
			 	
			 	Conditional on $(\bbY,\bbW_2)$, the rows of $\bbW_1$ are
			 	independent Gaussian vectors with covariance $n^{-1}\bbI_n$
			 	and are independent of
			 	$\bbF_n(z)$, $\bbL_n(z)$ and
			 	$\bbR_{n,\varepsilon}(z)$.
			 	Therefore,
			 	\begin{align}
			 		&\mathbb E
			 		\left(
			 		\bbQ_{n,\varepsilon}(z)
			 		\mid
			 		\bbY,\bbW_2
			 		\right)
			 		\nonumber\\
			 		&\qquad=
			 		\bbF_n(z)^\top
			 		\bbR_{n,\varepsilon}(z)
			 		\bbF_n(z)
			 		+
			 		q_{n,\varepsilon}(z)
			 		\boldsymbol{\Gamma}_k^2,
			 		\label{eq:Q-cond-nonspike}
			 	\end{align}
			 	where
			 	\begin{equation}\label{eq:q-def-nonspike}
			 		q_{n,\varepsilon}(z)
			 		:=
			 		\frac1n
			 		\operatorname{tr}
			 		\left(
			 		\bbR_{n,\varepsilon}(z)
			 		\bbL_n(z)\bbL_n(z)^\top
			 		\right).
			 	\end{equation}
			 	
			 	Using $\bbP_y^2=\bbP_y$, we obtain
			 	\[
			 	\bbL_n(z)\bbL_n(z)^\top
			 	=
			 	z^2\bbI_{p-k}
			 	+
			 	(1-2z)\bbT_n.
			 	\]
			 	Thus
			 	\begin{equation}\label{eq:q-trace-nonspike}
			 		q_{n,\varepsilon}(z)
			 		=
			 		\frac1n
			 		\operatorname{tr}
			 		\frac{
			 			z^2\bbI_{p-k}
			 			+(1-2z)\bbT_n
			 		}{
			 			(\bbT_n-z\bbI_{p-k})^2
			 			+\varepsilon^2\bbI_{p-k}
			 		}.
			 	\end{equation}
			 	
			 	For every fixed $\varepsilon>0$, Gaussian quadratic-form
			 	concentration gives
			 	\begin{equation}\label{eq:Q-concentration-nonspike}
			 		\left\|
			 		\bbQ_{n,\varepsilon}(d_+)
			 		-
			 		\mathbb E
			 		\left(
			 		\bbQ_{n,\varepsilon}(d_+)
			 		\mid\bbY,\bbW_2
			 		\right)
			 		\right\|
			 		=o_p(1).
			 	\end{equation}
			 	To see this, conditional on $(\bbY,\bbW_2)$, each entry of
			 	the difference is a finite sum of centered Gaussian linear and
			 	quadratic forms in the rows of $\bbW_1$. For example, if
			 	\[
			 	\bbH_n
			 	=
			 	\bbL_n(d_+)^\top
			 	\bbR_{n,\varepsilon}(d_+)
			 	\bbL_n(d_+),
			 	\]
			 	then, for fixed $\varepsilon$,
			 	\[
			 	\|\bbH_n\|=O_p(1),
			 	\]
			 	and for $\bbw\sim N(0,n^{-1}\bbI_n)$,
			 	\[
			 	\operatorname{Var}
			 	\left(
			 	\bbw^\top\bbH_n\bbw
			 	\mid\bbY,\bbW_2
			 	\right)
			 	=
			 	\frac{2}{n^2}\operatorname{tr}\bbH_n^2
			 	=
			 	O_p(n^{-1}).
			 	\]
			 	The same estimate applies to the off-diagonal and cross terms.
			 	Since $k$ is fixed, \eqref{eq:Q-concentration-nonspike} follows.
			 	
			 	Moreover,
			 	\[
			 	\boldsymbol{\Gamma}_k^2
			 	=
			 	\operatorname{diag}(1-r_1,\ldots,1-r_k)
			 	\succeq
			 	(1-r_1)\bbI_k.
			 	\]
			 	Since the first term on the right-hand side of
			 	\eqref{eq:Q-cond-nonspike} is positive semidefinite,
			 	\eqref{eq:Q-concentration-nonspike} yields
			 	\begin{equation}\label{eq:Q-min-nonspike}
			 		\lambda_{\min}
			 		\left(
			 		\bbQ_{n,\varepsilon}(d_+)
			 		\right)
			 		\geq
			 		(1-r_1)q_{n,\varepsilon}(d_+)
			 		+o_p(1).
			 	\end{equation}
			 	
			 	Let
			 	\[
			 	\widetilde l_1,\ldots,\widetilde l_q
			 	\]
			 	denote the nonzero eigenvalues of $\bbT_n$.
			 	Their empirical distribution converges to the Wachter
			 	distribution with density
			 	\[
			 	f(x)
			 	=
			 	\frac{1}{2\pi c_2}
			 	\frac{\sqrt{(d_+-x)(x-d_-)}}
			 	{x(1-x)}
			 	I(d_-\leq x\leq d_+).
			 	\]
			 	For fixed $\varepsilon>0$, define
			 	\[
			 	g_\varepsilon(x)
			 	:=
			 	\frac{
			 		d_+^2+(1-2d_+)x
			 	}{
			 		(x-d_+)^2+\varepsilon^2
			 	}.
			 	\]
			 	For $0\leq x\leq1$,
			 	\[
			 	g_\varepsilon(x)\geq0.
			 	\]
			 	Therefore, by \eqref{eq:q-trace-nonspike},
			 	\[
			 	q_{n,\varepsilon}(d_+)
			 	\geq
			 	\frac1n
			 	\sum_{\alpha=1}^{q}
			 	g_\varepsilon(\widetilde l_\alpha).
			 	\]
			 	Hence, for every fixed $\varepsilon>0$,
			 	\begin{equation}\label{eq:q-limit-nonspike}
			 		\liminf_{n\to\infty}
			 		q_{n,\varepsilon}(d_+)
			 		\geq
			 		q_\varepsilon
			 		:=
			 		c_2
			 		\int_{d_-}^{d_+}
			 		g_\varepsilon(x)f(x)\,dx
			 	\end{equation}
			 	in probability.
			 	
			 	We now study $q_\varepsilon$ as
			 	$\varepsilon\downarrow0$.
			 	Since
			 	\[
			 	d_+^2+(1-2d_+)d_+
			 	=
			 	d_+(1-d_+)>0,
			 	\]
			 	there exist constants $\delta>0$ and $C_1>0$ such that
			 	\[
			 	d_+^2+(1-2d_+)x
			 	\geq C_1,
			 	\qquad
			 	d_+-\delta\leq x\leq d_+.
			 	\]
			 	Furthermore, from the explicit Wachter density, there exists
			 	$c_2'>0$ such that
			 	\[
			 	f(d_+-s)\geq c_2'\sqrt{s},
			 	\qquad
			 	0<s<\delta.
			 	\]
			 	It follows that
			 	\begin{align}
			 		q_\varepsilon
			 		&\geq
			 		c
			 		\int_0^\delta
			 		\frac{s^{1/2}}
			 		{s^2+\varepsilon^2}\,ds
			 		\nonumber\\
			 		&=
			 		c\varepsilon^{-1/2}
			 		\int_0^{\delta/\varepsilon}
			 		\frac{t^{1/2}}{1+t^2}\,dt.
			 		\label{eq:q-diverge-nonspike}
			 	\end{align}
			 	Consequently,
			 	\begin{equation}\label{eq:qeps-infty-nonspike}
			 		q_\varepsilon\longrightarrow\infty
			 		\qquad
			 		\text{as }\varepsilon\downarrow0.
			 	\end{equation}
			 	
			 	It remains to replace $d_+$ by $l_i$.
			 	For every fixed $\varepsilon>0$, the map
			 	\[
			 	z\longmapsto\bbQ_{n,\varepsilon}(z)
			 	\]
			 	is locally Lipschitz on a fixed neighborhood of $d_+$,
			 	with Lipschitz constant bounded in probability.
			 	Indeed,
			 	\[
			 	\|\bbR_{n,\varepsilon}(z)\|
			 	\leq\varepsilon^{-2},
			 	\]
			 	and, for fixed $\varepsilon$, its derivative with respect to
			 	$z$ is uniformly bounded on this neighborhood, while
			 	$\|\bbK_n(z)\|=O_p(1)$.
			 	Therefore, by \eqref{eq:nonspike-li-edge},
			 	\begin{equation}\label{eq:Q-li-dplus-nonspike}
			 		\left\|
			 		\bbQ_{n,\varepsilon}(l_i)
			 		-
			 		\bbQ_{n,\varepsilon}(d_+)
			 		\right\|
			 		=o_p(1).
			 	\end{equation}
			 	
			 	We finally take the limits in the appropriate order.
			 	Let $M>0$ be arbitrary.
			 	By \eqref{eq:qeps-infty-nonspike}, we can first choose a
			 	sufficiently small but fixed $\varepsilon>0$ such that
			 	\[
			 	C^{-1}(1-r_1)q_\varepsilon>2M.
			 	\]
			 	Keeping this $\varepsilon$ fixed and then letting $n\to\infty$,
			 	\eqref{eq:C-D-lower-nonspike},
			 	\eqref{eq:Q-min-nonspike},
			 	\eqref{eq:q-limit-nonspike},
			 	and \eqref{eq:Q-li-dplus-nonspike} imply
			 	\[
			 	\mathbb P
			 	\left(
			 	\lambda_{\min}
			 	\left(
			 	\bbC^\top
			 	\bbD^{-2}
			 	\bbC
			 	\right)>M
			 	\right)
			 	\longrightarrow1.
			 	\]
			 	Since $M$ is arbitrary,
			 	\eqref{eq:nonspike-min-diverge} follows.
			 	
			 	Returning to \eqref{eq:nonspike-varpi-bound}, we have
			 	\[
			 	\|\boldsymbol{\varpi}_{i1}\|^2
			 	\xrightarrow{p}0.
			 	\]
			 	However, the conclusion of the theorem concerns only the
			 	$i$-th population direction. Since
			 	\[
			 	\langle\bbu_i,\bbnu_i\rangle
			 	=
			 	\bbe_i^\top\boldsymbol{\varpi}_{i1},
			 	\]
			 	we have
			 	\[
			 	0
			 	\leq
			 	\langle\bbu_i,\bbnu_i\rangle^2
			 	\leq
			 	\|\boldsymbol{\varpi}_{i1}\|^2.
			 	\]
			 	Therefore,
			 	\[
			 	\langle\bbu_i,\bbnu_i\rangle^2
			 	\xrightarrow{p}0.
			 	\]
			 	This proves part (ii), and hence completes the proof of
			 	Theorem \ref{th1}.
			 \end{proof}
			 
			 \begin{proof}[Proof of Theorem \ref{th2}]
			 	By Lemma \ref{ccalemma2}, we have
			 	\[
			 	[\bbM_i]_{ii}
			 	=
			 	1+d(r_i)+n^{-1/2}G_{i,n}+o_p(n^{-1/2}),
			 	\]
			 	where
			 	\[
			 	G_{i,n}\xrightarrow{D}\mathcal N(0,\sigma^2(r_i)).
			 	\]
			 	Using the strengthened normalization reduction,
			 	\[
			 	\langle\bbu_i,\bbnu_i\rangle^2[\bbM_i]_{ii}
			 	=
			 	1+O_p(n^{-1}),
			 	\]
			 	we obtain
			 	\[
			 	\langle\bbu_i,\bbnu_i\rangle^2
			 	=
			 	\frac{1+O_p(n^{-1})}
			 	{1+d(r_i)+n^{-1/2}G_{i,n}+o_p(n^{-1/2})}.
			 	\]
			 	Since \(1+d(r_i)>0\), a first-order Taylor expansion yields
			 	\[
			 	\langle\bbu_i,\bbnu_i\rangle^2
			 	=
			 	\frac{1}{1+d(r_i)}
			 	-
			 	\frac{n^{-1/2}G_{i,n}}
			 	{(1+d(r_i))^2}
			 	+
			 	o_p(n^{-1/2}).
			 	\]
			 	Therefore,
			 	\[
			 	\sqrt n
			 	\left(
			 	\langle\bbu_i,\bbnu_i\rangle^2
			 	-
			 	\frac{1}{1+d(r_i)}
			 	\right)
			 	=
			 	-
			 	\frac{G_{i,n}}{(1+d(r_i))^2}
			 	+
			 	o_p(1).
			 	\]
			 	Since
			 	\[
			 	G_{i,n}\xrightarrow{D}\mathcal N(0,\sigma^2(r_i)),
			 	\]
			 	Slutsky's theorem implies
			 	\[
			 	\sqrt n
			 	\left(
			 	\langle\bbu_i,\bbnu_i\rangle^2
			 	-
			 	\frac{1}{1+d(r_i)}
			 	\right)
			 	\xrightarrow{D}
			 	\mathcal N\left(
			 	0,
			 	\frac{\sigma^2(r_i)}{(1+d(r_i))^4}
			 	\right).
			 	\]
			 	This proves Theorem \ref{th2}.
			 \end{proof}
			 		
			 		\section{Proofs of Technical Lemmas}\label{prooflemma} 
			 		This section provides the proofs for the key technical lemmas, \ref{ccalemma1} and \ref{ccalemma2}. 
			 		
			 		\begin{proof}[Proof of Lemma \ref{ccalemma1}]
			 			Based on the decomposition in Section~\ref{zmcl}, write
			 			\[
			 			I_{k1}
			 			=
			 			\boldsymbol{\Gamma}_k \bbU_1^\top \bbW(\bbP_y-l_i\bbI_n)\bbW^\top \bbU_2
			 			\Phi^2(l_i)
			 			\bbU_2^\top \bbW(\bbP_y-l_i\bbI_n)\bbW^\top \bbU_1
			 			\boldsymbol{\Gamma}_k,
			 			\]
			 			where
			 			\[
			 			\Phi(z)=
			 			\left(
			 			\bbU_2^\top\bbW(\bbP_y-z\bbI_n)\bbW^\top\bbU_2
			 			\right)^{-1}.
			 			\]
			 			Conditional on \(\bbY\) and \(\bbU_2^\top\bbW\), the rows of \(\bbU_1^\top\bbW\) are independent centered Gaussian vectors with covariance \(n^{-1}\bbI_n\). 
			 			\iffalse
			 			Hence
			 			\[
			 			\mathbb E\!\left(I_{k1}\mid \bbY,\bbU_2^\top\bbW\right)
			 			=
			 			\frac{1}{n}
			 			\tr\!\left( 
			 			(\bbP_y-l_i\bbI_n)\bbW^\top\bbU_2
			 			\Phi^2(l_i)
			 			\bbU_2^\top\bbW(\bbP_y-l_i\bbI_n)
			 			\right) 
			 			\boldsymbol{\Gamma}_k\boldsymbol{\Gamma}_k^\top .
			 			\]
			 			
			 			Define
			 			\[
			 			\mathcal V_{k1}
			 			:=
			 			I_{k1}
			 			-\frac{1}{n}
			 			\tr\!\left( 
			 			(\bbP_y-l_i\bbI_n)\bbW^\top\bbU_2
			 			\Phi^2(l_i)
			 			\bbU_2^\top\bbW(\bbP_y-l_i\bbI_n)
			 			\right) 
			 			\boldsymbol{\Gamma}_k\boldsymbol{\Gamma}_k^\top.
			 			\]
			 		\fi
			 			
			 			Under the Gaussian assumption on the entries of $\bbW$ and the orthogonality of
			 			$\bbU_1$ and $\bbU_2$, it follows that $\bbU_1^\top \bbW$ and $\bbU_2^\top \bbW$
			 			are independent. Moreover, by the proof of Theorem~1.1 in Bai and Silverstein (2004)
			 			\cite{BS04CLT}, we have $\mathcal{V}_{k1}=O_p(n^{-1/2})$. Throughout the proof, we adopt the standard normalization that the entries of $\bbW$ are i.i.d. $\mathcal{N}(0,1/n)$.
			 			
			 			Let $\Phi(z)=\left( \bbU_2^\top \bbW(\bbP_y\!-\!z\bbI_n)\bbW^\top \bbU_2\right) ^{-1}$.
			 			Note that
			 			$n\bbE=n\bbU_2^\top \bbW\bbP_y\bbW^\top \bbU_2\sim\mathcal{W}_{p-k}(\bbI_{p-k},q)$
			 			and
			 			$n\bbH=n\bbU_2^\top \bbW(\bbI-\bbP_y)\bbW^\top \bbU_2\sim\mathcal{W}_{p-k}(\bbI_{p-k},n-q)$.
			 			Hence $\Phi(z)=\left( (1-z)\bbE-z\bbH\right) ^{-1}$. Consequently,
			 			\begin{align*}
			 				&\frac{1}{n}\text{tr}(\bbP_y\!-\!l_i)\bbW^\top \bbU_2\left( \bbU_2^\top \bbW(\bbP_y\!-\!l_i)\bbW^\top \bbU_2\right) ^{-2}\bbU_2^\top \bbW(\bbP_y\!-\!l_i)\\
			 				=&(1\!-\!l_i)^2\frac{1}{n}\text{tr} \bbE\Phi^2(l_i)\!+\!l_i^2\frac{1}{n}\text{tr}\bbH\Phi^2(l_i).
			 			\end{align*}
			 			Differentiating with respect to $z$ yields
			 			\begin{equation*}
			 				\left\{
			 				\begin{aligned}
			 					\left( \frac{1}{n}\text{tr}\Phi(l_i)\right) '&=\frac{1}{n}\text{tr} \Phi(l_i)(\bbE+\bbH)\Phi(l_i),\\
			 					\frac{1}{n}\text{tr}\Phi(l_i)	&=\frac{1-l_i}{n}\text{tr}\bbE\Phi^2-\frac{l_i}{n}\text{tr}\bbH\Phi^2(l_i),
			 				\end{aligned}
			 				\right.
			 			\end{equation*}
			 			that is,
			 			\begin{equation*}
			 				\left\{
			 				\begin{aligned}
			 					\frac{1}{n}\text{tr}\bbE\Phi^2(l_i)&=\frac{1}{n}\text{tr} \Phi(l_i)+l_i\left( \frac{1}{n}\text{tr} \Phi(l_i)\right) ',\\
			 					\frac{1}{n}\text{tr}\bbH\Phi^2(l_i)&=(1-l_i)\left( \frac{1}{n}\text{tr}\Phi(l_i)\right) '-\frac{1}{n}\text{tr}\Phi(l_i).
			 				\end{aligned}
			 				\right.
			 			\end{equation*}
			 			Therefore, it suffices to estimate the two key quantities
			 			$n^{-1}\text{tr}\Phi(l_i)$ and $\bigl(n^{-1}\text{tr}\Phi(l_i)\bigr)'$.
			 			
			 			Let $\mathcal{X}_{(p-k)\times q}$ and $\mathcal{Y}_{(p-k)\times(n-q)}$ be independent
			 			matrices with i.i.d.\ standard normal entries. Since
			 			\[
			 			\Phi(z)=\left((1-z)\bbE-z\bbH\right)^{-1}
			 			=
			 			n\left((1-z)n\bbE-zn\bbH\right)^{-1},
			 			\]
			 			we have
			 			\[
			 			\frac{1}{n}\tr\Phi(z)
			 			=
			 			\frac{1}{1-z}
			 			\tr\left(
			 			n\bbE-\frac{z}{1-z}n\bbH
			 			\right)^{-1}.
			 			\]
			 			Therefore,
			 			\[
			 			\frac{1}{n}\tr\Phi(z)
			 			\overset{d}{=}
			 			\frac{1}{1-z}
			 			\tr\left(
			 			\mathcal X\mathcal X^\top
			 			-
			 			\frac{z}{1-z}\mathcal Y\mathcal Y^\top
			 			\right)^{-1}.
			 			\]
			 			Applying the Poincar\'e inequality in Lemma~\ref{lemma-pi}, we obtain the variance bound
			 			\begin{align}\label{pi01}
			 				&\operatorname{Var}(\text{tr}\Phi(z))\notag\\
			 				\leq &\mathbb{E}\frac{1}{n}\sum_{i=1}^{p-k}\sum_{j=1}^{n-q}\left[\frac{\partial \text{tr}\Phi(z)}{\partial \mathcal{y}_{ij}}\right]^2
			 				+\mathbb{E}\frac{1}{n}\sum_{i=1}^{p-k}\sum_{j=1}^q\left[\frac{\partial \text{tr}\Phi(z)}{\partial \mathcal{x}_{ij}}\right]^2
			 				=O\left( \frac{p+q}{n}\right) .
			 			\end{align}
			 			The Poincaré inequality gives
			 			\[
			 			\operatorname{Var}(\tr\Phi(z))=O(1),
			 			\]
			 			uniformly for \(z\) in a compact set away from the limiting spectral support. Hence
			 			\[
			 			\frac{1}{n}\tr\Phi(z)
			 			-
			 			\mathbb E\frac{1}{n}\tr\Phi(z)
			 			=
			 			O_p(n^{-1}).
			 			\]
			 			Combining this with the cumulant expansion formula in Lemma~\ref{lemma-cumulant},
			 			we obtain
			 			\begin{align*}
			 				&\mathbb{E}\text{tr}(\mathcal{X}\mathcal{X}^\top (\mathcal{Y}\mathcal{Y}^\top )^{-1}-z\bbI_{p-k})^{-1}\\[0.5em]
			 				=&\mathbb{E}\text{tr} \mathcal{Y}\mathcal{Y}^\top (\mathcal{X}\mathcal{X}^\top -z\mathcal{Y}\mathcal{Y}^\top )^{-1}\\[0.5em]
			 				=&\mathbb{E}\frac{1}{n}
			 				\sum_{a=1}^{p-k}\sum_{b=1}^{n-q}
			 				y_{ab}
			 				\left[
			 				\mathcal Y^\top
			 				(\mathcal X\mathcal X^\top-z\mathcal Y\mathcal Y^\top)^{-1}
			 				\right]_{ba}\\[0.5em]
			 				=&\frac{n-q}{n}\mathbb{E}\text{tr}(\mathcal{X}\mathcal{X}^\top \!-\!z\mathcal{Y}\mathcal{Y}^\top )^{-1}
			 				\!+\!\frac{z}{n}\mathbb{E}\text{tr}(\mathcal{X}\mathcal{X}^\top (\mathcal{Y}\mathcal{Y}^\top )^{-1}\!-\!z\bbI_{p-k})^{-1}\text{tr}(\mathcal{X}\mathcal{X}^\top \!-\!z\mathcal{Y}\mathcal{Y}^\top )^{-1}\!+\!O(1)\\[0.5em]
			 				=&\frac{n-q}{n}\mathbb{E}\text{tr}(\mathcal{X}\mathcal{X}^\top \!\!-\!\!z\mathcal{Y}\mathcal{Y}^\top )^{-1}
			 				\!+\!\frac{z}{n}\mathbb{E}\text{tr}(\mathcal{X}\mathcal{X}^\top (\mathcal{Y}\mathcal{Y}^\top )^{-1}\!\!-\!z\bbI_{p-k})^{-1}\mathbb{E}\text{tr}(\mathcal{X}\mathcal{X}^\top \!\!-\!\!z\mathcal{Y}\mathcal{Y}^\top )^{-1}\!+\!O(1),
			 			\end{align*}
			 			where the last equality follows from the Cauchy--Schwarz inequality.
			 			Consequently,
			 			\begin{align*}
			 				&\left| \mathbb{E}\frac{1}{n}\text{tr}(\mathcal{X}\mathcal{X}^\top (\mathcal{Y}\mathcal{Y}^\top )^{-1}-z\bbI)^{-1}\text{tr}(\mathcal{X}\mathcal{X}^\top -z\mathcal{Y}\mathcal{Y}^\top )^{-1}\right.\\[0.5em]
			 				&~~~~~~~~~~~~~~~~~~~~~~~~~~\left.-\mathbb{E}\frac{1}{n}\text{tr}(\mathcal{X}\mathcal{X}^\top (\mathcal{Y}\mathcal{Y}^\top )^{-1}-z\bbI)^{-1}\mathbb{E}\text{tr}(\mathcal{X}\mathcal{X}^\top -z\mathcal{Y}\mathcal{Y}^\top )^{-1}\right| \\[0.5em]
			 				=&\left| \mathbb{E}(1-\mathbb{E})\frac{1}{n}\text{tr}(\mathcal{X}\mathcal{X}^\top (\mathcal{Y}\mathcal{Y}^\top )^{-1}-z\bbI)^{-1}(1-\mathbb{E})\text{tr}(\mathcal{X}\mathcal{X}^\top -z\mathcal{Y}\mathcal{Y}^\top )^{-1}\right| \\[0.5em]
			 				\leq&\sqrt{\mathbb{E}\left| (1-\mathbb{E})\frac{1}{n}\text{tr}(\mathcal{X}\mathcal{X}^\top (\mathcal{Y}\mathcal{Y}^\top )^{-1}-z\bbI)^{-1}\right| ^2}
			 				\sqrt{\mathbb{E}\left| (1-\mathbb{E})\text{tr}(\mathcal{X}\mathcal{X}^\top -z\mathcal{Y}\mathcal{Y}^\top )^{-1}\right| ^2}\\[0.5em]
			 				\leq &O\left( \frac{1}{n}\right) .
			 			\end{align*}
			 			Here the replacement of $\mathbb{E}(\tr \bbA\,\tr \bbB)$ by $\mathbb{E}\tr \bbA\,\mathbb{E}\tr \bbB$ introduces an $O(1)$ error,
			 			since by the Poincar\'e inequality both $\operatorname{Var}(\tr \bbA)$ and $\operatorname{Var}(\tr B)$ are uniformly bounded, and hence the covariance is $O(1)$.
			 			After dividing by $n$, this error becomes $O(1/n)$.
			 			
			 			It follows that
			 			\begin{align*}
			 				\mathbb{E}\text{tr}(\mathcal{X}\mathcal{X}^\top -z\mathcal{Y}\mathcal{Y}^\top )^{-1}
			 				=\frac{\mathbb{E}\text{tr}(\mathcal{X}\mathcal{X}^\top (\mathcal{Y}\mathcal{Y}^\top )^{-1}-z\bbI)^{-1}}
			 				{\frac{n-q}{n}+\frac{z}{n}\mathbb{E}\text{tr}(\mathcal{X}\mathcal{X}^\top (\mathcal{Y}\mathcal{Y}^\top )^{-1}-z\bbI)^{-1}}+O(1),
			 			\end{align*}
			 			and hence
			 			\begin{align}
			 				\frac{1}{n}\text{tr}\Phi(z)=&\frac{1}{1-z}\frac{\frac{1}{n}\mathbb{E}\text{tr}\left( \mathcal{X}\mathcal{X}^\top (\mathcal{Y}\mathcal{Y}^\top )^{-1}-\frac{z}{1-z}\bbI\right)^{-1} }{\frac{n-q}{n}+\frac{z}{1-z}\frac{1}{n}\mathbb{E}\text{tr}\left( \mathcal{X}\mathcal{X}^\top (\mathcal{Y}\mathcal{Y}^\top )^{-1}-\frac{z}{1-z}\bbI\right) ^{-1}}+O_{p}\left( \frac{1}{n}\right) \nonumber\\[0.5em]
			 				=&\frac{1}{z}\frac{-c_1+(1-z)s(z)}{1-c_2-c_1+(1-z)s(z)}+O_{p}\left( \frac{1}{n}\right) \nonumber\\[0.5em]
			 				:=&QQ_1(z)+O_p\left( \frac{1}{n}\right).\label{phi}
			 			\end{align}
			 			Since $n^{-1}\tr\Phi(z)$ is analytic on $\mathbb{C}\setminus\mathbb{R}$ and converges to its deterministic limit
			 			uniformly on compact subsets of a neighborhood of $l_i$ (by standard resolvent convergence results for the Fisher-type matrix),
			 			Cauchy's integral formula implies that the derivatives also converge, and thus the limit and differentiation can be interchanged:
			 			\begin{align*}
			 				\lim_{n\to\infty}\frac{\partial}{\partial z}\mathbb{E}\frac{1}{n}\text{tr}\Phi(z)
			 				=\frac{\partial}{\partial z}\lim_{n\to\infty}\mathbb{E}\frac{1}{n}\text{tr}\Phi(z).
			 			\end{align*}
			 			Therefore,
			 			\begin{align}
			 				\left( \frac{1}{n}\text{tr}\Phi(z)\right) '=&\frac{d}{dz}\left( \frac{1}{z}\frac{-c_1+(1-z)s(z)}{1-c_2-c_1+(1-z)s(z)}\right) +O_p\left( \frac{1}{\sqrt{n}}\right) \nonumber\\[0.5em]
			 				=&\frac{1}{z^2}\frac{c_1-(1-z)s(z)}{1-c_1-c_2+(1-z)s(z)}
			 				\!+\!\frac{(1-c_2)(-s(z)+(1-z)s'(z))}{z(1-c_1-c_2+(1-z)s(z))^2}
			 				\!+\!O_p\left( \frac{1}{\sqrt{n}}\right) \nonumber\\[0.5em]
			 				:=&QQ_2(z)+O_p\left( \frac{1}{\sqrt{n}}\right).\label{phi'}
			 			\end{align}
			 			The above resolvent estimates hold uniformly for \(z\) in a fixed neighborhood of \(\gamma_i\) that stays away from the limiting spectral support. Since \(r_i\) is spike, $l_i-\gamma_i=O_p(n^{-1/2})$.
			 			Therefore, by the resolvent identity and Taylor expansion of the deterministic limits, these estimates may be evaluated at the random point \(z=l_i\), with an additional error of order \(O_p(n^{-1/2})\). Consequently, combining
			 			\eqref{phi} and \eqref{phi'} yields
			 			\begin{align}
			 				QQ_1(l_i)=&\frac{-c_1r_i((1-c_2)r_i+c_2)}{r_i(\eta r_i-\omega)-(\eta r_i-\omega)^2(1-r_i)}+O_p\left( \frac{1}{\sqrt{n}}\right)\notag\\[0.5em]
			 				:=&F_1(r_i)+O_p\left( \frac{1}{\sqrt{n}}\right), \label{QQ1}\\[0.5em]
			 				QQ_2(l_i)=& \frac{
			 					r_i^2
			 					\left(
			 					\omega(1-c_2)r_i\big((1-c_1)r_i+c_1\big)
			 					+
			 					c_1(\eta r_i-\omega)(\eta r_i^2-\omega)
			 					\right)
			 				}
			 				{
			 					\big((1-c_1)r_i+c_1\big)^2
			 					\big((1-c_2)r_i+c_2\big)
			 					(\eta r_i-\omega)^2
			 					(\eta r_i^2-\omega)}
			 					+O_p\left( \frac{1}{\sqrt{n}}\right)\notag\\[0.5em]
			 				:=&F_2(r_i)+O_p\left( \frac{1}{\sqrt{n}}\right), \label{QQ2}\\[0.5em]
			 				\frac{1}{n}\text{tr}\bbE\Phi^2(l_i)=&\frac{\omega(1-c_2)r_i^2}{(\eta r_i-\omega)^2(\eta r_i^2-\omega)}+O_p\left( \frac{1}{\sqrt{n}}\right) ,\label{EP2}\\[0.5em]
			 				\frac{1}{n}\text{tr}\bbH\Phi^2(l_i)=&\frac{c_1r_i^2}{(\eta r_i^2-\omega)((1-c_1)r_i+c_1)^2}+O_p\left( \frac{1}{\sqrt{n}}\right) .\label{HP2}
			 			\end{align}
			 			Here $\eta=(1-c_1)(1-c_2)$ and $\omega=c_1c_2$.
			 			Substituting these expansions into $I_{k1}$ yields
			 			\begin{align}
			 				I_{k1}=&\left( (1\!-\!l_i)^2\frac{1}{n}\text{tr} \bbE\Phi^2(l_i)\!+\!l_i^2\frac{1}{n}\text{tr}\bbH\Phi^2(l_i)\right) \boldsymbol{\Gamma}_k\boldsymbol{\Gamma}_k^\top +\mathcal{V}_{k1}\notag\\
			 				=&\left( (1-2l_i)\frac{1}{n}\text{tr}\Phi(l_i)+l_i(1-l_i)\left( \frac{1}{n}\text{tr}\Phi(l_i)\right)'\right) \boldsymbol{\Gamma}_k\boldsymbol{\Gamma}_k^\top +\mathcal{V}_{k1}\notag\\
			 				=&\left( (1-2l_i)QQ_1(l_i)+l_i(1-l_i)QQ_2(l_i)\right) \times\boldsymbol{\Gamma}_k\boldsymbol{\Gamma}_k^\top +O_p\left( \frac{1}{\sqrt{n}}\right).\label{rIk1}
			 			\end{align}
			 			
			 			We next analyze $I_{k2}$. Let $\bbY=\bbU_y\boldsymbol{\Lambda}_y\bbV_y^\top$.
			 			Following the treatment in Bao et al. (2019) \cite{bao2019canonical}, since $\bbY$ has i.i.d. Gaussian entries, in its SVD the left singular vectors $\bbU_y$ are Haar distributed on the orthogonal group and are independent of $(\boldsymbol{\Lambda}_y,\bbV_y)$. 
			 			Let $\tilde{\bbU}_y$ be an independent Haar orthogonal matrix with the same dimension as $\bbU_y$. 
			 			Then replacing $\bbU_y$ by $\tilde{\bbU}_y$ does not change the distribution of $I_{k2},I_{k3},I_{k4}$. Accordingly, $I_{k2}$ can be decomposed as
			 			\begin{align}
			 				I_{k2}=&\boldsymbol{\Lambda}_k\bbV_1^\top \tilde{\bbU}_y\boldsymbol{\Lambda}_y\bbV_y^\top (\bbP_y\!-\!l_i\bbI_n)\bbW^\top \bbU_2(\bbU_2^\top \bbW(\bbP_y\!-\!l_i\bbI_n)\bbW^\top \bbU_2)^{-2}\notag\\
			 				&\times \bbU_2^\top \bbW(\bbP_y\!-\!l_i\bbI_n)\bbV_y\boldsymbol{\Lambda}_y\tilde{\bbU}_y^\top \bbV_1\boldsymbol{\Lambda}_k+O_p\left( \frac{\log n}{n}\right) \nonumber\\
			 				=&\frac{1}{q}\text{tr}\boldsymbol{\Lambda}_y\bbV_y^\top (\bbP_y\!-\!l_i\bbI_n)\bbW^\top\bbU_2(\bbU_2^\top \bbW(\bbP_y\!-\!l_i\bbI_n)\bbW^\top \bbU_2)^{-2}\notag\\
			 				&\times \bbU_2^\top \bbW(\bbP_y\!-\!l_i\bbI_n)\bbV_y\boldsymbol{\Lambda}_y\times\boldsymbol{\Lambda}_k\boldsymbol{\Lambda}_k^\top +O_p\left( \frac{\log n}{n}\right) \nonumber\\
			 				:=&(1-l_i)^2\frac{1}{q}\text{tr}\bbY\bbW^\top \bbU_2\Phi(l_i)^{2}\bbU_2^\top \bbW\bbY^\top\times\boldsymbol{\Lambda}_k\boldsymbol{\Lambda}_k^\top +\mathcal{V}_{k2}+O_p\left( \frac{\log n}{n}\right) .\label{Vk2}
			 			\end{align}
			 			By the Poincar\'e inequality, the variance of
			 			the normalized trace $q^{-1}\text{tr}\bbY\bbW^\top \bbU_2\Phi(l_i)^{2}\bbU_2^\top \bbW\bbY^\top$ is uniformly bounded,
			 			which is analogous to \eqref{pi01}. Using the cumulant expansion, we obtain
			 			\begin{align*}
			 				&\mathbb{E}\text{tr}\bbY\bbW^\top \bbU_2\Phi(l_i)^{2}\bbU_2^\top \bbW\bbY^\top\\
			 				=&\mathbb{E}\frac{q}{n}\text{tr}(\bbE+\bbH)\Phi^2(l_i)\!-\!\mathbb{E}\frac{1}{n}\left(\text{tr}\bbH\Phi(l_i)\text{tr}\bbE\Phi^2(l_i)\!+\!\text{tr}\bbE\Phi(l_i)\text{tr}\bbH\Phi^2(l_i)\right)\!+\!O\left(\frac{q}{n}\right),
			 			\end{align*}
			 			and hence
			 			\begin{align*}
			 				&\mathbb{E}\frac{1}{n}\text{tr}\bbY\bbW^\top \bbU_2\Phi(l_i)^{2}\bbU_2^\top \bbW\bbY^\top\\
			 				=&\mathbb{E}\frac{q}{n}\frac{1}{n}\text{tr}(\bbE+\bbH)\Phi^2(l_i)
			 				\!-\!\mathbb{E}\frac{1}{n}\text{tr}\bbH\Phi(l_i)\frac{1}{n}\text{tr}\bbE\Phi^2(l_i)
			 				\!-\!\mathbb{E}\frac{1}{n}\text{tr}\bbE\Phi(l_i)\frac{1}{n}\text{tr}\bbH\Phi^2(l_i)
			 				\!+\!O\left(\frac{q}{n^2}\right)\\
			 				=&\mathbb{E}\frac{q}{n}\frac{1}{n}\text{tr}(\bbE+\bbH)\Phi^2(l_i)
			 				\!-\!\mathbb{E}\frac{1}{n}\text{tr}\bbH\Phi(l_i)\mathbb{E}\frac{1}{n}\text{tr}\bbE\Phi^2(l_i)
			 				\!-\!\mathbb{E}\frac{1}{n}\text{tr}\bbE\Phi(l_i)\mathbb{E}\frac{1}{n}\text{tr}\bbH\Phi^2(l_i)
			 				\!+\!O\left(\frac{q}{n^2}\right).
			 			\end{align*}
			 			By the same argument as for $I_{k1}$, we further obtain
			 			\begin{align*}
			 				&\frac{1}{n}\text{tr}\bbY\bbW^\top \bbU_2\Phi(l_i)^{2}\bbU_2^\top \bbW\bbY^\top\\[0.5em]
			 				=&\left(\frac{c_1}{l_i}-\frac{1-l_i}{l_i}s(l_i)+s(l_i)\right)\frac{1}{n}\text{tr}\Phi(l_i)\\[0.5em]
			 				&+\left( c_1+c_2-2(1-l_i)s(l_i)\right) \left(\frac{1}{n}\text{tr}\Phi(l_i)\right)'+O_p\left(\frac{1}{\sqrt{n}}\right)\\[0.5em]
			 				=&\left( \frac{c_1}{l_i}-\frac{1-l_i}{l_i}s(l_i)+s(l_i)\right) QQ_1(l_i)
			 				+\left( c_1+c_2-2(1-l_i)s(l_i)\right) QQ_2(l_i)+O_p\left(\frac{1}{\sqrt{n}}\right).
			 			\end{align*}
			 			Therefore, the asymptotic expansion of $I_{k2}$ is
			 			\begin{align}
			 				I_{k2}=&\frac{(1-l_i)^2}{c_2}\left( \frac{c_1-(1-2l_i)s(l_i)}{l_i} QQ_1(l_i)
			 				+\left( c_1+c_2-2(1-l_i)s(l_i)\right) QQ_2(l_i)\right)\notag\\
			 				&\times\boldsymbol{\Lambda}_k\boldsymbol{\Lambda}_k^\top  +O_p\left(\frac{1}{\sqrt{n}}\right).
			 			\end{align}
			 			
			 			For the cross terms $I_{k3}$ and $I_{k4}$, conditional on $\bbU_2^\top \bbW$ and $\bbY$,
			 			they can be viewed as linear combinations of Gaussian random variables.
			 			Using the boundedness of the operator norm,
			 			\begin{align*}
			 				\left\| \bbV_1^\top \bbY(\bbP_y\!-\!l_i)\bbW^\top \bbU_2\Phi(l_i)^{2}\bbU_2^\top \bbW(\bbP_y\!-\!l_i)^2\bbW^\top \bbU_2\Phi(l_i)^{2}\bbU_2^\top \bbW(\bbP_y\!-\!l_i)\bbY^\top \bbV_1\right\| =O(1),
			 			\end{align*}
			 			Since \(k\) is fixed, this gives the entrywise bounds
			 			\[
			 			[I_{k3}]_{ab}=O_p(n^{-1/2}),
			 			\qquad
			 			[I_{k4}]_{ab}=O_p(n^{-1/2}),
			 			\qquad 1\le a,b\le k.
			 			\]
			 			In particular,
			 			\[
			 			[I_{k3}]_{ii}=O_p(n^{-1/2}),
			 			\qquad
			 			[I_{k4}]_{ii}=O_p(n^{-1/2}).
			 			\]
			 			%we conclude that $I_{k3}=O_p(n^{-1/2})$ and $I_{k4}=O_p(n^{-1/2})$.
			 			
			 			It remains to control the off-diagonal entries. We explain the argument for \(I_{k1}\), the other terms are handled in the same way. Conditional on \(\bbY\) and \(\bbU_2^\top\bbW\), the rows of \(\bbU_1^\top\bbW\) are independent centered Gaussian vectors and are independent of \(\bbU_2^\top\bbW\). Hence, for \(a\ne b\),
			 			\[
			 			\mathbb E\left([I_{k1}]_{ab}\mid \bbY,\bbU_2^\top\bbW\right)=0.
			 			\]
			 			Moreover, by the resolvent bound \(\|\bbD^{-1}\|=O_p(1)\) and the fact that \(k\) is fixed, the corresponding conditional variance satisfies
			 			\[
			 			\operatorname{Var}\left([I_{k1}]_{ab}\mid \bbY,\bbU_2^\top\bbW\right)=O_p(n^{-1}).
			 			\]
			 			Therefore,
			 			\[
			 			[I_{k1}]_{ab}=O_p(n^{-1/2}),\qquad a\ne b.
			 			\]
			 			
			 			For \(I_{k2}\), after the Haar replacement of the left singular vectors of \(\bbY\), the off-diagonal entries have conditional mean zero and conditional variance \(O_p(n^{-1})\). Thus
			 			\[
			 			[I_{k2}]_{ab}=O_p(n^{-1/2}),\qquad a\ne b.
			 			\]
			 			Similarly, the cross terms \(I_{k3}\) and \(I_{k4}\) are conditionally centered Gaussian bilinear forms with conditional variances of order \(n^{-1}\). Hence
			 			\[
			 			[I_{k3}]_{ab}=O_p(n^{-1/2}),
			 			\qquad
			 			[I_{k4}]_{ab}=O_p(n^{-1/2}),
			 			\qquad a\ne b.
			 			\]
			 			Combining these bounds gives
			 			\[
			 			[\bbM_i]_{ab}=O_p(n^{-1/2}),\qquad a\ne b.
			 			\]
			 			Since \(k\) is fixed and the diagonal entries of \(\bbM_i\) are \(O_p(1)\), it follows that
			 			\[
			 			\|\bbM_i\|=O_p(1).
			 			\]
			 			This completes the proof of Lemma~\ref{ccalemma1}.
			 		\end{proof}

		 			\begin{proof}[Proof of Lemma \ref{ccalemma2}] 
		 			We first record the fluctuation of the sample spike \(l_i\).
		 			From the \(i\)-th equation of
		 			\[
		 			M_{\Lambda}(l_i)\boldsymbol{\varpi}_{i1}=0,
		 			\]
		 			we have
		 			\[
		 			(M_{\Lambda})_{ii}(l_i)\boldsymbol{\varpi}_{i1,i}
		 			=
		 			-
		 			(M_{\Lambda}(l_i))_{i,-i}
		 			\boldsymbol{\varpi}_{i1,-i}.
		 			\]
		 			By the estimates obtained in the Schur complement reduction,
		 			\[
		 			(M_{\Lambda}(l_i))_{i,-i}=O_p(n^{-1/2}),
		 			\qquad
		 			\boldsymbol{\varpi}_{i1,-i}
		 			=
		 			O_p(n^{-1/2})\boldsymbol{\varpi}_{i1,i}.
		 			\]
		 			Hence
		 			\[
		 			(M_{\Lambda})_{ii}(l_i)=O_p(n^{-1}).
		 			\]
		 			Taking \(\alpha=\beta=i\) in \eqref{eq:Mlambda-expansion}
		 			and using \(M_{ii}(\gamma_i)=0\), we obtain
		 			\[
		 			\sqrt n(l_i-\gamma_i)
		 			=
		 			\frac{\mathcal R_{ii}}{\Delta_{ii}}
		 			+o_p(1).
		 			\]
		 			Using the expansions obtained in the proof of Lemma
		 			\ref{ccalemma1}, we have
		 			\[
		 			\begin{aligned}
		 				[I_{k1}]_{ii}
		 				={}&
		 				(1-r_i)
		 				\left[
		 				(1-2l_i)QQ_1(l_i)
		 				+l_i(1-l_i)QQ_2(l_i)
		 				\right]
		 				\\
		 				&+
		 				[\mathcal V_{k1}]_{ii}
		 				+o_p(n^{-1/2}),
		 			\end{aligned}
		 			\]
		 			and
		 			\[
		 			\begin{aligned}
		 				[I_{k2}]_{ii}
		 				={}&
		 				\frac{r_i(1-l_i)^2}{c_2}
		 				\Bigg[
		 				\frac{c_1-(1-2l_i)s(l_i)}{l_i}QQ_1(l_i)
		 				\\
		 				&\qquad+
		 				\left(
		 				c_1+c_2-2(1-l_i)s(l_i)
		 				\right)QQ_2(l_i)
		 				\Bigg]
		 				\\
		 				&+
		 				[\mathcal V_{k2}]_{ii}
		 				+o_p(n^{-1/2}).
		 			\end{aligned}
		 			\]
		 			Combining the preceding expansions of $[I_{k1}]_{ii}$ and
		 			$[I_{k2}]_{ii}$, and recalling that their deterministic parts
		 			evaluated at $z=\gamma_i$ give $d(r_i)$, we expand the corresponding
		 			deterministic terms around $z=\gamma_i$. Since
		 			\[
		 			l_i-\gamma_i=O_p(n^{-1/2}),
		 			\]
		 			and the required derivatives are uniformly bounded in a fixed
		 			neighborhood of $\gamma_i$, the second-order remainder is
		 			$O_p(n^{-1})$ and hence is negligible after multiplication by
		 			$\sqrt n$. Moreover,
		 			\[
		 			\begin{aligned}
		 				&\left.
		 				\frac{d}{dz}
		 				\Bigg(
		 				(1-r_i)
		 				\left( 
		 				(1-2z)QQ_1(z)+z(1-z)QQ_2(z)
		 				\right) 
		 				\right.
		 				\\
		 				&\left.
		 				\qquad+
		 				\frac{r_i(1-z)^2}{c_2}
		 				\left( 
		 				\frac{c_1-(1-2z)s(z)}{z}QQ_1(z)
		 				+
		 				\left(
		 				c_1+c_2-2(1-z)s(z)
		 				\right)QQ_2(z)
		 				\right) 
		 				\Bigg)
		 				\right|_{z=\gamma_i}
		 				\\
		 				&=
		 				-
		 				\frac{
		 					2\left( c_2(1-r_i)+r_i(1-\gamma_i)\right) 
		 				}{
		 					c_2\gamma_i
		 				}
		 				\left( 
		 				F_1(r_i)
		 				+\gamma_iF_2(r_i)
		 				-(1-\gamma_i)Q_5(r_i)
		 				\right) .
		 			\end{aligned}
		 			\]
		 			Here we have used
		 			\[
		 			QQ_1(\gamma_i)=F_1(r_i),
		 			\qquad
		 			QQ_2(\gamma_i)=F_2(r_i),
		 			\]
		 			together with
		 			\[
		 			Q_5(r_i)
		 			=
		 			F_1(r_i)
		 			+
		 			2\gamma_iF_2(r_i)
		 			+
		 			\frac{\gamma_i^2}{2}QQ_2'(\gamma_i).
		 			\]
		 			Therefore, the first-order Taylor expansion yields
		 			\[
		 			\begin{aligned}
		 				G_{i,n}
		 				={}&
		 				X_{1n}+X_{2n}+X_{3n}
		 				\\
		 				&-
		 				\frac{
		 					2\left( c_2(1-r_i)+r_i(1-\gamma_i)\right) 
		 				}{
		 					c_2\gamma_i
		 				}
		 				\left( 
		 				F_1(r_i)
		 				+\gamma_iF_2(r_i)
		 				-(1-\gamma_i)Q_5(r_i)
		 				\right) 
		 				\sqrt n(l_i-\gamma_i)
		 				+o_p(1).
		 			\end{aligned}
		 			\]
		 			
		 			On the other hand, as shown above from the Schur complement equation
		 			\[
		 			M_{\Lambda}(l_i)\boldsymbol{\varpi}_{i1}=0
		 			\]
		 			and the expansion \eqref{eq:Mlambda-expansion}, we have
		 			\[
		 			\sqrt n(l_i-\gamma_i)
		 			=
		 			\frac{\mathcal R_{ii}}{\Delta_{ii}}
		 			+o_p(1).
		 			\]
		 			Choose a compact interval $\mathcal N_i$ containing $\gamma_i$ and separated by a fixed positive distance from the limiting null spectrum. On the spectral-separation event $\Omega_n$, the resolvent $\Phi(z)$ and its first two derivatives are uniformly bounded for $z\in\mathcal N_i$. Differentiating the centered Gaussian and Haar quadratic forms and using the corresponding conditional second-moment formulas give
		 			\[
		 			\sup_{z\in\mathcal N_i}
		 			\left|\partial_z[\mathcal V_{k1}(z)]_{ii}\right|
		 			=O_p(n^{-1/2}),
		 			\qquad
		 			\sup_{z\in\mathcal N_i}
		 			\left|\partial_z[\mathcal V_{k2}(z)]_{ii}\right|
		 			=O_p(n^{-1/2}).
		 			\]
		 			Since $l_i-\gamma_i=O_p(n^{-1/2})$, the mean-value theorem yields
		 			\[
		 			[\mathcal V_{k1}(l_i)]_{ii}-[\mathcal V_{k1}(\gamma_i)]_{ii}
		 			=O_p(n^{-1})=o_p(n^{-1/2}),
		 			\]
		 			and the same holds for $\mathcal V_{k2}$ and for the centered cross term. Consequently, all stochastic $n^{-1/2}$-order fluctuations may be evaluated at $z=\gamma_i$; only the derivative of the deterministic equivalent contributes a first-order term proportional to $l_i-\gamma_i$.
		 			Substituting this relation into the preceding expansion, we obtain
		 			\[
		 			\begin{aligned}
		 				G_{i,n}
		 				={}&
		 				X_{1n}+X_{2n}+X_{3n}
		 				\\
		 				&-
		 				\frac{
		 					2\left( c_2(1-r_i)+r_i(1-\gamma_i)\right) 
		 				}{
		 					c_2\gamma_i
		 				}
		 				\left( 
		 				F_1(r_i)
		 				+\gamma_iF_2(r_i)
		 				-(1-\gamma_i)Q_5(r_i)
		 				\right) 
		 				\frac{\mathcal R_{ii}}{\Delta_{ii}}
		 				+o_p(1).
		 			\end{aligned}
		 			\]
		 			This gives the required stochastic expansion of $G_{i,n}$.
		 			
		 			For arbitrary fixed $a_1,a_2,a_3,a_4\in\mathbb R$, consider
		 			\[
		 			T_n=a_1X_{1n}+a_2X_{2n}+a_3X_{3n}+a_4\mathcal R_{ii}.
		 			\]
		 			Conditional on $\mathcal F_i$, the Gaussian contributions to $T_n$ are finite sums of centered Gaussian quadratic and bilinear forms with coefficient matrices whose operator norms are bounded in probability on the spectral-separation event introduced below. A bilinear form $\bbg_1^\top \bbB\bbg_2$ can be written as a centered quadratic form in $(\bbg_1^\top,\bbg_2^\top)^\top$ with the symmetric block matrix
		 			\[
		 			\frac12\begin{pmatrix}\mathbf{0}&\bbB\\ \bbB^\top&\mathbf{0}\end{pmatrix}.
		 			\]
		 			For $\bbg\sim \mathcal{N}(0,n^{-1}\bbI_m)$ and a conditionally fixed symmetric matrix $\bbA$, the conditional cumulants satisfy
		 			\[
		 			\kappa_r\!\left(\bbg^\top \bbA\bbg-\frac1n\operatorname{tr}\bbA\mid\mathcal F_i\right)
		 			=2^{r-1}(r-1)!\frac{\operatorname{tr}\bbA^r}{n^r},
		 			\qquad r\ge2.
		 			\]
		 			Since $m=O(n)$ and $\|\bbA\|=O_p(1)$, after multiplication by $\sqrt n$ every conditional cumulant of order $r\ge3$ is $O_p(n^{1-r/2})=o_p(1)$, whereas the second cumulants converge to the covariance limits computed below.
		 			
		 			For the Haar component, if $\mathbf{u}$ is a row of a Haar orthogonal matrix, use the Gaussian representation $\mathbf{u}=\bbg/\|\bbg\|$ with $\bbg\sim \mathcal{N}(0,\bbI_q)$. Uniformly for coefficient matrices with bounded operator norm,
		 			\[
		 			\mathbf{u}^\top \bbA\mathbf{u}-\frac1q\operatorname{tr}\bbA
		 			=\frac1q\bbg^\top\left(\bbA-\frac{\operatorname{tr}\bbA}{q}\bbI_q\right)\bbg
		 			+o_p(q^{-1/2}).
		 			\]
		 			Hence, at the $\sqrt n$ scale, the centered Haar quadratic form has the same Gaussian-quadratic representation and the same vanishing higher-cumulant property.
		 			
		 			Finally, the outlier fluctuation $\mathcal R_{ii}$ entering the expansion of $l_i-\gamma_i$ is represented by the same underlying Gaussian blocks and can be included in the enlarged quadratic-form vector. Therefore the conditional covariance matrix of $(X_{1n},X_{2n},X_{3n},\mathcal R_{ii})$ converges in probability to the deterministic matrix calculated below, while all conditional cumulants of order at least three vanish. The conditional characteristic function of every Cramér--Wold combination $T_n$ converges in probability to that of a centered Gaussian random variable. Taking expectations and applying the Cramér--Wold device proves the asserted joint Gaussian convergence.
		 			\iffalse
		 			Then we justify the joint Gaussian convergence of the fluctuation terms. By the Cramér--Wold device, it is enough to consider an arbitrary fixed linear combination
		 			\[
		 			T_n
		 			=
		 			a_1X_{1n}
		 			+a_2X_{2n}
		 			+a_3X_{3n}
		 			+a_4\mathcal R_{ii},
		 			\qquad
		 			a_1,a_2,a_3,a_4\in\mathbb R.
		 			\]
		 			Since \(l_i-\gamma_i=O_p(n^{-1/2})\), the uniform resolvent
		 			bounds and the corresponding derivative bounds imply that
		 			replacing \(l_i\) by \(\gamma_i\) in the centered quadratic and
		 			bilinear fluctuation terms changes \(X_{1n},X_{2n}\), and
		 			\(X_{3n}\) by \(o_p(1)\). At the fixed spectral parameter \(z=\gamma_i\),
		 			\(X_{1n},X_{2n}\), and \(X_{3n}\) are finite linear combinations
		 			of centered Gaussian quadratic and bilinear forms, together with
		 			the Haar quadratic-form fluctuation arising from the Gaussian
		 			SVD of \(\bbY\). Moreover, by the construction of
		 			\(\mathcal R_{ii}\) in Lemma 6.2 of
		 			\cite{bao2019canonical}, \(\mathcal R_{ii}\) is also a finite
		 			linear combination of centered Gaussian quadratic and bilinear
		 			forms involving the same Gaussian rows. Therefore, for arbitrary fixed
		 			\(a_1,a_2,a_3,a_4\), the same Gaussian integration-by-parts
		 			argument, together with the Haar fourth-moment identity and the
		 			Cramér--Wold device, yields the joint convergence
		 			\[
		 			(X_{1n},X_{2n},X_{3n},\mathcal R_{ii})
		 			\]
		 			to a centered Gaussian vector.
		 			\fi
		 			
		 			It remains to compute the limiting variances and covariances which determine the variance of the final linear combination appearing in \(G_{i,n}\). We first consider $X_{1n}$. Since the entries of $\mathcal V_{k1}$ are centered Gaussian quadratic forms, the usual Gaussian quadratic-form identity gives
		 			\begin{align*}
		 				\operatorname{Var}(X_{1n})=&\mathbb{E}\frac{2(1-r_i)^2}{n}\text{tr}(\bbP_y\!-\!\gamma_i\bbI_n)\bbW^\top \bbU_2\Phi^{2}(\gamma_i)\bbU_2^\top \bbW(\bbP_y\!-\!\gamma_i\bbI_n)^2\\
		 				&\times\bbW^\top \bbU_2\Phi^{2}(\gamma_i)\bbU_2^\top \bbW(\bbP_y\!-\!\gamma_i\bbI_n).
		 				%\operatorname{Var}([\sqrt{n}\mathcal{V}_{k2}]_{ii})=&\mathbb{E}(1-l_i)^4\frac{2r_i^2}{q}\text{tr}\bbY\bbW^\top \bbU_2\Phi^2(l_i)\bbU_2^\top \bbW\bbY^\top \bbY\bbW^\top \bbU_2\Phi^2(l_i)\bbU_2^\top \bbW\bbY^\top .
		 			\end{align*}
		 			The fluctuation $X_{2n}$ has a similar quadratic-form structure, but after the Haar replacement in \eqref{Vk2} the leading random direction is Haar distributed rather than an i.i.d.\ Gaussian vector. Therefore, its variance has to be computed from the fourth-moment identity for a Haar row. We treat this term separately below. 
		 			
		 			Using the matrix identity $(\bbE+\bbH)\Phi(z)=z^{-1}(\bbE\Phi-\bbI_{p-k})$,
		 			we may rewrite the variance of $X_{1n}$ into the limiting forms of three expectation terms:
		 			\begin{align}
		 				\operatorname{Var}(X_{1n})=&\mathbb{E}\frac{2(1-r_i)^2}{n}\text{tr}\Phi^{2}(\gamma_i)\bbU_2^\top \bbW(\bbP_y\!-\!\gamma_i\bbI_n)^2\bbW^\top \bbU_2\Phi^{2}(\gamma_i)\bbU_2^\top \bbW(\bbP_y\!-\!\gamma_i\bbI_n)^2\bbW^\top \bbU_2\nonumber\\[0.5em]
		 				=&\mathbb{E}\frac{2(1-r_i)^2}{n}\text{tr}\Phi^{2}(\gamma_i)((1-\gamma_i)^2\bbE+\gamma_i^2\bbH)\Phi^{2}(\gamma_i)((1-\gamma_i)^2\bbE+\gamma_i^2\bbH)\nonumber\\[0.5em]
		 				=&\mathbb{E}\frac{2(1-r_i)^2}{n}\text{tr}((1-\gamma_i)\bbE\Phi(\gamma_i)-\gamma_i\bbI)\Phi(\gamma_i)((1-l_i)\bbE\Phi(\gamma_i)-\gamma_i\bbI)\Phi(\gamma_i)\nonumber\\[0.5em]
		 				=&2(1\!-\!r_i)^2\left( (1\!-\!\gamma_i)^2\mathbb{E}\frac{1}{n}\text{tr}\bbE\Phi^2(\gamma_i)\bbE\Phi^2(\gamma_i)\!\right.\notag\\[0.5em]
		 				&~~~~~~~~~~~~~~~~~~~~~~~~~~\left.-\!2(1\!-\!\gamma_i)\gamma_i\mathbb{E}\frac{1}{n}\text{tr}\bbE\Phi^3(\gamma_i)\!+\!\mathbb{E}\frac{\gamma_i^2}{n}\text{tr}\Phi^2(\gamma_i)\right) .\label{eq52}
		 			\end{align}
		 			This shows that the limiting variance of $X_{1n}$ reduces to the limits
		 			of the three expectation terms in \eqref{eq52}.
		 			
		 			We now turn to
		 			\begin{align*}
		 				\operatorname{Var}([\sqrt{n}\mathcal{V}_{k2}(l_i)]_{ii})
		 				=\mathbb{E}\frac{2r_i^2(1-l_i)^4}{q}\text{tr}\bbY\bbW^\top \bbU_2\Phi^2(l_i)\bbU_2^\top \bbW\bbY^\top \bbY\bbW^\top \bbU_2\Phi^2(l_i)\bbU_2^\top \bbW\bbY^\top .
		 			\end{align*}
		 			Expanding this variance via the cumulant expansion yields
		 			\begin{align*}
		 				&\mathbb{E}\frac{1}{q}\text{tr}\bbY\bbW^\top \bbU_2\Phi^2(l_i)\bbU_2^\top \bbW\bbY^\top \bbY\bbW^\top \bbU_2\Phi^2(l_i)\bbU_2^\top \bbW\bbY^\top \\
		 				=&-\mathbb{E}\frac{1}{n}\text{tr}\bbH\Phi(l_i)\frac{1}{q}\text{tr}\bbE\Phi^2(l_i)\bbU_2^\top \bbW\bbY^\top \bbY\bbW^\top \bbU_2\Phi^2(l_i)\\
		 				&-\mathbb{E}\frac{1}{n}\text{tr}\bbH\Phi^2(l_i)\frac{1}{q}\text{tr}\bbE\Phi(l_i)\bbU_2^\top \bbW\bbY^\top \bbY\bbW^\top \bbU_2\Phi^2(l_i)\\
		 				&+\mathbb{E}\frac{1}{n}\text{tr}\bbE\Phi^2(l_i)\frac{1}{q}\text{tr}\bbU_2^\top \bbW\bbY^\top \bbY\bbW^\top \bbU_2\Phi^2(l_i)\\
		 				&-\mathbb{E}\frac{1}{q}\text{tr}\bbH\Phi^2(l_i)\bbU_2^\top \bbW\bbY^\top \bbY\bbW^\top \bbU_2\Phi(l_i)\frac{1}{n}\text{tr}\bbE\Phi^2(l_i)\\
		 				&-\mathbb{E}\frac{1}{q}\text{tr}\bbH\Phi^2(l_i)\bbU_2^\top \bbW\bbY^\top \bbY\bbW^\top \bbU_2\Phi^2(l_i)\frac{1}{n}\text{tr}\bbE\Phi(l_i)+O\left( \frac{1}{n}\right) \\
		 				=&J_{k1}(r_i)\!+\!J_{k2}(r_i)\!+\!J_{k3}(r_i)\!+\!J_{k4}(r_i)\!+\!J_{k5}(r_i)\!+\!O\left( \frac{1}{n}\right) .
		 			\end{align*}
		 			Indeed, each normalized trace term is a smooth function of Gaussian entries with gradient of order $O(n^{-1})$,
		 			so the Gaussian Poincar\'e inequality implies the variances of the above trace terms are of order $O(n^{-2})$.
		 			Using the previously obtained estimates for
		 			$n^{-1}\text{tr}\bbE\Phi$, $n^{-1}\text{tr}\bbH\Phi$, $n^{-1}\text{tr}\bbE\Phi^2$, and
		 			$n^{-1}\text{tr}\bbH\Phi^2$, and controlling the approximation errors via the
		 			Cauchy--Schwarz inequality, it remains to compute the limits of three key quantities
		 			stated in Lemma~\ref{ccalemma3}:
		 			\begin{align*}
		 				&\frac{1}{q}\text{tr}\bbU_2^\top \bbW\bbY^\top \bbY\bbW^\top \bbU_2\Phi^2(l_i),\ 
		 				\frac{1}{q}\text{tr}\bbE\Phi^2(l_i)\bbU_2^\top  \bbW\bbY^\top \bbY\bbW^\top \bbU_2\Phi^2(l_i),\\
		 				& \frac{1}{q}\text{tr}\bbH\Phi^2(l_i)\bbU_2^\top \bbW\bbY^\top \bbY\bbW^\top \bbU_2 \Phi^2(l_i).
		 			\end{align*}

		 			Thus, the variance of \(X_{2n}\) is determined by two quantities: the second moment of the double-\(\bbY^\top\bbY\) trace and the square of the corresponding normalized trace. The resolvent trace estimates needed below, as well as those required for \(X_{1n}\), \(X_{3n}\), and the five mixed trace quantities \(J_{k1},\ldots,J_{k5}\), are collected in the following lemma.
		 			
		 			\begin{lemma}\label{ccalemma3}
		 				Under the same conditions as in Theorem~\ref{th2}, the following estimates hold:
		 				\begin{align}
		 					&\mathbb{E}\frac{1}{n}\text{tr}\Phi^2(l_i)=\frac{c_1(1-c_2)r_i^4}{\left( (1-c_1)r_i+c_1\right) ^2(\eta r_i-\omega)^2(\eta r_i^2-\omega)}+O\left( \frac{1}{\sqrt{n}}\right)\notag\\[0.5em]
		 					&~~~~~~~~~~~~~~~~~~:=P_1(r_i)+O\left( \frac{1}{\sqrt{n}}\right),\label{l31}\\[0.5em]
		 					&\mathbb{E}\frac{1}{n}\text{tr}\bbE\Phi^3(l_i)=F_3(r_i)+O\left( n^{-\frac{1}{4}}\right), \label{l32}\\[0.5em]
		 					&\mathbb{E}\frac{1}{n}\text{tr}\bbE\Phi^2(l_i)\bbE\Phi^2(l_i)=Q_3(r_i)+O\left( n^{-\frac{1}{16}}\right), \label{l33}\\[0.5em]
		 					&\mathbb{E}\frac{1}{q}\text{tr}\bbY\bbW^\top \bbU_2\Phi^2(l_i)\bbE\Phi^2(l_i)\bbU_2^\top \bbW\bbY^\top =\frac{Q_4(r_i)}{c_2}+O\left( n^{-\frac{1}{16}}\right), \label{l34}\\[0.5em]
		 					&\mathbb{E}\ \frac{1}{q}\text{tr}\bbY\bbW^\top \bbU_2\Phi^2(l_i)\bbE\Phi(l_i)\bbU_2^\top \bbW\bbY^\top =\frac{Q_5(r_i)}{c_2}+O\left( n^{-\frac{1}{4}}\right), \label{l35}\\[0.5em]
		 					&\mathbb{E}\ \frac{1}{q}\text{tr}\bbY\bbW^\top \bbU_2\Phi^3(l_i)\bbU_2^\top \bbW\bbY^\top =\frac{Q_7(r_i)}{c_2}+O\left( n^{-\frac{1}{8}}\right) ,\label{l36}
		 				\end{align}
		 				For notational convenience, the definitions of $F_3(r_i)$, $Q_3(r_i)$, $Q_4(r_i)$,
		 				$Q_5(r_i)$, and $Q_7(r_i)$ are given in \eqref{EPhi3}, \eqref{Q3}, \eqref{Q4},
		 				\eqref{Q5}, and \eqref{Q7}, respectively.
		 			\end{lemma}
		 			The detailed proof of Lemma~\ref{ccalemma3} is deferred to Section~\ref{ccalemma3proof}.
		 			
		 			For the cross terms $I_{k3}$ and $I_{k4}$, one has $I_{k3}+I_{k4}=o_p(1)$.
		 			Moreover, their fluctuation can be quantified as follows:
		 			\begin{align*}
		 				[I_{k3}+I_{k4}]_{ii}=2(1-l_i)\boldsymbol{\Gamma}_k \bbU_1^\top \bbW(\bbP_y\!-\!l_i\bbI_n)\bbW^\top \bbU_2\Phi^{2}(l_i)\bbU_2^\top \bbW\bbV_y^\top \boldsymbol{\Lambda}_y\tilde{\bbU}_y\boldsymbol{\Lambda}_k,
		 			\end{align*}
		 			where \(\bbU_1^\top\bbW\) and the Haar matrix \(\tilde{\bbU}_y\) are independent of
		 			\[
		 			\mathcal B=(\bbP_y-l_i\bbI_n)\bbW^\top \bbU_2
		 			\Phi^2(l_i)\bbU_2^\top \bbW\bbV_y^\top\boldsymbol{\Lambda}_y.
		 			\]
		 			Therefore,
		 			\begin{align*}
		 				&\operatorname{Var}(\sqrt{n}[I_{k3}+I_{k4}]_{ii})\\[0.5em]
		 				=&4(1-l_i)^2(1-r_i)r_i\mathbb{E}\frac{1}{q}\text{tr}\mathcal{B}\mathcal{B}^\top \\[0.5em]
		 				=&4(1-l_i)^3(1-r_i)r_i\mathbb{E}\frac{1}{q}\text{tr}\bbY\bbW^\top \bbU_2\Phi^2(l_i)\bbE\Phi^2(l_i)\bbU_2^\top \bbW\bbY^\top \\[0.5em]
		 				&-4(1-l_i)^2(1-r_i)r_i\mathbb{E}\frac{p}{n}\frac{1}{q}\text{tr}\bbY\bbW^\top \bbU_2\Phi^3(l_i)\bbU_2^\top \bbW\bbY^\top \\[0.5em]
		 				=&\frac{
		 					4r_i(1-r_i)(1-\gamma_i)^2
		 				}{
		 					c_2
		 				}
		 				\left(
		 				(1-\gamma_i)Q_3(r_i)
		 				-
		 				\gamma_i F_3(r_i)
		 				\right)+O(n^{-1/8})\\
		 				=&A_3(r_i)+o(1).
		 			\end{align*}
		 			
		 			Combining \eqref{eq52} with \eqref{l31}--\eqref{l33} yields
		 			\begin{align}
		 				\operatorname{Var}(X_{1n})=&2(1\!-\!r_i)^2P_1(r_i)-4\frac{\eta r_i-\omega}{r_i}(1-r_i)^3(P_1(r_i)+F_3(r_i))\notag\\
		 				&+\frac{2(\eta r_i-\omega)^2}{r_i^2}(1-r_i)^4(P_1(r_i)+2F_3(r_i)+Q_3(r_i))+O\left(n^{-1/16} \right)\notag\\
		 				=&A_1(r_i)+o(1), \label{varvk1}
		 			\end{align}
		 			where the explicit expressions of $P_1(r_i)$, $F_3(r_i)$, and $Q_3(r_i)$ are given in
		 			\eqref{l31}, \eqref{EPhi3}, and \eqref{Q3}, respectively.
		 			
		 			We next compute $\operatorname{Var}(X_{2n})$. For completeness, we retain the five trace quantities arising from the integration-by-parts calculation of the mixed resolvent traces. Although the exact Haar--Wishart second-moment calculation below is used to determine $A_2(r_i)$, these five quantities give explicit reductions of the corresponding trace
		 			terms. We obtain
		 			\begin{align}
		 				&J_{k1}(r_i)=\frac{c_1r_i+\omega(1-r_i)}{c_2r_i-c_2(\eta r_i-\omega)(1-r_i)}Q_4(r_i),\label{Jk1}\\[0.5em]
		 				&J_{k2}(r_i)=\frac{-c_1r_i^2Q_5(r_i)}{c_2(\eta r_i^2-\omega)((1-c_1)r_i+c_1)^2} ,\label{Jk2}
		 			\end{align}
		 			where $Q_4(r_i)$ and $Q_5(r_i)$ are defined in \eqref{Q4} and \eqref{Q5}.
		 			
		 			Next, the term $n^{-1}\text{tr}\bbU_2^\top \bbW\bbY^\top \bbY\bbW^\top \bbU_2\Phi^2(l_i)
		 			=n^{-1}\text{tr}\bbY\bbW^\top \bbU_2\Phi^2(l_i)\bbU_2^\top \bbW\bbY^\top$
		 			appearing in $J_{k3}$ has already been computed in the analysis of $I_{k2}$, and thus
		 			\begin{align}
		 				J_{k3}(r_i)=\frac{1}{c_2}
		 				\left(
		 				\frac{
		 					\omega(1-c_2)r_i^2
		 				}{
		 					(\eta r_i-\omega)^2
		 					(\eta r_i^2-\omega)
		 				}
		 				\right)^2.\label{Jk3}
		 			\end{align}
		 			Finally, we simplify $J_{k4}$ and $J_{k5}$. Note that
		 			\begin{align*}
		 				\frac{1}{n}\text{tr}\bbH\Phi^2(l_i)\bbU_2^\top \bbW\bbY^\top \bbY\bbW^\top \bbU_2\Phi(l_i)
		 				=&\frac{1-l_i}{l_i}\frac{1}{n}\text{tr} \bbE\Phi^2(l_i)\bbU_2^\top \bbW\bbY^\top \bbY\bbW^\top \bbU_2\Phi(l_i)\\[0.5em]
		 				&-\frac{1}{l_in}\text{tr}\Phi(l_i)\bbU_2^\top \bbW\bbY^\top \bbY\bbW^\top \bbU_2\Phi(l_i),\\[0.5em]
		 				\frac{1}{n}\text{tr}\bbH\Phi^2(l_i)\bbU_2^\top \bbW\bbY^\top \bbY\bbW^\top \bbU_2\Phi^2(l_i)
		 				=&\frac{1-l_i}{l_i}\frac{1}{n}\text{tr} \bbE\Phi^2(l_i)\bbU_2^\top \bbW\bbY^\top \bbY\bbW^\top \bbU_2\Phi^2(l_i)\\[0.5em]
		 				&-\frac{1}{l_in}\text{tr}\Phi(l_i)\bbU_2^\top \bbW\bbY^\top \bbY\bbW^\top \bbU_2\Phi^2(l_i),
		 			\end{align*}
		 			which leads to
		 			\begin{align}
		 				&J_{k4}(r_i)={}
		 				-\frac{1}{c_2\gamma_i}
		 				\frac{
		 					\omega(1-c_2)r_i^2
		 				}{
		 					(\eta r_i-\omega)^2(\eta r_i^2-\omega)
		 				}
		 				\left(
		 				(1-\gamma_i)Q_5(r_i)
		 				-
		 				\frac{
		 					\omega(1-c_2)r_i^2
		 				}{
		 					(\eta r_i-\omega)^2(\eta r_i^2-\omega)
		 				}
		 				\right),\label{Jk4}\\[0.5em]
		 				&J_{k5}(r_i)=\frac{
		 					\omega
		 				}{
		 					c_2\gamma_i(\eta r_i-\omega)
		 				}
		 				\left(
		 				(1-\gamma_i)Q_3(r_i)-F_3(r_i)
		 				\right).\label{Jk5}
		 			\end{align}
		 			These identities follow from the resolvent equation $((1-z)\bbE-z\bbH)\Phi(z)=\bbI_{p-k}$, which implies $\bbH\Phi^2(z)=(1-z)\bbE\Phi^2(z)/z-\Phi(z)/z$.
		 			
		 			Recall that the fourth-moment identity for the centered Haar quadratic form gives
		 			\[
		 			\begin{aligned}
		 				\operatorname{Var}(X_{2n})
		 				={}&
		 				\frac{2nr_i^2(1-\gamma_i)^4}{q+2}
		 				\Bigg(
		 					\mathbb E\frac1q
		 					\operatorname{tr}
		 					\left(
		 					\bbY\bbW^\top\bbU_2
		 					\Phi^2(\gamma_i)
		 					\bbU_2^\top\bbW\bbY^\top
		 					\right)^2
		 					\\
		 					&\qquad\qquad-
		 					\mathbb E
		 					\left( 
		 					\frac1q
		 					\operatorname{tr}
		 					\bbY\bbW^\top\bbU_2
		 					\Phi^2(\gamma_i)
		 					\bbU_2^\top\bbW\bbY^\top
		 					\right) ^2
		 					\Bigg)
		 				+o(1).
		 			\end{aligned}
		 			\]
		 			Using $n/(q+2)\to c_2^{-1}$, we therefore obtain
		 			\[
		 			\begin{aligned}
		 				\operatorname{Var}(X_{2n})
		 				={}&
		 				\frac{
		 					2r_i^2(1-\gamma_i)^4
		 				}{
		 					c_2^2
		 				}
		 				\Bigg(
		 				Q_3(r_i)
		 				+
		 				\left(1-\frac1{c_2}\right)
		 				\left(
		 				\frac{
		 					\omega(1-c_2)r_i^2
		 				}{
		 					(\eta r_i-\omega)^2(\eta r_i^2-\omega)
		 				}
		 				\right)^2
		 				\Bigg)
		 				+o(1)
		 				\\
		 				={}&A_2(r_i)+o(1).
		 			\end{aligned}
		 			\]
                    
		 			By the definition of the conditional centering,
		 			\[
		 			\begin{aligned}
		 				X_{1n}
		 				={}&\sqrt n(1-r_i)
		 				\Bigg(
		 				\Big[
		 				\bbU_1^\top\bbW
		 				(\bbP_y-\gamma_i\bbI_n)
		 				\bbW^\top\bbU_2\Phi^2(\gamma_i)
		 				\bbU_2^\top\bbW
		 				(\bbP_y-\gamma_i\bbI_n)
		 				\bbW^\top\bbU_1
		 				\Big]_{ii}
		 				\\
		 				&\qquad
		 				-\frac1n\operatorname{tr}\!
		 				\Big(
		 				(\bbP_y-\gamma_i\bbI_n)
		 				\bbW^\top\bbU_2\Phi^2(\gamma_i)
		 				\bbU_2^\top\bbW
		 				(\bbP_y-\gamma_i\bbI_n)
		 				\Big)
		 				\Bigg).
		 			\end{aligned}
		 			\]
		 			The coefficient matrix of this Gaussian quadratic form is
		 			$\mathcal F_i$-measurable. Hence
		 			\[
		 			\mathbb E\!\left(
		 			X_{1n}\mid\mathcal F_i,\bbU_y
		 			\right)=0.
		 			\]
		 			Next, $X_{2n}$ is measurable with respect to
		 			$\mathcal F_i\vee\sigma(\bbU_y)$, since $I_{k2}(\gamma_i)$ depends
		 			only on $\bbY$ and $\bbU_2^\top\bbW$. Its conditional centering gives
		 			\[
		 			\mathbb E(X_{2n}\mid\mathcal F_i)=0.
		 			\]
		 			Since $I_{k4}(\gamma_i)=I_{k3}(\gamma_i)^\top$ and
		 			$\bbP_y\bbY^\top=\bbY^\top$, the diagonal cross term satisfies
		 			\[
		 			\begin{aligned}
		 				X_{3n}
		 				=2\sqrt n(1-\gamma_i)
		 				\Big[
		 				\boldsymbol\Gamma_k\bbU_1^\top\bbW
		 				(\bbP_y-\gamma_i\bbI_n)
		 				\bbW^\top\bbU_2\Phi^2(\gamma_i)
		 				\bbU_2^\top\bbW\bbY^\top
		 				\bbV_1\boldsymbol\Lambda_k
		 				\Big]_{ii}.
		 			\end{aligned}
		 			\]
		 			Conditional on $(\mathcal F_i,\bbU_y)$, this is a linear form in
		 			the centered Gaussian entries of $\bbU_1^\top\bbW$. Therefore,
		 			\[
		 			\mathbb E\!\left(
		 			X_{3n}\mid\mathcal F_i,\bbU_y
		 			\right)=0.
		 			\]
		 			We now compute the three conditional cross-moments.
		 			Since $X_{2n}$ is measurable with respect to
		 			$\mathcal F_i\vee\sigma(\bbU_y)$, we have
		 			\[
		 			\begin{aligned}
		 				\mathbb E(X_{1n}X_{2n}\mid\mathcal F_i)
		 				&=
		 				\mathbb E\!\left[
		 				X_{2n}\,
		 				\mathbb E(X_{1n}\mid\mathcal F_i,\bbU_y)
		 				\mid\mathcal F_i
		 				\right]
		 				=0,
		 				\\
		 				\mathbb E(X_{2n}X_{3n}\mid\mathcal F_i)
		 				&=
		 				\mathbb E\!\left[
		 				X_{2n}\,
		 				\mathbb E(X_{3n}\mid\mathcal F_i,\bbU_y)
		 				\mid\mathcal F_i
		 				\right]
		 				=0.
		 			\end{aligned}
		 			\]
		 			For the remaining pair, conditional on $(\mathcal F_i,\bbU_y)$,
		 			$X_{1n}$ is an even function of $\bbU_1^\top\bbW$, whereas
		 			$X_{3n}$ is an odd function of $\bbU_1^\top\bbW$.
		 			The conditional Gaussian law is invariant under
		 			$\bbU_1^\top\bbW\mapsto-\bbU_1^\top\bbW$. Consequently,
		 			\[
		 			\mathbb E(X_{1n}X_{3n}\mid\mathcal F_i,\bbU_y)=0,
		 			\]
		 			and the tower property yields
		 			\[
		 			\mathbb E(X_{1n}X_{3n}\mid\mathcal F_i)=0.
		 			\]
		 			Equivalently, this identity follows from the vanishing of all
		 			centered Gaussian moments of odd total degree.
		 			Together with the conditional centering of the three terms,
		 			these identities prove
		 			\[
		 			\operatorname{Cov}(X_{an},X_{bn}\mid\mathcal F_i)=0,
		 			\qquad 1\leq a<b\leq3.
		 			\]
		 			
		 			We next compute the covariance between
		 			$X_{1n}+X_{2n}+X_{3n}$ and the outlier fluctuation.
		 			We take the deterministic-point fluctuation to be
		 			\[
		 			\mathcal R_{ii}
		 			=
		 			\sqrt n\bigl(
		 			(M_\Lambda)_{ii}(\gamma_i)-m_i(\gamma_i)
		 			\bigr),
		 			\]
		 			where $m_i(\gamma_i)=0$.
		 			All second moments below are localized on the event $\Omega_n$
		 			introduced above. In particular,
		 			\[
		 			\Omega_n\in
		 			\sigma(\bbP_y,\bbU_2^\top\bbW)
		 			\subset\mathcal F_i.
		 			\]
		 			This choice of localization preserves the conditional moment
		 			identities used below.
		 			Since
		 			$\mathbb E(X_{an}\mid\mathcal F_i)=0$ for $a=1,2,3$,
		 			the $\mathcal F_i$-measurable conditional mean of
		 			$(M_\Lambda)_{ii}(\gamma_i)$ does not contribute to the
		 			mixed covariance. Thus
		 			\[
		 			\begin{aligned}
		 				&\operatorname{Cov}\bigl(
		 				\mathbf1_{\Omega_n}X_{an},
		 				\mathbf1_{\Omega_n}\mathcal R_{ii}
		 				\bigr)
		 				\\
		 				&\qquad=
		 				\sqrt n\,
		 				\mathbb E\!\left[
		 				\mathbf1_{\Omega_n}X_{an}
		 				\left(
		 				(M_\Lambda)_{ii}(\gamma_i)
		 				-\mathbb E\bigl(
		 				(M_\Lambda)_{ii}(\gamma_i)\mid\mathcal F_i
		 				\bigr)
		 				\right)
		 				\right].
		 			\end{aligned}
		 			\]
		 			Expanding the Schur complement gives a Gaussian quadratic term,
		 			a Haar quadratic term, and a Gaussian--Haar cross term.
		 			The same conditional-centering and symmetry arguments as above
		 			show that only matching terms contribute: $X_{1n}$ pairs with
		 			the Gaussian quadratic term, $X_{2n}$ with the Haar quadratic
		 			term, and $X_{3n}$ with the cross term.
		 			
		 			The Gaussian quadratic-form covariance identity gives
		 			\[
		 			\begin{aligned}
		 				&\operatorname{Cov}\bigl(
		 				\mathbf1_{\Omega_n}X_{1n},
		 				\mathbf1_{\Omega_n}\mathcal R_{ii}
		 				\bigr)
		 				\\
		 				&=
		 				\frac{2(1-r_i)^2}{n}
		 				\mathbb E\Bigg(
		 				\mathbf1_{\Omega_n}
		 				\operatorname{tr}\Bigg(
		 				\Phi^2(\gamma_i)
		 				\bbU_2^\top\bbW
		 				(\bbP_y-\gamma_i\bbI_n)^3
		 				\bbW^\top\bbU_2
		 				\\
		 				&\qquad\qquad
		 				-\Phi^2(\gamma_i)
		 				\bigl((1-\gamma_i)^2\bbE+\gamma_i^2\bbH\bigr)
		 				\Phi(\gamma_i)
		 				\bigl((1-\gamma_i)^2\bbE+\gamma_i^2\bbH\bigr)
		 				\Bigg)
		 				\Bigg).
		 			\end{aligned}
		 			\]
		 			Indeed, the covariance of two quadratic forms in a Gaussian
		 			vector with covariance $n^{-1}\bbI_n$ is twice the trace of the
		 			product of their coefficient matrices divided by $n^2$;
		 			the two factors $\sqrt n$ produce the prefactor $2/n$.
		 			Using the preceding resolvent identity, the trace in braces
		 			reduces to
		 			\[
		 			\begin{aligned}
		 				(1-\gamma_i)\operatorname{tr}\bbE\Phi^2(\gamma_i)
		 				-(1-\gamma_i)^2
		 				\operatorname{tr}
		 				\bbE\Phi(\gamma_i)\bbE\Phi^2(\gamma_i).
		 			\end{aligned}
		 			\]
		 			Consequently,
		 			\[
		 			\begin{aligned}
		 				&\operatorname{Cov}\bigl(
		 				\mathbf1_{\Omega_n}X_{1n},
		 				\mathbf1_{\Omega_n}\mathcal R_{ii}
		 				\bigr)
		 				\\
		 				=&
		 				2(1-r_i)^2(1-\gamma_i)
		 				\mathbb E\Bigg(
		 				\mathbf1_{\Omega_n}
		 				\left(
		 				\frac1n\operatorname{tr}\bbE\Phi^2(\gamma_i)
		 				-\frac{1-\gamma_i}{n}
		 				\operatorname{tr}
		 				\bbE\Phi(\gamma_i)\bbE\Phi^2(\gamma_i)
		 				\right)
		 				\Bigg).
		 			\end{aligned}
		 			\]

		 			Both cross terms are linear in the Gaussian row corresponding
		 			to the $i$-th population direction. Integrating this row first
		 			produces a factor $1/n$, which cancels the factor $n$ from the
		 			two $\sqrt n$ normalizations. The second moment of the Haar
		 			direction produces a factor $1/q$. Conditional on $(\bbP_y,\bbU_2^\top\bbW)$, the first moment
		 			of $\bbY^\top\bbY$ is
		 			\[
		 			\mathbb E\bigl(
		 			\bbY^\top\bbY\mid
		 			\bbP_y,\bbU_2^\top\bbW
		 			\bigr)=\bbP_y.
		 			\]
		 			Using this identity and the resolvent identities above gives
		 			\[
		 			\begin{aligned}
		 				&\operatorname{Cov}\bigl(
		 				\mathbf1_{\Omega_n}X_{3n},
		 				\mathbf1_{\Omega_n}\mathcal R_{ii}
		 				\bigr)
		 				\\
		 				=&
		 				\frac{4r_i(1-r_i)(1-\gamma_i)^2}{q}
		 				\mathbb E\Bigg(
		 				\mathbf1_{\Omega_n}
		 				\Big(
		 				\operatorname{tr}\bbE\Phi^2(\gamma_i)(1-\gamma_i)
		 				\operatorname{tr}
		 				\bbE\Phi(\gamma_i)\bbE\Phi^2(\gamma_i)
		 				\Big)
		 				\Bigg).
		 			\end{aligned}
		 			\]
		 		
		 			For a uniform unit vector in $\mathbb R^q$, the covariance of
		 			two quadratic forms with symmetric coefficient matrices
		 			$\bbA$ and $\bbB$ is
		 			\[
		 			\frac{2}{q(q+2)}
		 			\left(
		 			\operatorname{tr}(\bbA\bbB)
		 			-\frac1q\operatorname{tr}\bbA\operatorname{tr}\bbB
		 			\right).
		 			\]
		 			Applying this identity conditionally on $\mathcal F_i$ yields
		 			\[
		 			\begin{aligned}
		 				&\operatorname{Cov}\bigl(
		 				\mathbf1_{\Omega_n}X_{2n},
		 				\mathbf1_{\Omega_n}\mathcal R_{ii}
		 				\bigr)
		 				\\
		 				=&
		 				\frac{2nr_i^2(1-\gamma_i)^3}{q(q+2)}
		 				\mathbb E\Bigg(
		 				\mathbf1_{\Omega_n}
		 				\Bigg(
		 				\operatorname{tr}\Bigg(
		 				\bigl(
		 				\bbY\bbW^\top\bbU_2\Phi^2(\gamma_i)
		 				\bbU_2^\top\bbW\bbY^\top
		 				\bigr)
		 				\bigl(
		 				\bbY\bbY^\top
		 				-(1-\gamma_i)
		 				\bbY\bbW^\top\bbU_2\Phi(\gamma_i)
		 				\bbU_2^\top\bbW\bbY^\top
		 				\bigr)
		 				\Bigg)
		 				\\
		 				&\qquad
		 				-\frac1q
		 				\operatorname{tr}\bigl(
		 				\bbY\bbW^\top\bbU_2\Phi^2(\gamma_i)
		 				\bbU_2^\top\bbW\bbY^\top
		 				\bigr)
		 				\operatorname{tr}\bigl(
		 				\bbY\bbY^\top
		 				-(1-\gamma_i)
		 				\bbY\bbW^\top\bbU_2\Phi(\gamma_i)
		 				\bbU_2^\top\bbW\bbY^\top
		 				\bigr)
		 				\Bigg)
		 				\Bigg).
		 			\end{aligned}
		 			\]
		 			
		 			We evaluate the two terms in braces by conditioning further
		 			on $(\bbP_y,\bbU_2^\top\bbW)$.
		 			On the $q$-dimensional range of $\bbP_y$, the conditional
		 			distribution of $\bbY^\top\bbY$ is the orthogonally invariant
		 			Wishart distribution with $n$ degrees of freedom and scale
		 			$n^{-1}\bbI_q$.
		 			Its second moments follow from
		 			\[
		 			\mathbb E(S_{ab}S_{cd})
		 			=
		 			\delta_{ab}\delta_{cd}
		 			+\frac1n
		 			\bigl(
		 			\delta_{ac}\delta_{bd}+\delta_{ad}\delta_{bc}
		 			\bigr),
		 			\]
		 			where $(S_{ab})$ denotes this Wishart matrix in an orthonormal
		 			basis of the range of $\bbP_y$.
		 			
		 			In particular, for symmetric deterministic test matrices
		 			$\bbA,\bbB$ supported on that range,
		 			\[
		 			\begin{aligned}
		 				&\mathbb E\!\left( 
		 				\operatorname{tr}
		 				\bigl(\bbY^\top\bbY\bbA\bbY^\top\bbY\bbB\bigr)
		 				\mid\bbP_y,\bbU_2^\top\bbW
		 				\right)=
		 				\left(1+\frac1n\right)
		 				\operatorname{tr}(\bbP_y\bbA\bbP_y\bbB)
		 				+\frac1n
		 				\operatorname{tr}(\bbP_y\bbA)
		 				\operatorname{tr}(\bbP_y\bbB),
		 				\\[1mm]
		 				&\mathbb E\!\left( 
		 				\operatorname{tr}(\bbY^\top\bbY\bbA)
		 				\operatorname{tr}(\bbY^\top\bbY\bbB)
		 				\mid\bbP_y,\bbU_2^\top\bbW
		 				\right)=
		 				\operatorname{tr}(\bbP_y\bbA)
		 				\operatorname{tr}(\bbP_y\bbB)
		 				+\frac2n
		 				\operatorname{tr}(\bbP_y\bbA\bbP_y\bbB).
		 			\end{aligned}
		 			\]
		 			Here the test matrices may depend on
		 			$(\bbP_y,\bbU_2^\top\bbW)$, since these quantities are fixed
		 			under the conditioning.
		 			Applying these identities with
		 			\[
		 			\bbA=
		 			\bbP_y\bbW^\top\bbU_2\Phi^2(\gamma_i)
		 			\bbU_2^\top\bbW\bbP_y
		 			\]
		 			and
		 			\[
		 			\bbB=
		 			\bbP_y
		 			-(1-\gamma_i)
		 			\bbP_y\bbW^\top\bbU_2\Phi(\gamma_i)
		 			\bbU_2^\top\bbW\bbP_y,
		 			\]
		 			we obtain
		 			\[
		 			\begin{aligned}
		 				&\operatorname{Cov}\bigl(
		 				\mathbf1_{\Omega_n}X_{2n},
		 				\mathbf1_{\Omega_n}\mathcal R_{ii}
		 				\bigr)
		 				\\
		 				&=
		 				\frac{2nr_i^2(1-\gamma_i)^3}{q(q+2)}
		 				\mathbb E\Bigg(
		 				\mathbf1_{\Omega_n}
		 				\Bigg(
		 				\left(1+\frac1n-\frac{2}{nq}\right)
		 				\Big(
		 				\operatorname{tr}\bbE\Phi^2(\gamma_i)
		 				-(1-\gamma_i)
		 				\operatorname{tr}
		 				\bbE\Phi(\gamma_i)\bbE\Phi^2(\gamma_i)
		 				\Big)
		 				\\
		 				&\qquad+
		 				\left(\frac1n-\frac1q\right)
		 				\operatorname{tr}\bbE\Phi^2(\gamma_i)
		 				\Big(
		 				q-(1-\gamma_i)\operatorname{tr}\bbE\Phi(\gamma_i)
		 				\Big)
		 				\Bigg)
		 				\Bigg).
		 			\end{aligned}
		 			\]
		 			The coefficient $(1/n-1/q)$ is precisely the contribution
		 			of the trace-product correction in the Haar covariance formula.
		 			
		 			The deterministic equivalents established above give
		 			\[
		 			\frac1n\operatorname{tr}\bbE\Phi^2(\gamma_i)
		 			\xrightarrow{p}
		 			\frac{\omega(1-c_2)r_i^2}
		 			{(\eta r_i-\omega)^2(\eta r_i^2-\omega)}
		 			\]
		 			and
		 			\[
		 			\frac1n\operatorname{tr}\bbE\Phi(\gamma_i)
		 			\xrightarrow{p}
		 			-\frac{\omega}{\eta r_i-\omega}.
		 			\]
		 			Furthermore,
		 			\[
		 			\frac1n\operatorname{tr}
		 			\bbE\Phi(\gamma_i)\bbE\Phi^2(\gamma_i)
		 			\xrightarrow{p}Q_5(r_i).
		 			\]
		 			To identify the last limit, differentiating the resolvent
		 			at deterministic $z$ gives
		 			\[
		 			\begin{aligned}
		 				\frac1n\operatorname{tr}\bbE\Phi(z)\bbE\Phi^2(z)
		 				&=
		 				\frac1n\operatorname{tr}\bbE\Phi^2(z)
		 				+\frac z2
		 				\left(
		 				\frac1n\operatorname{tr}\bbE\Phi^2(z)
		 				\right)'
		 				\\
		 				&=
		 				\frac1n\operatorname{tr}\Phi(z)
		 				+2z\left(\frac1n\operatorname{tr}\Phi(z)\right)'
		 				+\frac{z^2}{2}
		 				\left(\frac1n\operatorname{tr}\Phi(z)\right)''.
		 			\end{aligned}
		 			\]
		 			At $z=\gamma_i$, its deterministic limit is therefore
		 			\[
		 			F_1(r_i)+2\gamma_iF_2(r_i)
		 			+\frac{\gamma_i^2}{2}QQ_2'(\gamma_i)
		 			=Q_5(r_i).
		 			\]
		 			
		 			On $\Omega_n$, the normalized traces in these formulas,
		 			and their products, are uniformly bounded.
		 			Since $\mathbb P(\Omega_n)\to1$, their convergence in probability
		 			therefore also gives convergence of the localized expectations.
		 			Using $q/n\to c_2$, we conclude that
		 			\[
		 			\begin{aligned}
		 				&\lim_{n\to\infty}
		 				\operatorname{Cov}\bigl(
		 				\mathbf1_{\Omega_n}X_{1n},
		 				\mathbf1_{\Omega_n}\mathcal R_{ii}
		 				\bigr)
		 				\\
		 				&\quad=
		 				2(1-r_i)^2(1-\gamma_i)
		 				\left(
		 				\frac{\omega(1-c_2)r_i^2}
		 				{(\eta r_i-\omega)^2(\eta r_i^2-\omega)}
		 				-(1-\gamma_i)Q_5(r_i)
		 				\right),
		 			\end{aligned}
		 			\]
		 			\[
		 			\begin{aligned}
		 				&\lim_{n\to\infty}
		 				\operatorname{Cov}\bigl(
		 				\mathbf1_{\Omega_n}X_{3n},
		 				\mathbf1_{\Omega_n}\mathcal R_{ii}
		 				\bigr)
		 				\\
		 				&\quad=
		 				\frac{4r_i(1-r_i)(1-\gamma_i)^2}{c_2}
		 				\left(
		 				\frac{\omega(1-c_2)r_i^2}
		 				{(\eta r_i-\omega)^2(\eta r_i^2-\omega)}
		 				-(1-\gamma_i)Q_5(r_i)
		 				\right),
		 			\end{aligned}
		 			\]
		 			and
		 			\[
		 			\begin{aligned}
		 				&\lim_{n\to\infty}
		 				\operatorname{Cov}\bigl(
		 				\mathbf1_{\Omega_n}X_{2n},
		 				\mathbf1_{\Omega_n}\mathcal R_{ii}
		 				\bigr)
		 				\\
		 				&\quad=
		 				\frac{2r_i^2(1-\gamma_i)^3}{c_2^2}
		 				\Bigg(
		 				\frac{\omega(1-c_2)r_i^2}
		 				{(\eta r_i-\omega)^2(\eta r_i^2-\omega)}
		 				-(1-\gamma_i)Q_5(r_i)
		 				\\
		 				&\qquad\qquad+
		 				\left(1-\frac1{c_2}\right)
		 				\frac{\omega(1-c_2)r_i^2}
		 				{(\eta r_i-\omega)^2(\eta r_i^2-\omega)}
		 				\left(
		 				c_2+\frac{(1-\gamma_i)\omega}{\eta r_i-\omega}
		 				\right)
		 				\Bigg).
		 			\end{aligned}
		 			\]
		 			
		 			Finally, the identity
		 			\[
		 			1-\gamma_i
		 			=\frac{(\eta r_i-\omega)(1-r_i)}{r_i}
		 			\]
		 			implies
		 			\[
		 			c_2+\frac{(1-\gamma_i)\omega}{\eta r_i-\omega}
		 			=
		 			c_2+\frac{\omega(1-r_i)}{r_i}.
		 			\]
		 			Adding the three covariance limits and using
		 			\[
		 			\begin{aligned}
		 				&2(1-r_i)^2(1-\gamma_i)
		 				+\frac{4r_i(1-r_i)(1-\gamma_i)^2}{c_2}
		 				+\frac{2r_i^2(1-\gamma_i)^3}{c_2^2}
		 				\\
		 				&\qquad=
		 				2(1-\gamma_i)
		 				\left(
		 				(1-r_i)+\frac{r_i(1-\gamma_i)}{c_2}
		 				\right)^2,
		 			\end{aligned}
		 			\]
		 			we obtain
		 			\[
		 			\begin{aligned}
		 				&\lim_{n\to\infty}
		 				\operatorname{Cov}\left(
		 				\mathbf1_{\Omega_n}(X_{1n}+X_{2n}+X_{3n}),
		 				\mathbf1_{\Omega_n}\mathcal R_{ii}
		 				\right)
		 				\\
		 				={}&
		 				2\left(
		 				(1-r_i)+\frac{r_i(1-\gamma_i)}{c_2}
		 				\right)^2(1-\gamma_i)
		 				\left(
		 				\frac{\omega(1-c_2)r_i^2}
		 				{(\eta r_i-\omega)^2(\eta r_i^2-\omega)}
		 				-(1-\gamma_i)Q_5(r_i)
		 				\right)
		 				\\
		 				&+
		 				\frac{2r_i^2}{c_2^2}
		 				\left(1-\frac1{c_2}\right)(1-\gamma_i)^3
		 				\frac{\omega(1-c_2)r_i^2}
		 				{(\eta r_i-\omega)^2(\eta r_i^2-\omega)}
		 				\left(
		 				c_2+\frac{\omega(1-r_i)}{r_i}
		 				\right).
		 			\end{aligned}
		 			\]
		 			This is the mixed covariance entering the limiting Gaussian
		 			variance. Since localization changes the fluctuation vector
		 			only on $\Omega_n^c$, whose probability tends to zero, it
		 			does not change its joint weak limit.
		 			\iffalse
		 			Note that the entries of $\bbU_1^\top \bbW$ are i.i.d.\ standard normal and are
		 			independent of $\bbY$ and $\bbU_2^\top \bbW$. Conditional on $\bbY$ and $\bbU_2^\top \bbW$,
		 			we obtain
		 			\begin{align*}
		 				\mathbb{E}(\sqrt{n}[I_{k3}]_{ii}\sqrt{n}[\mathcal{V}_{k1}]_{ii}\mid\bbY,\bbU_2^\top \bbW)=0,
		 			\end{align*}
		 			and similarly
		 			\begin{align*}
		 				\mathbb{E}(\sqrt{n}[I_{k4}]_{ii}\sqrt{n}[\mathcal{V}_{k1}]_{ii}\mid\bbY,\bbU_2^\top \bbW)=0.
		 			\end{align*}
		 			By properties of the standard normal distribution, we also have
		 			\begin{align*}
		 				\mathbb{E}(\sqrt{n}[I_{k3}]_{ii}\sqrt{n}[\mathcal{V}_{k2}]_{ii}\mid\bbY,\bbU_2^\top \bbW)
		 				=\mathbb{E}(\sqrt{n}[I_{k4}]_{ii}\sqrt{n}[\mathcal{V}_{k2}]_{ii}\mid\bbY,\bbU_2^\top \bbW)=0.
		 			\end{align*}
		 			\fi
		 			
		 			We now justify the derivative estimates, the boundedness of the relevant matrices, and the use of the Poincar\'e inequality. We begin with the basic deterministic approximation for the resolvent trace. By the cumulant expansion, uniformly for \(z\) in a fixed neighborhood of \(\gamma_i\) separated from the limiting spectral support,
		 			\[
		 			\mathbb E\frac{1}{n}\operatorname{tr}\Phi(z)-QQ_1(z)
		 			=
		 			O\left(\frac1n\right).
		 			\]
		 			Moreover, by the Gaussian Poincar\'e inequality,
		 			\[
		 			\frac{1}{n}\operatorname{tr}\Phi(z)
		 			-
		 			\mathbb E\frac{1}{n}\operatorname{tr}\Phi(z)
		 			=
		 			O_p\left(\frac1n\right).
		 			\]
		 			Hence
		 			\[
		 			\frac{1}{n}\operatorname{tr}\Phi(z)-QQ_1(z)
		 			=
		 			O_p\left(\frac1n\right),
		 			\]
		 			uniformly for \(z\) in the same neighborhood.
		 			
		 			We next explain how derivative estimates are obtained from this approximation. Let \(h=n^{-\alpha}\), where \(h\downarrow0\). By the resolvent identity,
		 			\[
		 			\frac{
		 				\frac{1}{n}\operatorname{tr}\Phi(l_i+h)
		 				-
		 				\frac{1}{n}\operatorname{tr}\Phi(l_i)
		 			}{h}
		 			=
		 			\frac{1}{n}\operatorname{tr}\left(
		 			\Phi(l_i)\bbU_2^\top\bbW\bbW^\top\bbU_2\Phi(l_i)
		 			\right)
		 			+
		 			O_p(h).
		 			\]
		 			On the other hand, using the deterministic equivalent \(QQ_1\), we have
		 			\[
		 			\begin{aligned}
		 				&\frac{
		 					\frac{1}{n}\operatorname{tr}\Phi(l_i+h)
		 					-
		 					\frac{1}{n}\operatorname{tr}\Phi(l_i)
		 				}{h}
		 				\\
		 				=&
		 				\frac{QQ_1(l_i+h)-QQ_1(l_i)}{h}
		 				\\
		 				&+
		 				\frac{
		 					\frac{1}{n}\operatorname{tr}\Phi(l_i+h)-QQ_1(l_i+h)
		 					-
		 					\left(
		 					\frac{1}{n}\operatorname{tr}\Phi(l_i)-QQ_1(l_i)
		 					\right)
		 				}{h}
		 				\\
		 				=&
		 				QQ_2(l_i+\xi h)
		 				+
		 				O_p\left(\frac{1}{nh}\right),
		 			\end{aligned}
		 			\]
		 			for some \(\xi\in(0,1)\). Since \(QQ_2(z)\) is smooth in a neighborhood of \(\gamma_i\), Taylor expansion gives
		 			\[
		 			QQ_2(l_i+\xi h)
		 			=
		 			QQ_2(l_i)+O_p(h).
		 			\]
		 			Therefore,
		 			\[
		 			\frac{
		 				\frac{1}{n}\operatorname{tr}\Phi(l_i+h)
		 				-
		 				\frac{1}{n}\operatorname{tr}\Phi(l_i)
		 			}{h}
		 			=
		 			QQ_2(l_i)
		 			+
		 			O_p(h)
		 			+
		 			O_p\left(\frac{1}{nh}\right).
		 			\]
		 			Choosing \(h=n^{-1/2}\) balances the two error terms and yields
		 			\[
		 			\frac{
		 				\frac{1}{n}\operatorname{tr}\Phi(l_i+n^{-1/2})
		 				-
		 				\frac{1}{n}\operatorname{tr}\Phi(l_i)
		 			}{n^{-1/2}}
		 			=
		 			QQ_2(l_i)+O_p(n^{-1/2}).
		 			\]
		 			Combining this with the resolvent identity above, we obtain the derivative-type estimate
		 			\[
		 			\frac{1}{n}\operatorname{tr}
		 			\Phi(l_i)\bbU_2^\top\bbW\bbW^\top\bbU_2\Phi(l_i)
		 			-
		 			QQ_2(l_i)
		 			=
		 			O_p(n^{-1/2}).
		 			\]
		 			
		 			The same finite-difference argument will be used repeatedly for the resolvent functionals appearing below. More precisely, suppose that, on a fixed neighborhood of $\gamma_i$ separated from the limiting spectral support, a resolvent functional $f_n(z)$ admits a uniform approximation
		 			\[
		 			\sup_{z}
		 			|f_n(z)-f(z)|
		 			=
		 			O_p(n^{-a}),
		 			\]
		 			where the deterministic limit $f$ is sufficiently smooth and the relevant resolvents and their derivatives are uniformly bounded with probability tending to one. Applying the resolvent identity to the difference quotient and choosing
		 			\[
		 			h=n^{-a/2}
		 			\]
		 			balances the deterministic finite-difference error with the approximation error divided by $h$. This yields the derivative estimates required below for the specific resolvent functionals under consideration.
		 			
		 			We emphasize that this argument is used only for the resolvent trace functionals appearing in the present proof, for which the required uniform spectral bounds and smoothness properties have been established.
		 			
		 			We also record the boundedness estimates used throughout the proof. All matrices involved are products of factors such as \(\Phi(z)\), \(\bbE\), \(\bbH\), \(\bbW\bbW^\top\), and \(\bbY\bbY^\top\). By the spectral norm inequality,
		 			\[
		 			\|\bbA\bbB\|\leq \|\bbA\|\|\bbB\|,
		 			\]
		 			it is enough to control these factors separately. From equation (S3.23) in Bao et al.\ (2019) \cite{bao2019canonical}, for \(z\) in the above neighborhood,
		 			\[
		 			\|\Phi(z)\|
		 			\leq
		 			\|(C_{w_2y}-z)^{-1}\|\,\|S_{w_2w_2}^{-1}\|
		 			\leq C
		 			\]
		 			with probability tending to one. Moreover, since \(\bbW\bbW^\top\) and \(\bbY\bbY^\top\) are Gaussian sample covariance matrices, their spectral norms are bounded with probability tending to one. The random matrices \(\bbE\) and \(\bbH\) also satisfy, with probability tending to one,
		 			\[
		 			\|\bbE\|\leq C,
		 			\qquad
		 			\|\bbH\|\leq C.
		 			\]
		 			Consequently, all normalized trace functionals appearing below are uniformly bounded in probability.
		 			
		 			Finally, we justify the use of the Poincar\'e inequality. Let \(F_n\) denote any normalized trace functional of the form considered above, for example \(F_n=n^{-1}\operatorname{tr}\Phi(z)\), or a normalized trace involving additional bounded factors. Differentiating \(F_n\) with respect to the Gaussian entries of \(\bbW\) and \(\bbY\) produces sums of normalized traces with bounded resolvent factors. Using the uniform spectral norm bounds above, we obtain
		 			\[
		 			\sum_{\alpha}
		 			\mathbb E
		 			\left|
		 			\frac{\partial F_n}{\partial \xi_\alpha}
		 			\right|^2
		 			\leq
		 			\frac{C}{n^2},
		 			\]
		 			where the sum runs over all Gaussian entries involved. Therefore, the Gaussian Poincar\'e inequality gives
		 			\[
		 			\operatorname{Var}(F_n)\leq \frac{C}{n^2}.
		 			\]
		 			In particular,
		 			\[
		 			F_n-\mathbb EF_n=O_p\left(\frac1n\right).
		 			\]
		 			The same variance estimate applies to each class of normalized trace functionals used below. Indeed, every such functional is obtained from a finite product of resolvent factors and the matrices already shown above to have uniformly bounded operator norm. Differentiation with respect to a Gaussian entry therefore produces only finitely many additional terms of the same normalized trace type, and the corresponding gradient square remains of order $O(n^{-2})$. Consequently, for each normalized trace functional $F_n$ explicitly appearing below,
		 			\[
		 			\operatorname{Var}(F_n)
		 			=
		 			O(n^{-2}),
		 			\]
		 			and hence
		 			\[
		 			F_n-\mathbb EF_n
		 			=
		 			O_p(n^{-1}).
		 			\]
		 			This justifies replacing products of these normalized trace functionals by products of their expectations up to the error orders required in the subsequent calculations. Hence the products of such normalized traces can be replaced by products of their expectations up to negligible errors by the Cauchy--Schwarz inequality.
		 			This completes the proof of Lemma~\ref{ccalemma2}.
		 		\end{proof}
		 		
		 		\section{Proof of Lemma~\ref{ccalemma3}}\label{ccalemma3proof}
		 		
		 		We prove \eqref{l31}--\eqref{l36} in four parts. Throughout the proof, \(r_i\) is fixed and \(l_i\) denotes the sample spiked eigenvalue associated with \(r_i\). Since \(l_i\) is separated from the limiting bulk, we have
		 		\[
		 		l_i-\gamma(r_i)=O_p(n^{-1/2}),
		 		\]
		 		where
		 		\[
		 		\gamma(r_i)
		 		=
		 		\frac{r_i-(\eta r_i-\omega)(1-r_i)}{r_i}.
		 		\]
		 		Hence
		 		\[
		 		1-\gamma(r_i)
		 		=
		 		\frac{(\eta r_i-\omega)(1-r_i)}{r_i}.
		 		\]
		 		We shall also use the identity
		 		\[
		 		r_i-(\eta r_i-\omega)(1-r_i)
		 		=
		 		\big((1-c_1)r_i+c_1\big)\big((1-c_2)r_i+c_2\big).
		 		\]
		 		
		 		First, the deterministic equivalent of \(n^{-1}\operatorname{tr}\Phi(l_i)\) is
		 		\[
		 		\frac1n\operatorname{tr}\Phi(l_i)
		 		=
		 		F_1(r_i)+O_p(n^{-1/2}),
		 		\]
		 		where
		 		\[
		 		F_1(r_i)
		 		=
		 		-\frac{c_1r_i}
		 		{\big((1-c_1)r_i+c_1\big)(\eta r_i-\omega)}.
		 		\]
		 		Moreover,
		 		\[
		 		\frac1n\operatorname{tr}\bbE\Phi^2(l_i)
		 		=
		 		\frac1n\operatorname{tr}\Phi(l_i)
		 		+
		 		l_i\left(\frac1n\operatorname{tr}\Phi(l_i)\right)'.
		 		\]
		 		Therefore the deterministic equivalent of \(QQ_2(l_i)\) is
		 		\[
		 		QQ_2(l_i)=F_2(r_i)+O_p(n^{-1/2}),
		 		\]
		 		where
		 		\[
		 		F_2(r_i)
		 		=
		 		\frac{
		 			r_i^2
		 			\left(
		 			\omega(1-c_2)r_i\big((1-c_1)r_i+c_1\big)
		 			+
		 			c_1(\eta r_i-\omega)(\eta r_i^2-\omega)
		 			\right)
		 		}
		 		{
		 			\big((1-c_1)r_i+c_1\big)^2
		 			\big((1-c_2)r_i+c_2\big)
		 			(\eta r_i-\omega)^2
		 			(\eta r_i^2-\omega)
		 		}.
		 		\]
		 		With this expression,
		 		\[
		 		F_1(r_i)+\gamma(r_i)F_2(r_i)
		 		=
		 		\frac{\omega(1-c_2)r_i^2}
		 		{(\eta r_i-\omega)^2(\eta r_i^2-\omega)}.
		 		\]
		 		Thus
		 		\[
		 		\frac1n\operatorname{tr}\bbE\Phi^2(l_i)
		 		=
		 		\frac{\omega(1-c_2)r_i^2}
		 		{(\eta r_i-\omega)^2(\eta r_i^2-\omega)}
		 		+
		 		O_p(n^{-1/2}).
		 		\]
		 		
		 		We shall also need the derivative of \(QQ_2\) with respect to \(z\). Since
		 		\[
		 		\gamma'(r_i)=\frac{\eta r_i^2-\omega}{r_i^2},
		 		\]
		 		the chain rule gives
		 		\[
		 		\begin{aligned}
		 			QQ_2'(\gamma_i)
		 			=&
		 			\frac{
		 				r_i^4
		 				\left(
		 				\omega(1-c_2)r_i\big((1-c_1)r_i+c_1\big)
		 				+
		 				c_1(\eta r_i-\omega)(\eta r_i^2-\omega)
		 				\right)
		 			}
		 			{
		 				\big((1-c_1)r_i+c_1\big)^2
		 				\big((1-c_2)r_i+c_2\big)
		 				(\eta r_i-\omega)^2
		 				(\eta r_i^2-\omega)^2
		 			}
		 			\\
		 			&\times
		 			\left(
		 			\frac{2}{r_i}
		 			+
		 			\frac{
		 				\omega(1-c_2)\big(2(1-c_1)r_i+c_1\big)
		 				+
		 				c_1\left(
		 				\eta(\eta r_i^2-\omega)
		 				+
		 				2\eta r_i(\eta r_i-\omega)
		 				\right)
		 			}
		 			{
		 				\omega(1-c_2)r_i\big((1-c_1)r_i+c_1\big)
		 				+
		 				c_1(\eta r_i-\omega)(\eta r_i^2-\omega)
		 			}
		 			\right.
		 			\\
		 			&\qquad\left.
		 			-\frac{2(1-c_1)}{(1-c_1)r_i+c_1}
		 			-\frac{1-c_2}{(1-c_2)r_i+c_2}
		 			-\frac{2\eta}{\eta r_i-\omega}
		 			-\frac{2\eta r_i}{\eta r_i^2-\omega}
		 			\right).
		 		\end{aligned}
		 		\]
		 		
		 		\textbf{Step 1. Proof of \eqref{l31}.} By the cumulant expansion formula, we have
		 		\[
		 		\begin{aligned}
		 			\mathbb E\frac1n\operatorname{tr}\bbE\Phi^2(l_i)
		 			=&
		 			\frac qn\mathbb E\frac1n\operatorname{tr}\Phi^2(l_i)
		 			\\
		 			&-(1-l_i)\mathbb E\left(
		 			\frac1n\operatorname{tr}\bbE\Phi(l_i)
		 			\frac1n\operatorname{tr}\Phi^2(l_i)
		 			+
		 			\frac1n\operatorname{tr}\bbE\Phi^2(l_i)
		 			\frac1n\operatorname{tr}\Phi(l_i)
		 			\right)
		 			+
		 			O(n^{-1}).
		 		\end{aligned}
		 		\]
		 		By the Gaussian Poincar\'e inequality and the Cauchy--Schwarz inequality, the product terms can be factorized up to negligible errors. Hence
		 		\[
		 		\mathbb E\frac1n\operatorname{tr}\Phi^2(l_i)
		 		=
		 		\frac{
		 			1+(1-l_i)\mathbb E\frac1n\operatorname{tr}\Phi(l_i)
		 		}
		 		{
		 			\frac1n\operatorname{tr}\bbP_y
		 			-
		 			(1-l_i)\mathbb E\frac1n\operatorname{tr}\bbE\Phi(l_i)
		 		}
		 		\mathbb E\frac1n\operatorname{tr}\bbE\Phi^2(l_i)
		 		+
		 		O(n^{-1}).
		 		\]
		 		Substituting the deterministic equivalents above gives
		 		\[
		 		\mathbb E\frac1n\operatorname{tr}\Phi^2(l_i)
		 		=
		 		P_1(r_i)+O(n^{-1/2}),
		 		\]
		 		where
		 		\[
		 		P_1(r_i)
		 		=
		 		\frac{
		 			c_1(1-c_2)r_i^4
		 		}
		 		{
		 			\big((1-c_1)r_i+c_1\big)^2
		 			(\eta r_i-\omega)^2
		 			(\eta r_i^2-\omega)
		 		}.
		 		\]
		 		This proves \eqref{l31}.
		 		
		 		\textbf{Step 2. Proof of \eqref{l32}.} Differentiating \(n^{-1}\operatorname{tr}\Phi^2(z)\) with respect to \(z\), we obtain
		 		\[
		 		\left(\frac1n\operatorname{tr}\Phi^2(z)\right)'
		 		=
		 		\frac{2}{nz}
		 		\left(
		 		\operatorname{tr}\bbE\Phi^3(z)
		 		-
		 		\operatorname{tr}\Phi^2(z)
		 		\right).
		 		\]
		 		Therefore, at \(z=l_i\),
		 		\[
		 		\mathbb E\frac1n\operatorname{tr}\bbE\Phi^3(l_i)
		 		=
		 		\mathbb E\frac1n\operatorname{tr}\Phi^2(l_i)
		 		+
		 		\frac{l_i}{2}
		 		\mathbb E\left(\frac1n\operatorname{tr}\Phi^2(l_i)\right)'.
		 		\]
		 		Since
		 		\[
		 		P_1(r_i)
		 		=
		 		\frac{
		 			c_1(1-c_2)r_i^4
		 		}
		 		{
		 			\big((1-c_1)r_i+c_1\big)^2
		 			(\eta r_i-\omega)^2
		 			(\eta r_i^2-\omega)
		 		},
		 		\]
		 		we get
		 		\[
		 		\mathbb E\left(\frac1n\operatorname{tr}\Phi^2(l_i)\right)'
		 		=
		 		P_2(r_i)+O(n^{-1/4}),
		 		\]
		 		where
		 		\[
		 		\begin{aligned}
		 			P_2(r_i)
		 			=&
		 			\frac{
		 				c_1(1-c_2)r_i^6
		 			}
		 			{
		 				\big((1-c_1)r_i+c_1\big)^2
		 				(\eta r_i-\omega)^2
		 				(\eta r_i^2-\omega)^2
		 			}
		 			\\
		 			&\times
		 			\left(
		 			\frac{4}{r_i}
		 			-\frac{2(1-c_1)}{(1-c_1)r_i+c_1}
		 			-\frac{2\eta}{\eta r_i-\omega}
		 			-\frac{2\eta r_i}{\eta r_i^2-\omega}
		 			\right).
		 		\end{aligned}
		 		\]
		 		Consequently,
		 		\[
		 		\mathbb E\frac1n\operatorname{tr}\bbE\Phi^3(l_i)
		 		=
		 		F_3(r_i)+O(n^{-1/4}),
		 		\]
		 		where
		 		\begin{align}
		 		F_3(r_i)
		 		=
		 		P_1(r_i)
		 		+
		 		\frac{r_i-(\eta r_i-\omega)(1-r_i)}{2r_i}P_2(r_i).\label{EPhi3}
		 		\end{align}
		 		This proves \eqref{l32}.
		 		
		 		For later use, we also record the derivative of \(QQ_6\). Since \(QQ_6(l_i)\) is the deterministic equivalent of
		 		\[
		 		\left(\frac1n\operatorname{tr}\Phi^2(l_i)\right)',
		 		\]
		 		we have \(QQ_6(l_i)=P_2(r_i)+O_p(n^{-1/4})\). Hence
		 		\[
		 		\begin{aligned}
		 			QQ_6'(\gamma_i)
		 			=&
		 			\frac{
		 				c_1(1-c_2)r_i^8
		 			}
		 			{
		 				\big((1-c_1)r_i+c_1\big)^2
		 				(\eta r_i-\omega)^2
		 				(\eta r_i^2-\omega)^3
		 			}
		 			\\
		 			&\times
		 			\left(
		 			\left(
		 			\frac{4}{r_i}
		 			-\frac{2(1-c_1)}{(1-c_1)r_i+c_1}
		 			-\frac{2\eta}{\eta r_i-\omega}
		 			-\frac{2\eta r_i}{\eta r_i^2-\omega}
		 			\right)^2
		 			\right.
		 			\\
		 			&\quad
		 			+
		 			\left(
		 			\frac{2}{r_i}
		 			-\frac{2\eta r_i}{\eta r_i^2-\omega}
		 			\right)
		 			\left(
		 			\frac{4}{r_i}
		 			-\frac{2(1-c_1)}{(1-c_1)r_i+c_1}
		 			-\frac{2\eta}{\eta r_i-\omega}
		 			-\frac{2\eta r_i}{\eta r_i^2-\omega}
		 			\right)
		 			\\
		 			&\quad\left.
		 			-\frac{4}{r_i^2}
		 			+
		 			\frac{2(1-c_1)^2}{\big((1-c_1)r_i+c_1\big)^2}
		 			+
		 			\frac{2\eta^2}{(\eta r_i-\omega)^2}
		 			+
		 			\frac{2\eta(\eta r_i^2+\omega)}{(\eta r_i^2-\omega)^2}
		 			\right).
		 		\end{aligned}
		 		\]
		 		
		 		\textbf{Step 3. Proof of \eqref{l33}.} Applying the cumulant expansion formula to
		 		\[
		 		\mathbb E\frac1n\operatorname{tr}\bbE\Phi^2(l_i)\bbE\Phi^2(l_i),
		 		\]
		 		we obtain
		 		\[
		 		\begin{aligned}
		 			&\mathbb E\frac1n\operatorname{tr}\bbE\Phi^2(l_i)\bbE\Phi^2(l_i)
		 			\\
		 			=&
		 			\mathbb E\left(
		 			\frac1n\operatorname{tr}\bbP_y
		 			-
		 			\frac{1-l_i}{n}\operatorname{tr}\bbE\Phi(l_i)
		 			\right)
		 			\frac1n\operatorname{tr}\bbE\Phi^4(l_i)
		 			\\
		 			&-
		 			\mathbb E\frac{1-l_i}{n}
		 			\operatorname{tr}\bbE\Phi^2(l_i)\bbE\Phi^2(l_i)
		 			\frac1n\operatorname{tr}\Phi(l_i)
		 			\\
		 			&+
		 			\mathbb E\frac1n\operatorname{tr}\bbE\Phi^2(l_i)
		 			\left(
		 			\frac1n\operatorname{tr}\Phi^2(l_i)
		 			-
		 			\frac{1-l_i}{n}\operatorname{tr}\bbE\Phi^3(l_i)
		 			\right)
		 			\\
		 			&-
		 			\mathbb E\frac{1-l_i}{n}
		 			\operatorname{tr}\bbE\Phi(l_i)\bbE\Phi^2(l_i)
		 			\frac1n\operatorname{tr}\Phi^2(l_i)
		 			+
		 			O(n^{-1}).
		 		\end{aligned}
		 		\]
		 		The Gaussian Poincar\'e inequality again permits factorization of products of normalized traces, up to negligible errors.
		 		
		 		First,
		 		\[
		 		\begin{aligned}
		 			\mathbb E\frac1n\operatorname{tr}\bbE\Phi(l_i)\bbE\Phi^2(l_i)
		 			=&
		 			F_1(r_i)
		 			+
		 			2\frac{r_i-(\eta r_i-\omega)(1-r_i)}{r_i}F_2(r_i)
		 			\\
		 			&+
		 			\frac{\big(r_i-(\eta r_i-\omega)(1-r_i)\big)^2}{2r_i^2}
		 			QQ_2'(\gamma_i)
		 			+
		 			O(n^{-1/4}).
		 		\end{aligned}
		 		\]
		 		Second, applying the cumulant expansion formula to
		 		\[
		 		\mathbb E\frac1n\operatorname{tr}\bbE\Phi^3(l_i),
		 		\]
		 		we have
		 		\[
		 		\mathbb E\frac1n\operatorname{tr}\Phi^3(l_i)
		 		=
		 		P_3(r_i)+O(n^{-1/8}),
		 		\]
		 		where
		 		\[
		 		\begin{aligned}
		 			P_3(r_i)
		 			=&
		 			\frac{r_i^2}
		 			{c_2\big((1-c_1)r_i+c_1\big)^2}
		 			F_3(r_i)
		 			\\
		 			&+
		 			\frac{(\eta r_i-\omega)(1-r_i)}
		 			{c_2\big((1-c_1)r_i+c_1\big)}
		 			P_1(r_i)
		 			\frac{\omega(1-c_2)r_i^2}
		 			{(\eta r_i-\omega)^2(\eta r_i^2-\omega)}.
		 		\end{aligned}
		 		\]
		 		Differentiating \(n^{-1}\operatorname{tr}\Phi^3(z)\) gives
		 		\[
		 		\left(\frac1n\operatorname{tr}\Phi^3(z)\right)'
		 		=
		 		\frac{3}{nz}
		 		\left(
		 		\operatorname{tr}\bbE\Phi^4(z)
		 		-
		 		\operatorname{tr}\Phi^3(z)
		 		\right).
		 		\]
		 		Thus
		 		\[
		 		\mathbb E\frac1n\operatorname{tr}\bbE\Phi^4(l_i)
		 		=
		 		\mathbb E\frac1n\operatorname{tr}\Phi^3(l_i)
		 		+
		 		\frac{l_i}{3}
		 		\mathbb E\left(\frac1n\operatorname{tr}\Phi^3(l_i)\right)'.
		 		\]
		 		Therefore,
		 		\[
		 		\mathbb E\frac1n\operatorname{tr}\bbE\Phi^4(l_i)
		 		=
		 		P_3(r_i)
		 		+
		 		\frac{r_i-(\eta r_i-\omega)(1-r_i)}{3r_i}P_{3,1}(r_i)
		 		+
		 		O(n^{-1/16}),
		 		\]
		 		where
		 		\begin{align*}
		 			P_{3,1}(r_i)
		 			=&
		 			\frac{2c_1r_i^3}
		 			{c_2\big((1-c_1)r_i+c_1\big)^3(\eta r_i^2-\omega)}
		 			F_3(r_i)
		 			\\
		 			&+
		 			\frac{r_i^2}
		 			{c_2\big((1-c_1)r_i+c_1\big)^2}
		 			\left(
		 			\frac32P_2(r_i)
		 			+
		 			\frac{r_i-(\eta r_i-\omega)(1-r_i)}{2r_i}QQ_6'(\gamma_i)
		 			\right)
		 			\\
		 			&+
		 			\frac{r_i^2}{\eta r_i^2-\omega}
		 			\left(
		 			\frac{\eta(1-r_i)-(\eta r_i-\omega)}
		 			{c_2\big((1-c_1)r_i+c_1\big)}
		 			-
		 			\frac{(\eta r_i-\omega)(1-r_i)(1-c_1)}
		 			{c_2\big((1-c_1)r_i+c_1\big)^2}
		 			\right)
		 			\\
		 			&\qquad\times
		 			P_1(r_i)
		 			\frac{\omega(1-c_2)r_i^2}
		 			{(\eta r_i-\omega)^2(\eta r_i^2-\omega)}
		 			\\
		 			&+
		 			\frac{(\eta r_i-\omega)(1-r_i)}
		 			{c_2\big((1-c_1)r_i+c_1\big)}
		 			\\
		 			&\quad\times
		 			\left(
		 			P_2(r_i)
		 			\frac{\omega(1-c_2)r_i^2}
		 			{(\eta r_i-\omega)^2(\eta r_i^2-\omega)}
		 			+
		 			P_1(r_i)
		 			\left(
		 			2F_2(r_i)
		 			+
		 			\frac{r_i-(\eta r_i-\omega)(1-r_i)}{r_i}QQ_2'(\gamma_i)
		 			\right)
		 			\right).
		 		\end{align*}
		 		
		 		Substituting these estimates into the cumulant expansion above yields
		 		\[
		 		\mathbb E\frac1n\operatorname{tr}\bbE\Phi^2(l_i)\bbE\Phi^2(l_i)
		 		=
		 		Q_3(r_i)+O(n^{-1/16}),
		 		\]
		 		where
		 		
		 		\begin{align}\label{Q3}
		 			Q_3(r_i)
		 			=&
		 			\frac{
		 				c_2\big((1-c_1)r_i+c_1\big)^2
		 			}
		 			{r_i^2}
		 			\left(
		 			P_3(r_i)
		 			+
		 			\frac{r_i-(\eta r_i-\omega)(1-r_i)}{3r_i}P_{3,1}(r_i)
		 			\right)
		 			\\\notag
		 			&+
		 			\frac{(1-c_1)r_i+c_1}{r_i}
		 			\left(
		 			P_1(r_i)
		 			-
		 			\frac{(\eta r_i-\omega)(1-r_i)}{r_i}F_3(r_i)
		 			\right)
		 			\frac{\omega(1-c_2)r_i^2}
		 			{(\eta r_i-\omega)^2(\eta r_i^2-\omega)}
		 			\\\notag
		 			&-
		 			\frac{
		 				\big((1-c_1)r_i+c_1\big)(\eta r_i-\omega)(1-r_i)
		 			}
		 			{r_i^2}
		 			P_1(r_i)
		 			\\\notag
		 			&\qquad\times
		 			\left(
		 			F_1(r_i)
		 			+
		 			2\frac{r_i-(\eta r_i-\omega)(1-r_i)}{r_i}F_2(r_i)
		 			+
		 			\frac{\big(r_i-(\eta r_i-\omega)(1-r_i)\big)^2}{2r_i^2}QQ_2'(\gamma_i)
		 			\right).\notag
		 		\end{align}
		 		
		 		This proves \eqref{l33}.
		 		
		 		\textbf{Step 4. Proof of \eqref{l34}--\eqref{l36}.} First, we prove \eqref{l34}. Applying the cumulant expansion formula to the \(Y\)-entries gives
		 		\[
		 		\begin{aligned}
		 			&\mathbb E\frac1n\operatorname{tr}
		 			\bbY\bbW^\top\bbU_2\Phi^2(l_i)\bbE\Phi^2(l_i)\bbU_2^\top\bbW\bbY^\top
		 			\\
		 			=&
		 			\mathbb E\left(
		 			c_2-\frac1n\operatorname{tr}\bbH\Phi(l_i)
		 			\right)
		 			\frac1n\operatorname{tr}\bbE\Phi^2(l_i)\bbE\Phi^2(l_i)
		 			\\
		 			&+
		 			\mathbb E\left(
		 			c_2-\frac1n\operatorname{tr}\bbE\Phi(l_i)
		 			\right)
		 			\frac1n\operatorname{tr}\bbH\Phi^2(l_i)\bbE\Phi^2(l_i)
		 			\\
		 			&+
		 			\mathbb E\frac1n\operatorname{tr}\bbH\Phi^2(l_i)
		 			\frac1n\operatorname{tr}\bbE\Phi^2(l_i)
		 			\\
		 			&-
		 			\mathbb E\frac1n\operatorname{tr}\bbE\Phi^2(l_i)
		 			\frac1n\operatorname{tr}\bbE\Phi(l_i)\bbH\Phi^2(l_i)
		 			\\
		 			&-
		 			\mathbb E\frac1n\operatorname{tr}\bbH\Phi^2(l_i)
		 			\frac1n\operatorname{tr}\bbE\Phi(l_i)\bbE\Phi^2(l_i)
		 			+
		 			O(n^{-1}).
		 		\end{aligned}
		 		\]
		 		Using
		 		\[
		 		\bbH\Phi^2(l_i)
		 		=
		 		\frac{1-l_i}{l_i}\bbE\Phi^2(l_i)
		 		-
		 		\frac1{l_i}\Phi(l_i),
		 		\]
		 		and substituting the estimates proved above, we obtain
		 		\[
		 		\mathbb E\frac1q\operatorname{tr}
		 		\bbY\bbW^\top\bbU_2\Phi^2(l_i)\bbE\Phi^2(l_i)\bbU_2^\top\bbW\bbY^\top
		 		=
		 		\frac{Q_4(r_i)}{c_2}
		 		+
		 		O(n^{-1/16}),
		 		\]
		 		where
		 		
		 		\begin{align}\label{Q4}
		 			Q_4(r_i)
		 			=Q_3(r_i).
		 		\end{align}
		 		
		 		This proves \eqref{l34}.
		 		
		 		Next, we prove \eqref{l35}. Applying the same cumulant expansion argument to
		 		\[
		 		\mathbb E\frac1q\operatorname{tr}
		 		\bbY\bbW^\top\bbU_2\Phi^2(l_i)\bbE\Phi(l_i)\bbU_2^\top\bbW\bbY^\top
		 		\]
		 		and using the same resolvent identity, we get
		 		\[
		 		\mathbb E\frac1q\operatorname{tr}
		 		\bbY\bbW^\top\bbU_2\Phi^2(l_i)\bbE\Phi(l_i)\bbU_2^\top\bbW\bbY^\top
		 		=
		 		\frac{Q_5(r_i)}{c_2}
		 		+
		 		O(n^{-1/4}),
		 		\]
		 		where
		 		\begin{align}\label{Q5}
		 			Q_5(r_i)
		 			=F_1(r_i)+2\gamma_iF_2(r_i)
		 			+\frac{\gamma_i^2}{2}QQ_2'(\gamma_i).
		 		\end{align}
		 		\iffalse
		 		\begin{align}\label{Q5}

		 			Q_5(r_i)
		 			=&
		 			\left(
		 			c_2+
		 			\frac{c_1r_i+\omega(1-r_i)}
		 			{r_i-(\eta r_i-\omega)(1-r_i)}
		 			\right)
		 			\left(
		 			\frac{\omega(1-c_2)r_i^2}
		 			{(\eta r_i-\omega)^2(\eta r_i^2-\omega)}
		 			+
		 			QQ_2'(r_i)
		 			\right)
		 			\\\notag
		 			&-
		 			\frac{
		 				c_1\omega\eta r_i^4
		 			}
		 			{
		 				\big((1-c_1)r_i+c_1\big)^2
		 				(\eta r_i-\omega)^2
		 				(\eta r_i^2-\omega)^2
		 			}
		 			\\\notag
		 			&+
		 			\left(
		 			c_2+\frac{\omega}{\eta r_i-\omega}
		 			\right)
		 			\left(
		 			\frac{(\eta r_i-\omega)(1-r_i)}
		 			{r_i-(\eta r_i-\omega)(1-r_i)}
		 			-
		 			\frac{r_i-(\eta r_i-\omega)(1-r_i)}
		 			{r_i}
		 			\right)
		 			\\\notag
		 			&\qquad\times
		 			\left(
		 			\frac{\omega(1-c_2)r_i^2}
		 			{(\eta r_i-\omega)^2(\eta r_i^2-\omega)}
		 			+
		 			QQ_2'(r_i)
		 			\right).\notag
		 			
		 		\end{align}
		 		\fi
		 		Here we used the simplification
		 		\[
		 		\begin{aligned}
		 			&
		 			\frac{\eta\omega r_i^2}{(\eta r_i^2-\omega)(\eta r_i-\omega)^2}
		 			-
		 			\frac{\eta\omega(1-r_i)r_i}{(\eta r_i^2-\omega)(\eta r_i-\omega)}
		 			-
		 			\frac{\omega}{\eta r_i-\omega}
		 			+
		 			\frac{\omega}{\eta r_i^2-\omega}
		 			\\
		 			=&
		 			\frac{\omega\eta r_i^2}
		 			{(\eta r_i-\omega)^2(\eta r_i^2-\omega)}.
		 		\end{aligned}
		 		\]
		 		This proves \eqref{l35}.
		 		
		 		Finally, we prove \eqref{l36}. Applying the cumulant expansion formula to
		 		\[
		 		\mathbb E\frac1q\operatorname{tr}
		 		\bbY\bbW^\top\bbU_2\Phi^3(l_i)\bbU_2^\top\bbW\bbY^\top
		 		\]
		 		and using
		 		\[
		 		\bbH\Phi^3(l_i)
		 		=
		 		\frac{1-l_i}{l_i}\bbE\Phi^3(l_i)
		 		-
		 		\frac1{l_i}\Phi^2(l_i),
		 		\]
		 		we obtain
		 		\[
		 		\mathbb E\frac1q\operatorname{tr}
		 		\bbY\bbW^\top\bbU_2\Phi^3(l_i)\bbU_2^\top\bbW\bbY^\top
		 		=
		 		\frac{Q_7(r_i)}{c_2}
		 		+
		 		O(n^{-1/8}),
		 		\]
		 		where
		 		
		 		\begin{align}\label{Q7}
		 			Q_7(r_i)
		 			=F_3(r_i).
		 		\end{align}
		 		%In the last simplification, we used
		 		%\[
		 		%r_i+(\eta r_i-\omega)(1-r_i)F_1(r_i)
		 		%=
		 		%\frac{r_i^2}{(1-c_1)r_i+c_1}.
		 		%\]
		 		This proves \eqref{l36}.
		 		
		 		Combining \eqref{l31}--\eqref{l36}, Lemma~\ref{ccalemma3} follows.
		 		
		 				\section{Proofs of Proposition \ref{pro1} and Proposition \ref{pro2}}\label{proofpro}
		 				\begin{proof}[Proof of Proposition \ref{pro1}]
		 					By Theorem \ref{th:1.4}, we have
		 					$$l_i\xrightarrow{p}\gamma(r_i).$$
		 					Since $\gamma(\cdot)$ is locally one-to-one in a neighborhood of $r_i$, and $\gamma^{-1}(\cdot)$ denotes the corresponding local inverse, there exists a neighborhood $U$ of $r_i$ such that $\gamma$ is one-to-one on $U$. Let $V=\gamma(U)$ be the corresponding neighborhood of $\gamma(r_i)$ on which the local inverse $\gamma^{-1}$ is well defined. Since $\gamma$ is continuous in the present setting, the local inverse $\gamma^{-1}$ is continuous at $\gamma(r_i)$.
		 					
		 					We now show that $\hat{r}_i\to r_i$ in probability. Let $\epsilon>0$, since $\gamma^{-1}$ is continuous at $\gamma(r_i)$, there exists $\delta>0$ such that
		 					$$
		 					|x-\gamma(r_i)|<\delta
		 					\quad\Rightarrow\quad
		 					|\gamma^{-1}(x)-r_i|<\epsilon .
		 					$$
		 					Consequently,
		 					$$
		 					\{|\hat r_i-r_i|\geq \epsilon\}
		 					\subset
		 					\{|l_i-\gamma(r_i)|\geq \delta\},
		 					$$
		 					and hence
		 					$$
		 					\mathbb P(|\hat r_i-r_i|\geq \epsilon)
		 					\leq
		 					\mathbb P(|l_i-\gamma(r_i)|\geq \delta)\to 0.
		 					$$
		 					Therefore,
		 					$$
		 					\hat r_i\xrightarrow{p}r_i.
		 					$$
		 				
		 					Next, by assumption, $d(\cdot)$ is continuous at $r_i$. Hence, by the continuous mapping theorem,
		 					$$d(\hat{r}_i)\xrightarrow{p}d(r_i).$$ Since the function $x\mapsto 1/(1+x)$ is continuous at $d(r_i)$. Applying the continuous mapping theorem again yields
		 					$$\hat{\mu}_i=\dfrac{1}{1+d(\hat{r}_i)}\xrightarrow{p}\frac{1}{1+d(r_i)}=\mu_i.$$
		 					This proves
		 					$$\hat{r}_i\xrightarrow{p}r_i,~~\hat{\mu}_i\xrightarrow{p}\mu_i.$$
		 				\end{proof}
		 				
		 				\begin{proof}[Proof of Proposition \ref{pro2}]
		 					By Proposition \ref{pro1}, $\hat r_i\xrightarrow{p}r_i$. Since $d(\cdot)$ and $\sigma^2(\cdot)$ are continuous at $r_i$, the continuous mapping theorem yields
		 					$$d(\hat r_i)\xrightarrow{p}d(r_i),~~\sigma^2(\hat r_i)\xrightarrow{p}\sigma^2(r_i).$$
		 					Moreover, $1+d(r_i)\neq0$. Hence,
		 					$$\hat{\tau}_i^2=\dfrac{\sigma^2(\hat r_i)}{(1+d(\hat r_i))^4}\xrightarrow{p}\dfrac{\sigma^2(r_i)}{(1+d(r_i))^4}=\tau_i^2.$$
		 					Since $0<\sigma^2(r_i)<\infty$, we have $0<\tau_i^2<\infty$. Therefore, by the continuity of $x\mapsto \sqrt{x}$ on $(0,\infty)$,
		 					$$\hat\tau_i=\sqrt{\hat{\tau}_i^2}\xrightarrow{p}\sqrt{\tau_i^2}=\tau_i.$$
		 					This proves the proposition.
		 				\end{proof}
		 				\textbf{Acknowledgments}
		 				
		 				The authors would like to thank Zhigang Bao for many illuminating discussions in an early stage of this research.
		 				
		 				\textbf{funding}
		 				
		 				Xiaozhuo Zhang was partially supported by NSFC Grants No. 12401338, the Science Research Project of the Education Department of Jilin Province Grant No. JJKH20261178KJ.
		 				
		 				Jiang Hu was partially supported by NSFC Grants No. 12571280, Science and Technology Development Plan Project of Jilin Province, China, No. 20260101010JJ.
\iffalse		 				
\begin{acks}[Acknowledgments]
The authors would like to thank Zhigang Bao for many illuminating discussions in an early stage of this research.
\end{acks}
		\begin{funding}
			Xiaozhuo Zhang was partially supported by NSFC Grants No. 12401338, the Science Research Project of the Education Department of Jilin Province Grant No. JJKH20261178KJ.
			
			Jiang Hu was partially supported by NSFC Grants No. 12571280, Science and Technology Development Plan Project of Jilin Province, China, No. 20260101010JJ. 
		\end{funding}
\fi
		%%%%%%%%%%%%%%%%%%%%%%%%%%%%%%%%%%%%%%%%%%%%%%%%%%%%%%%%%%%%% %% The Bibliography %% %% %% %% imsart-???.bst will be used to %% %% create a.BBL file for submission. %% %% %% %% Note that the displayed Bibliography will not %% %% necessarily be rendered by Latex exactly as specified %% %% in the online Instructions for Authors. %% %% %% %% MR numbers will be added by VTeX. %% %% %% %% Use \cite{...} to cite references in text. %% %% %% %%%%%%%%%%%%%%%%%%%%%%%%%%%%%%%%%%%%%%%%%%%%%%%%%%%%%%%%%%%%% %% if your bibliography is in bibtex format, uncomment commands: 
		\newpage
		\bibliographystyle{apalike} 
		% Style BST file (imsart-number.bst or imsart-nameyear.bst) 
		\bibliography{ref_cca.bib} % Bibliography file (usually '*.bib') 
		\end{document}